\documentclass[12pt]{template}

\begin{document}

\margin{Commts.\\ON}

\date{}

\title{A Universal Hochschild-Kostant-Rosenberg theorem}

\authortasos
\authormarco
\authorbertrand

\begin{abstract}
In this work we study the failure of the HKR theorem over rings
of positive and mixed characteristic. For this we construct 
a \emph{filtered circle} interpolating between the usual topological
circle and a formal version of it. By mapping to schemes we produce
this way an interpolation, realized in practice by the existence
of a natural filtration, from Hochschild and \textcolor{black}{(a filtered version of)} cyclic homology to 
derived de Rham cohomology. In particular, we show that this recovers the filtration of Antieau \cite{1808.05246} and Bhatt–Morrow–Scholze \cite{MR3949030}. The construction of our filtered circle is based on 
the theory of affine stacks and affinization introduced by the third author, together
with some facts about schemes of Witt vectors.
\end{abstract}

\maketitle

\tableofcontents

\section{Introduction}
\mylabel{Introduction}

The purpose of the present paper is to investigate the failure
of the Hochschild-Kostant-Rosenberg theorem in positive and mixed characteristic 
situations. For this, we construct a \emph{filtered circle}
$\Filcircle$,
an object of an algebro-homotopical nature, which interpolates between 
the usual homotopy type of a topological circle and 
a degenerate version of it called the \emph{formal circle}. Given an arbitrary derived scheme $X$, we take the   mapping 
stack, in the sense of derived algebraic geometry, from $\Filcircle$ to $X$; this provides an interpolation between 
Hochschild and (\textcolor{black}{a filtered version of}) cyclic homology on one side and derived de Rham cohomology of $X$ on the other. 
The existence of such an interpolation, realized concretely in terms
of a filtration, is the main content of this work.

\subsection{Circle action and de Rham differential in characteristic zero}
\mylabel{background}
 Over any commutative ring $k$, the HKR theorem \cite{MR0142598} identifies $\Omega_{\dderham}^\ast(X)$ -  the graded commutative algebra of differential forms on a smooth $k$-scheme $X=\Spec\, A$, with $\HH_\ast(A)$ - the graded algebra of Hochschild homology. When $k$ is of characteristic zero \personal{ \cite[9.8.12]{MR1269324}} this lifts to the level of chain complexes, identifying the
 de Rham complex of differential forms $\DR(A)$ with the Hochschild complex $\HH(A)$. More is true; the de Rham complex comes equipped with a natural differential which arises, via this identification, from the natural $S^1$ action on the Hochschild complex. 
 A precise implementation of this fact requires a further enhanced version of the HKR theorem, combining the intervention of the homotopical circle action and the multiplicative structure on the Hochschild complex on one side, and the full derived de Rham algebra with its natural grading and de Rham differential on the other. This is established in \cite{MR2862069} as a consequence of an equivalence of symmetric monoidal \icategories

\begin{equation}
    \label{equivalencecircleepsiloncharzero}
    \circle-\Mod_k^\otimes \simeq \epsilon-\Mod_k^\otimes,
\end{equation}
 
 \noindent where  $\circle-\Mod_k^\otimes$ denotes equivariant $k$-modules and  $\epsilon-\Mod_k^\otimes$ is the category of mixed complexes. Here $\epsilon$  is a generator of homological degree 1 that makes $k[\epsilon]\simeq k\oplus k[1]\simeq \mathrm{H}_\ast(\circle, k)$. \textcolor{black}{On both sides we have symmetric monoidal structures: on the l.h.s using the diagonal of $\circle$ and on the r.h.s using the co-multiplication given by $\epsilon\mapsto \epsilon\otimes 1+ 1\otimes \epsilon$.}

 Via this identification, one obtains a multiplicative equivalence
 
\begin{equation}
\label{HKRstackschar0functions}
\HH(A)= A\otimes_k \circle \simeq \Sym_{\rmA}(\cotangent_{\rmA/k}[1])= \DR(A),
\end{equation}

\noindent and the agreement of the symmetric monoidal structures in \eqformula{equivalencecircleepsiloncharzero} guarantees that the circle action on the left matches the de Rham differential on the right. Geometrically  as in \cite{MR2928082}, this can also be interpreted as an identification of the derived stack of free loops on an affine $k$-scheme $X=\Spec(\rmA)$, $\loops X:=\Map(\circle, X)$,  with the shifted tangent stack $\mathsf{T} [-1]\, X:=\Map(\Spec(k[\eta], X)$ where this time $k[\eta]:= k\oplus k[-1]\simeq \mathrm{H}^\ast(\circle, k)$ is the differential graded algebra given by the cohomology of the circle. Here $\eta$ is a generator of homological degree $-1$. In this language, the HKR theorem reads as an equivalence of derived stacks

\begin{equation}
\label{HKRstackschar0}
\loops X\simeq \mathsf{T} [-1]\, X
\end{equation}

\noindent and passing to global functions, recovers the isomorphism in (\ref{HKRstackschar0functions}).

\medskip

\noindent Our first observation is that  \eqformula{equivalencecircleepsiloncharzero} is no longer  a symmetric monoidal equivalence when we abandon the hypothesis that $k$ be a field of characteristic zero; indeed, the proof of \eqformula{equivalencecircleepsiloncharzero} uses two essential facts about $\B\Ga{k}$  - the classifying stack of the group $\Ga{k}$:

\medskip

\begin{enumerate}[A)]
    \item  \noindent In any characteristic, the stack $\B\Ga{k}$ is equivalent to $\Spec^{\Delta}( \Sym^{\mathsf{co}\Delta}_k(k[-1]))$ where $\Sym^{\mathsf{co}\Delta}_k(k[-1])$ is the free cosimplicial commutative $k$-algebra over one generator
    in (cosimplicial) degree $1$ (see \cref{notationcosimplicialandsimplicial} below and \cite[Lemma 2.2.5]{MR2244263}). This can be checked at the level of the functor of points. But when $k$ is a field of characteristic zero, since the cohomology of the symmetric groups with coefficients in $k$ vanishes, we recover an equivalence of commutative differential graded algebras
    
    $$
    \Sym^{\mathsf{co}\Delta}_k(k[-1])\simeq k\oplus k[-1]:= k[\eta]
    $$
    \noindent where on the r.h.s we have the split square zero extension. In particular, we have
    $$
    \Map(\B \Ga{\,,k}, X)\simeq \Map(\Spec(k[\eta]), X)=: \mathsf{T}[-1] \, X
    $$
    
    \medskip

    \item For any ring $k$ the complex of singular cochains $\C^\ast(\circle, k)$ is given by $k\oplus k[-1]$. The canonical map of groups $\bbZ \to \Ga{k}$ produces a map of group stacks $\circle:=\B \bbZ\to \B \Ga{k}$.  As in A), because the cohomology of symmetric groups with coefficients in a field of characteristic zero vanishes, the pullback map in cohomology $\C^\ast(\B\Ga{k}, \structuresheaf)\to \C^\ast(\circle, k)$ is an equivalence. This fact exhibits the abelian group stack $\B \Ga{k}$ as the \emph{affinization} of the constant group stack $\circle$ in the sense of \cite{MR2244263} (see also \cite{lurie-DAGVIII}, \cite[Lemma 3.13]{MR2928082} and our \cref{reviewaffinestacks}). It follows from the universal property of  affinization and from Zariski descent that
    $$
    \Map(\circle, X)\simeq \Map(\B \Ga{k}, X)
    $$

    \medskip
    
\end{enumerate}
 
\medskip

\noindent The accident that allows A) and B) in characteristic zero also makes the equivalences
\begin{equation}
\label{equivalenceHopfcaraczero}
    \Affinization(\circle)\underbrace{\simeq}_{B)} \B \Ga{}\underbrace{\simeq}_{A)} \Spec(k\oplus k[-1]) 
  \end{equation}
    \noindent compatible with the group structures\footnote{Essentially, by the uniqueness of the abelian group structure on the stack $\B \Ga{k}$. See \cref{corollaryequivalenceasHopfalgebras} for a similar argument in the context of this paper.}. From here it is easy to recover the symmetric monoidal equivalence \eqformula{equivalencecircleepsiloncharzero}: $k[\epsilon]\simeq \mathrm{H}_\ast(\circle, k)$ is dual to $k[\eta]:=\mathrm{H}^\ast(\circle, k)\simeq \Sym^{\mathsf{co}\Delta}_k(k[-1])) $ and this gives us a symmetric monoidal equivalence
    \begin{equation}
\label{equivalenceHopfcaraczero22}
    \epsilon-\Mod_k^\otimes\simeq k[\eta]-\mathrm{CoMod}^{\otimes}_k
\end{equation}
  \noindent where on the r.h.s we now have the tensor product induced by convolution with the Hopf algebra structure on $k[\eta]=\mathrm{H}^\ast(\circle, k)$ induced by the group structure on the circle.  Finally the equivalences of groups in  \eqref{equivalenceHopfcaraczero} gives a symmetric monoidal equivalence
    \begin{equation}
\label{equivalenceHopfcaraczero23}
   k[\eta]-\mathrm{CoMod}^{\otimes}\simeq \circle-\Mod^{\otimes}_k
\end{equation}

\medskip

\subsection{What we do in this paper}
\mylabel{thispaper}

Away from characteristic zero, equivalence \eqformula{equivalencecircleepsiloncharzero} still holds at the level of \icategories. However, the compatibility of the two symmetric monoidal structures fails dramatically. As a consequence, we can no longer identify the circle action with the de Rham differential, the latter which fundamentally requires a notion of a mixed complex.\\ 

A key observation we make in this paper is that although these symmetric monoidal structures are not equivalent, there is a natural degeneration between the two. More precisely, we equip the symmetric monoidal category of complexes with a circle action with a filtration, whose ``associated graded" is the symmetric monoidal category of mixed graded complexes, by which we mean complexes equipped with a \emph{strict} co-action of the trivial square zero extension $k \oplus k[-1]$. See \cref{proprepmonoidal,splitfiltration2}.\\

As a glimpse to our construction, we remark that the two copies of $\B\Ga{\,,k}$ appearing in A) and B) play distinct roles. What we propose, working over $\integerslocalp$, is a construction that interpolates between the two. It is inspired by an idea of \cite{MR2244263} of using the group scheme $\Wittp$ of $p$-typical Witt vectors as a natural extension of the additive group $\Ga{}$. The group $\Wittp$ is an abutment of infinitely many of copies of $\Ga{}$ and comes canonically equipped with a Frobenius map $\Frob_p$. The abelian subgroup $\Frobfixed$ of fixed points of the Frobenius map has a natural filtration whose associated graded is the kernel of the Frobenius, $\Frobkernel$. After base change from $\integerslocalp$ to $\bbQ$ both $\Frobfixed$ and $\Frobkernel$ are isomorphic to $\Ga{}$ (see the \cref{remark-characteristiczerocase} below ) but over $\integerslocalp$ they are very different.  Without further ado, our first main theorem is the following:

\begin{theorem}[See \cref{propositionaffinizationunderlying} and \cref{theorem-splitsquarezero}]\hfill\\
\mylabel{filteredcircletheorem} 
\begin{enumerate}
    \item The abelian group stack $\B\Frobfixed$ is the affinization of $\circle$ over $\integerslocalp$.
    
    
    \medskip
    \item The abelian group stack $\B\Frobkernel$ has cohomology ring given the (cosimplicial) split square zero extension (see \cref{splitsquarezeronotation})
    $\integerslocalp\oplus\integerslocalp[-1]$ given by the cohomology of the circle.
    \noindent 
    
    
    \medskip
    \item The group stack $\B\Frobfixed$ is equipped with a \emph{filtration}, compatible with the group structure, whose associated graded stack is $\B\Frobkernel$. 
    
    \medskip
    \item After base-change along $\Spec(\bbQ)\to \Spec(\integerslocalp)$, we have
    $$\B\Frobfixed\otimes \bbQ \simeq \B\Ga{\,,\bbQ}$$
   \noindent  Moreover, the filtration splits and we have

$$(\B\Frobfixed)^{\gr}_\bbQ\simeq \B\Frobkernel\otimes \bbQ \simeq \B\Ga{\,,\bbQ}$$
    \end{enumerate}
\end{theorem}

\medskip

\begin{definition}
\mylabel{Definition-filteredcirclefirst}
The group stack $\B\Frobfixed$, equipped with the \emph{filtration} of \cref{filteredcircletheorem} - (iii), will be called the \emph{filtered circle} and denoted as $\Filcircle$.
\end{definition}

\medskip

\begin{remark}[Filtrations]
In order to define filtrations on stacks, we will follow the point of view of C. Simpson \cite[Lemma 19]{MR1159261} and \cite{MR1492538} which identifies filtered objects with objects over the stack $\Filstack$ - see \cite{tasos} and \cref{definition-filteredstacks}.  The content of \cref{filteredcircletheorem} and \cref{Definition-filteredcirclefirst} can then be reformulated as the construction of an abelian group stack $\Filcircle$ over $\Filstack$ whose fiber at $0$ has the property in A); at $1$ has the property in B); and whose pullback to $\bbQ$ is the constant family  with values $\B\Ga{\,,\bbQ}$. The construction of $\Filcircle$ is the subject of \cref{filtrationskernelfixed}, after reviewing the basics of Witt vectors in \cref{section-remindersWitt}. The proof that $\Filcircle$ satisfies (i) and (ii) will be discussed later in \cref{associatedfilteredcircle} and \cref{underlyingfilteredcircle}. The proof of (iv) is explained in the \cref{remark-characteristiczerocase}. \\
\end{remark}

\noindent Having in mind the well known interpretation of cyclic homology in terms of 
derived loop spaces, our main theorem can now be stated as follows:

\medskip
\begin{theorem}[See \cref{thmhkr-stackversion} and  \cref{thmhkr}]
\mylabel{corollaryHKRpositive}
Let $X=\Spec(\rmA)$ be a derived affine scheme  over $\integerslocalp$. Then:\\

\begin{enumerate}

    \item We have canonical equivalences of derived mapping stacks
    
    $$\Map(\circle, X)\simeq \Map(\B\Frobfixed, X)\,\,\,\,\, \text{ and }\,\,\,\,\, \mathrm{T}[-1]X\simeq \Map(\B\Frobkernel, X)$$
   In particular, the derived mapping stack $\Map(\circle, X)$ admits the structure of a filtration compatible with an action of the filtered circle $\Filcircle$ and whose associated graded is $\mathsf{T} X[-1]$;
    
    \medskip
    
    \item Passing to global functions, (i) produces a filtration on $\HH(A)$, compatible with the circle action and the multiplicative structure, and whose associated graded is the derived de Rham algebra $\DR(A)$.
    
    \medskip
    
    \item Since the filtration is compatible with the circle action, by taking homotopy fixed points of $\HH(A)$ seen as a filtered object, get a new filtered object which we call $\HCminFil(A)$. It has a canonical map to the usual fixed points 
    \begin{equation}
\label{failurebasechangeformulakey}
    \HCminFil(A)\to \HCmin(A):=\HH(A)^{\mathrm{h}\circle}
 \end{equation}
     \medskip 
     \noindent and the associated graded pieces of  $\HCminFil(A)$ are the truncated complete derived de Rham complexes
    $\mathbb{L}\widehat{\DR}^{\,\,\geq p}(A/k)$.
\end{enumerate}
\end{theorem}

\begin{remark}
\mylabel{remark-comparisonfiltrationBMS}
As we will see later (\cref{filteredglobalsectionsfixedpoints}) the map \eqref{failurebasechangeformulakey} measures the difference between the underlying object of fixed points of a filtration and the fixed points of the underlying object. When $A$ is  \textcolor{black}{discrete and quasi-smooth}, the two processes coincide - see \cite[Lemma 4.10]{1808.05246} and \cref{thmhkr}-(d).

 In \cref{section-comparison-Antieu} we show that the point (iii) extends previously known filtrations constructed in  \cite{1808.05246} for discrete rings  and in  Bhatt-Morrow-Scholze in \cite{MR3949030} for $p$-adic rings.
The construction presented in our work is very different in nature.
 
The main emphasis of our result is that this filtration exists on the algebraic circle itself, before the HKR theorem.
\medskip
\end{remark}

\medskip

In \cref{lastsection} we will discuss several
applications, generalizations and possible
research directions suggested by \cref{filteredcircletheorem} and  \cref{corollaryHKRpositive}.\\

\begin{application}[Shifted Symplectic Structures in positive characteristic]\mylabel{applicationshiftedsymplectic}
In \cref{section-shiftedsymplectic} we discuss the extension of the notion of shifted symplectic structures of \cite{MR3090262} to derived stacks in positive characteristic. For this we use our HKR theorem in order to produce certain classes in the second layer of
the filtration induced on negative cyclic homology. This is achieved by 
analyzing the Chern character map at the first two graded pieces
of the filtration \textcolor{black}{on $\HCminFil$}. We also suggest a possible definition of 
$n$-shifted symplectic structures and show that the universal 
$2$-shifted symplectic structure on $\B G$ exists essentially over any 
base ring $k$. By the techniques developed in \cite{MR3090262}
we obtain this way extensions of various previously known $n$-shifted symplectic
structures over non-zero characteristic bases. 
\end{application}

\begin{application}
[Generalized Cyclic Homology and Formal groups]\mylabel{applicationellipticcurves}
The application discussed in \cref{section-generalizedcyclicformalgroups} comes from the observation that the degeneration from $\Frobfixed$  to $\Frobkernel$ of \cref{filteredcircletheorem}-(iii) is Cartier dual to the degeneration of the multiplicative formal group $\formalGm{}$ to the additive formal group $\formalGa{}$ (see \cref{recoverviaCartierduality}). In \cref{section-generalizedcyclicformalgroups} we will discuss how to generalize \cref{filteredcircletheorem} and  \cref{corollaryHKRpositive} replacing $\formalGm{}$ by a more general formal group law $\rmE$, in particular, one associated to an elliptic curve. 
\end{application}

\medskip

\begin{application}[Topological and q-analogues]\mylabel{applicationqanalogue}
In \cref{topoqanalogue} we briefly present topological and $q$-deformed 
possible generalizations of our filtered circle. We investigate two 
related ideas, a first one that predicts the existence of 
a topological, non-commutative version of $\Filcircle$ as a filtered object in spectra. A second one, along the same
spirit, predicting the existence of a $q$-deformed filtered circle $\Filcircle(q)$ possibly related to q-deformed de Rham complex (see for instance \cite{MR1792063}) in a similar fashion that $\Filcircle$ is related to de Rham theory. Again, such a \emph{quantum
circle} can only exist if one admits non-commutative objects in some sense.
\end{application}
\medskip

\textbf{Relation with other works:}\mylabel{section-futureworks} The object $\Filcircle$ and the constructions behind it can be placed in a general context. For instance, this construction is of homotopical significance, 
as the underlying object of $\Filcircle$ is the affinization of the topological 
circle over $\integerslocalp$ in the sense of \cite{MR2244263}. In \cite{1911.05509} the third author studies the integral version of the affinization and its relation to our group
scheme $\Frobfixed$. The result there may be used in order to extend our HKR theorem over 
$\bbZ$. Concerning, the filtration, 
we believe that 
our construction is much more general and that for any 
finite CW homotopy type $X$ the affinization $(X\otimes \bbZ)=\Spec\, \C^*(X,\bbZ)$ comes 
equipped with a canonical filtration whose associated graded is $\Spec\, \mathsf{H}^*(X,\bbZ)$
(at least when $X$ has torsion free cohomology groups). This is in a way the canonical
filtration that degenerates a homotopy type over $\bbZ$ to a formal 
homotopy type.

In a different direction, we would like to mention the work \cite{Arpon}, in which the 
author also constructs a filtered
circle, as filtered Hopf algebra object in a suitable $\infty$-category of \emph{derived commutative rings}. As far as the authors understand \cite{Arpon}, 
this Hopf algebra is a model for the 
Hopf algebra of functions on our filtered circle 
$\Filcircle$.



\medskip

\begin{acknowledgements}
\mylabel{acknowledgements}The idea of a filtered circle has been 
inspired by the construction 
of the Hodge filtration in Carlos Simpson's non-abelian 
Hodge theory, and we thank Carlos Simpson, Tony Pantev, Gabriele Vezzosi and Mauro Porta for several conversations on these themes along the years. We also thank  Arpon Raksit for discussions related to his approach to the filtered circle.  We are very thankful to Pavel Etingof for
having pointed to us the possible existence of the q-deformed
circle (presented in \cref{topoqanalogue}), this has also lead the authors to realize the existence of various circle associated to any commutative 
formal group law.

We are also thankful to the anonymous referee for encouraging us to include a new section comparing our filtration with the one of Antieau.

This paper was written while the first two authors were in residence at the Mathematical Sciences Research Institute in Berkeley, California, during the Spring 2019 semester, supported by the National Science Foundation under Grant No. DMS-1440140. Tasos Moulinos and Bertrand To\"en are 
supported by the grant NEDAG ERC-2016-ADG-741501. Marco Robalo was supported by the grant ANR-17-CE40-0014.
\end{acknowledgements}

\medskip

\begin{notation}[Simplicial, Cosimplicial and $\Einfinity$]
\mylabel{notationcosimplicialandsimplicial}
Unless mentioned otherwise, all higher categorical notations are borrowed from \cite{lurie-ha, lurie-htt}. Let $k$ be a discrete commutative ring. Throughout the paper we will denote by 
\medskip
\begin{enumerate}[a)]

    \item We use homological conventions. We write $\Mod_k$ for the \icategory of chain complexes of $k$-modules; $\Mod_{k}^{ \geq 0}$, resp. $\Mod_{k}^{\leq 0}$ the categories of connective and coconnective complexes.
      \medskip
     \item The notation $\CAlg$ will always be used to denote  $\Einfinity$-algebras. In particular $\CAlg_k$ will denote $\Einfinity$-algebras in $\Mod_k$; $\CAlg_k^{\mathsf{cn}}\simeq \CAlg(\Mod_k^{\geq 0})$ the full subcategory of connective algebras \cite[2.2.1.3, 2.2.1.8, 7.1.3.10]{lurie-ha}
     and $\CAlg_k^{\mathsf{ccn}}$ the category of $\Einfinity$-algebras in $\Mod_k^{\leq 0}$ for the symmetric monoidal structure induced from the fact $\tau_{\leq 0}$ is a monoidal localization. In particular, as the inclusion $\Mod_k^{\leq 0}\subseteq\Mod_k$ is lax monoidal, we have an induced map at the level of algebras $\CAlg_{k}^{\mathsf{ccn}}\to \CAlg_k$
  
    \medskip
    \item $\SCRings{k}$ the \icategory of simplicial commutative rings over $k$. This is the sifted completion of the discrete category of polynomial algebras $\Nerve(\Poly{k})$. See \cite[25.1.1.5]{lurie-sag} and \cite[5.5.9.3]{lurie-htt}. The universal property of sifted completion gives us the normalized Dold-Kan functor $\theta:\SCRings{k}\to \CAlg_k^{\mathsf{cn}}$. By  \cite[25.1.2.2, 25.1.2.4]{lurie-sag} this is both monadic and comonadic and if $k$ is of characteristic zero it is an equivalence. Given $A\in \SCRings{k}$, $\theta(A)$ will be called the \emph{underlying $\Einfinity$-algebra of $A$}. 
    \medskip

    \item By $\Sym$ we will always mean the simplicial version $\Sym^\Delta$ as a monad in $\Mod_{k}^{\geq 0}$;
     \medskip

    \item $\coSCRings{k}$ the \icategory of cosimplicial commutative rings over $k$ (See Definition \ref{cosimplicialcommalgebra} below.)  We also denote by $\theta:\coSCRings{k}\to  \CAlg_k$ the  conormalized Dold-Kan construction (see \cite[\S 2.1]{MR2244263}).  This functor is conservative, commutes with tensor products and preserves limits by the discussion in Section \ref{hopfalgebrasandcomodules}.  It can be factored by a functor $\theta^{\mathsf{ccn}}$
    
    $$
    \xymatrix{
    \coSCRings{k}\ar[rr]^-{\theta^{\mathsf{ccn}}}&& \CAlg_k^{\mathsf{ccn}}\subseteq \CAlg_k
    }
    $$
    
    \noindent where $\theta^{\mathsf{ccn}}$ is the co-dual Dold-Kan construction of \cite{MR1243609}. 

    \medskip
    
    
  \item By $\Sym^{\mathsf{co}\Delta}$ we will mean the free cosimplicial commutative algebra on $\Mod_k^{\mathsf{\leq 0}}$.
  \medskip
    
    \item $\Stacks{k}$ the \icategory of stacks over the site of discrete commutative $k$-algebras and $\dSt_k$ the \icategory of derived stacks, ie, stacks over 
    $\SCRings{k}$.
    \medskip
    \item $\dSch_k$ the $\infty$-category of derived schemes and $\Spec: \SCRings{k}^{\op}\to \dSch_k$ the affine derived scheme associated to of a simplicial commutative ring.
     \item $\cospec: \coSCRings{k}^{\op}\to \Stacks{k}$ the \ifunctor   sending an object $\rmA\in \coSCRings{k}$ to the (higher) stack which sends a classical commutative ring $\rmB$ to the mapping space $\Map_{\coSCRings{k}}(\rmA,\rmB)$. See also \cref{reviewaffinestacks}
    \medskip
    
\end{enumerate}
\end{notation}

\medskip

\begin{notation}
\mylabel{geometrynotations}
We assume $k$ as in \cref{notationcosimplicialandsimplicial}. $\Gm{k}$ and $\affineline{k}$ will always denote the flat versions of the multiplicative group and affine line over $k$.
\end{notation}

\medskip
\begin{notation}[Quasi-coherent sheaves on stacks and $t$-structures]
\mylabel{left-complete-t-structure}
Again, assume $k$ as in \cref{notationcosimplicialandsimplicial}.
\begin{enumerate}[a)]
\item $\Qcoh$ will always denote the \icategory of quasi-coherent  sheaves in the sense of \cite[6.2.2.1, 6.2.2.7, 6.2.3.4]{lurie-sag}, ie, $\Qcoh(X):=\lim_{\Spec A\to X}\Mod_A$ with $\Spec(A)$ a derived affine scheme with $A\in \SCRings{k}$ and $\Mod_A$ the category of modules in spectra of the connective $\Einfinity$-ring $\theta(A)$.
    
    \medskip
    
    \item Under some mild conditions, $\Qcoh$ admits a t-structure where, by definition, $F\in \Qcoh(X)$ is connective if  the pullback to every affine  is connective. See \cite[6.2.5.7, 6.2.5.8, 6.2.5.9, 6.2.3.4-(3)]{lurie-sag}. We will see this in \cref{left-complete-t-structure}. In particular, if for any map $f:X\to Y$, the pullback $f^\ast:$ $\Qcoh(Y)\to \Qcoh(X)$ is right t-exact (ie, preserves $\geq 0$) and the right adjoint $f_\ast$ is left t-exact (preserves $\leq 0$) \cite[1.3.3.1]{lurie-ha}.
    \medskip
    
    \item In the light of the previous item and particularly useful in this paper, is the condition that a stack $X$ over $k$ is presented by a simplicial scheme $\Spec(\rmA^\bullet)$, with $\rmA$ a cosimplicial commutative algebra $\rmA^\bullet$ with $\rmA^0=k$, each $\rmA^m$ flat and discrete over $\rmA^0$ and such that all boundary maps $\rmA^m\to \rmA^n$ are flat. Then we can proceed as in \cite[4.5.2]{lurie-DAGVIII} and \cite[6.2.5.7, 6.2.5.8, 6.2.5.9, 6.2.3.4-(3)]{lurie-sag} and prescribe a left-complete t-structure on $\Qcoh(X)$ compatible with the limit decomposition by descent

$$
\Qcoh(X)\simeq \,\, \lim_{[n]\in \Nerve(\Delta^\op)}\,\Mod_{\rmA^n}
$$

\noindent Since the transition maps $\rmA^m\to \rmA^n$ are flat, the extension of scalars $\Mod_{\rmA^m}\to \Mod_{\rmA^n}$ are both left and right $t$-exact and we can define

$$
\Qcoh(X)_{\geq 0}:=  \lim_{[n]\in \simplicial} \, (\Mod_{\rmA^{n}}^{\geq 0})
$$

 $$\Qcoh(X)_{\leq n}\simeq \lim_{[m]\in \simplicial}\, \Mod_{\rmA^m}^{\leq n}$$

\noindent In particular, an object $\rmM\in \Qcoh(X)$ connective or coconnective if it is so after pullback along the atlas and is in the heart if and only if its image along the first projection $\Qcoh(X)\to \Mod_k$ is in $\Mod_{k}^{\heartsuit}$ - the classical abelian category of $k$-modules. We conclude that $\structuresheaf_X$ is in the heart.
    
    \medskip
    
    \item For a (derived) stack $X$ over a ring $k$, we will denote by $\cochains(X,\structuresheaf)$ the $\Einfinity$-ring in $\Mod_k$ given by the quasi-coherent pushforward of $\structuresheaf_X$ along the structure map $X\to \Spec\, k$.
\end{enumerate}
\end{notation}
\medskip

\section{Filtrations, Fixed Points and Kernel of Frobenius on Witt vectors}
\mylabel{filtrationskernelfixed}

\subsection{Review of Witt-Vectors}
\mylabel{subsection-Witt}

We start by reviewing the materials concerning Witt vectors that will be used in the paper. We follow \cite[\S 1]{MR3316757} closely. Other references are \cite{lurie-ambidexterity,nikolaus-notes,
MR0209285,MR2553661,MR2987372, MR565469, MR0302656}.
\begin{guide}[Witt Vectors]
\mylabel{guideWitt}
\label{ptypicalWittvectors}
\label{Wittvectors}
\label{section-remindersWitt}

\hfill\\
\begin{enumerate}[1)]
\item Throughout this paper we fix $p$ a prime and unless mentioned otherwise, we work with $\integerslocalp$-algebras;
\medskip
    \item $\Wittp: \CRings_{\integerslocalp}\to \Ab$ the commutative group scheme of $p$-\emph{typical Witt vectors}, evaluated on $\integerslocalp$-algebras.  The underlying scheme of $\Wittp$ is the infinite product $\prod_{n\in S}\affineline{}$ where $S=\{1, p, p^2,...\}$.
    \medskip
    \item $\mathrm{Ghost}: \Wittp \to \prod_{n\in S}\Ga{}$, the map implemented by the Ghost coordinates, defined via $(\lambda_n)_{n\in S}\mapsto (\omega_n)_{n\in S }$ where $\omega_n:= \sum_{d|n}d.\lambda_d^{\frac{n}{d}}$. In particular $\omega_1=\lambda_1$. The group structure on $\Wittp$ of (i) is uniquely determined by the requirement that $\mathrm{Ghost}$ defines a map of groups, where on the r.h.s we have the degrewise additive group structure of $\Ga{}$. See \cite[Prop. 1.2]{MR3316757}. The group unit is the Witt vector $(1,0,..)\in \Wittp$.
    \medskip
    \item The map $\mathrm{Ghost}$ is an isomorphism after base change to $\bbQ$. See \cite[B.3(2)]{nikolaus-notes}.
    \medskip
    \item $\Wittpm$ the group scheme of truncated $p$-typical Witt vectors of length $m$. Let $S_m:=\{1, p, ..., p^{m-1}\}$. The underlying scheme of $\Wittpm$ is $\prod_{n\in S_m}\affineline{}$. As a group, $\Wittpm$ is again defined under the requirement that the truncated Ghost coordinates $\Wittpm\to \prod_{n\in S_m}\Ga{}$ define a map of groups. By the same argument as in \cite[B.3(2)]{nikolaus-notes} \begin{equation}
    \label{ghostcomponentstruncatedptypical}
\Wittpm\otimes \bbQ\simeq \prod_{\{1, p, ..., p^{m-1}\}} \Ga{\bbQ}
\end{equation}
\medskip
    \item There are restriction maps $\Wittpmplus\to \Wittpm$  The $\Wittp^{(1)}$ is canonically isomorphic to $\Ga{}$ and there are exact sequences
    \begin{equation}
    0\to \Ga{}\to \Wittpmplus\to \Wittpm\to 0    
    \end{equation}
    Passing to the limit we get a pro-structure $\Wittp\simeq \lim_{m} \Wittpm$ as group schemes.
    \medskip
    \item The group scheme $\Wittp$ comes naturally equipped with a Frobenius endomorphism $\Frob_p:\Wittp\to \Wittp$. This is uniquely defined as a map of abelian groups under the requirement that on Ghost coordinates it acts by
    $$
    (\omega_1, \omega_p, \omega_{p^2},...)\mapsto (\omega_p, \omega_{p^2}, ...)
    $$
    \noindent See \cite[Lemma 1.4]{MR3316757}. In terms of the pro-structure it decomposes as maps
    $$
    \Frob_p: \Wittpm\to \Wittpmone
    $$
    After base change to $\finitefieldp$, $\Frob_p$ is the standard Frobenius on each coordinate, namely, $(\lambda_n)_{n\in S}\mapsto (\lambda_n^p)_{n\in S}$. See \cite[Lemma 1.8]{MR3316757}. 
    \medskip
    \item We write $\Frobfixed$ for the group scheme given by the kernel of the map 
$\Frob_p-\Id:\Wittp\to\Wittp$ and $\Frobkernel$ for the kernel of $\Frob_p:\Wittp\to\Wittp$.
    \end{enumerate}
\end{guide}

\medskip

To illustrate how Witt vectors will be used in our HKR theorem, let us start with the following observation of what happens in  characteristic zero:

\medskip

\begin{remark}
\mylabel{remark-characteristiczerocase}
Let $\rmA$ be a $\bbQ$-algebra. Then the explicit formula for $\Frob_p$ on $\Wittp(\rmA)$ in terms of the Ghost coordinates tells us that the fixed points for the Frobenius are given by
the diagonal embedding
$$
\Frobfixed(\rmA) \simeq \Delta \subseteq \prod_{i\geq 0} \Ga{}(\rmA)
$$

\noindent Another easy computation in Ghost coordinates, also tells us that the Kernel of the Frobenius is given by
the inclusion of the first coordinate:

$$
\Frobkernel(\rmA) \simeq  (\Ga{}(\rmA),0,0,0,0,...) \subseteq \prod_{i\geq 0} \Ga{}(\rmA)
$$

\noindent In other words, as group-schemes we obtain 

$$\Frobfixed_{|_{\bbQ}}\simeq \Ga{\bbQ}\,\,\,\,\,\, \text{ and }\,\,\,\,\Frobkernel_{|_{\bbQ}}\simeq \Ga{\bbQ}$$

\end{remark}

\medskip

The \cref{remark-characteristiczerocase} shows that for $\bbQ$-algebras, the additive group scheme $\Ga{}$ can be defined abstractly via Witt vectors, either as Frobenius fixed points or as the kernel. In this paper we utilize this feature to understand the HKR theorem in positive characteristic. Away from $\bbQ$-algebras, the fixed points and the kernel of $\Frob_p$ on $\Wittp$ do not agree. However, we shall see that there is a natural degeneration from the first to the second, or more precisely, a filtration on $\Frobfixed$ whose associated graded is $\Frobkernel$. The delooping
of this filtration will be, by definition, our filtered circle.
For this purpose we will need to explain what is a filtration on a stack. Before addressing that question, let us be precise about the linear versions of filtrations and gradings used in this paper:

\subsection{Graded and filtered objects}
\mylabel{subsection-graded-filteredobjects}

We start this section by recalling the definitions of filtered and graded objects and how t-structures give rise to filtered objects. We conclude by discussing Simpson point of view on filtrations.

\begin{construction}\cite{lurie-K}
\mylabel{construction-Filteredobjects}
Let $\C$ be a cocomplete stable \icategory. The category of filtered objects in $\C$ is the \icategory of diagrams $\Fil(\C):=\Fun(\Nerve(\bbZ)^\op, \C)$, with $\Nerve(\bbZ)$ the nerve of the category associated to the poset $(\bbZ,\leq)$. The category of $\bbZ$-graded objects in $\C$ is the \icategory of diagrams 
$\C^{\bbZ-\mathrm{gr}}:=\Fun(\bbZ^\mathrm{disc,\op}, \C)\simeq \prod_{i\in \bbZ} \C$ where $\bbZ^\mathrm{disc}$ is the $\bbZ$ seen as a discrete category. If $\C$ has a symmetric monoidal structure with unit $\unit$, both categories are endowed with symmetric monoidal structures given by Day convolution. The unit of $\C^{\bbZ-\mathrm{gr}}$ is given by the graded object $(\cdots, \underbrace{0}_{1}, \underbrace{\unit}_{0}, \underbrace{0}_{-1}, \underbrace{0}_{-2},\cdots)$. The unit of $\Fil(\C)$ is the filtered object $\cdots \underbrace{0}_{1}\to \underbrace{\unit}_{0} =\underbrace{\unit}_{-1}=\underbrace{\unit}_{-2}=\cdots$. Following \cite[\S 3.1 and 3.2]{lurie-K}, the construction of the associated graded object, respectively, the underlying object, are implemented by symmetric monoidal functors

$$
\mathrm{gr}:\Fil(\C)\to \C^{\bbZ-\mathrm{gr}}\,\,\, \text{ and }\,\,\, \colim:\Fil(\C)\to \C.
$$

\end{construction}

\medskip

\begin{remark}
\mylabel{mapoffiltratinsisoiffassgradandunderiso}
A map in $\Fil(\C)$ is an equivalence if and only if both associated graded and underlying maps are equivalences. Indeed, since $\Fil(\C)$ is stable, passing to cofibers, this amounts to show that a filtered object $\rmE$ is zero if and only both ita associated graded and underlying objects are zero. But if the associated graded is zero, we see that all maps $\rmE_{n+1}\to \rmE_n$ are equivalences and therefore the filtered object is constant. If moreover its colimit is zero, then the object itself is zero.
\end{remark}

\begin{definition}
\mylabel{definition-completefiltration}
Assume furthermore that $\C$ has all limits. Let $\rmF\in \Fil(\C)$ be a filtered object. We say that $\rmF$ is \emph{complete} if $\lim\, \rmF\simeq 0$. We denote by $\Filcomplete(\C)$ the full subcategory of $\Fil(\C)$ spanned by complete filtered objects.
\end{definition}

\medskip
\begin{remark}
\mylabel{remark-completemodulesleftorthogonal}
For any filtered object $\rmF$ and for every $i\in \bbZ$ we have a cofiber sequence  in $\C$

\begin{equation}
\label{eq-exactsequencecompletionfiltered}
 \lim\, \rmF\underbrace{\simeq}_{cofinal}\lim_{j>i}\,\rmF_j \to \rmF_i\to \rmF_i/(\lim_{j>i}\, \rmF_j)\simeq \lim_{j>i} \,\rmF_i/\rmF_j
\end{equation}

\noindent where the last equivalence follows because $\Fil(\C)$ is stable. Therefore, a filtered object is complete if and only if the canonical map $
\rmF_i\to \lim_{j\geq i}\, \rmF_i/\rmF_j
$ is an equivalence for all $i$. Inspire by this, we define a new filtered object $\widehat{\rmF}$ by the formula $(\widehat{\rmF})_n:= \rmF_n/ (\lim\, \rmF)$ and by construction, it is complete \personal{(Use the exact sequence \eqformula{eq-exactsequencecompletionfiltered} with the constant filtered object $\lim\, \rmF)$.)} and comes with a natural map of filtered objects $\rmF\to \widehat{\rmF}$ that is an equivalence precisely when $\rmF$ is complete. Following \cite[Proposition 2.14, Lemma 2.15]{MR3806745}, the assignment $\rmF\mapsto \widehat{\rmF}$ provides a left adjoint $\widehat{(-)}:\Fil(\C)\to \Filcomplete(\C)$ to the inclusion of the full subcategory $\Filcomplete(\C)\subseteq \Fil(\C)$ and presents $\Filcomplete(\C)$ as a Bousfield localization of $\Fil(\C)$ with respect to the class of maps of filtered objects  $\rmE\to \rmF$ whose associated graded $\gr(\rmE)\to \gr(\rmF)$ is an equivalence of graded modules. This follows from the two fiber sequences
$$
\lim\, \rmF/\lim\, \rmE \to \rmF_i/\rmE_i\to \widehat{\rmF}_i/\widehat{\rmE}_i\,\,\,\,\text{ and  }\,\,\,\rmF_{i+1}/\rmE_{i+1} \to \rmF_i/\rmE_i\to \gr(\rmF/\rmE)_i$$

\noindent In particular, the inclusion of $\Filcomplete(\C)\subseteq \Fil(\C)$ is stable under all limits.
\end{remark}

\medskip

\begin{construction}
\mylabel{construction-whiteheadtowerfunctor}
\noindent Let $\C$ be stable $\infty$-category with a $t$-structure $(\C_{\geq 0}, \C_{\leq 0})$. Then we have an associated Whitehead tower $\infty$-functor
$$
\tau_{\geq}: \C\to \Fil(\C)\,\,\,, \text{ sending } X\mapsto \tau_\geq X:= [\cdots \to \tau_{\geq n+1}X \to\tau_{\geq n}X\to \cdots]
$$
\noindent  If the $t$-structure is right complete, then $\tau_{\geq }$ admits a left inverse by extracting the underlying object $\colim:\Fil(\C)\to \C$. The essential image of $\tau_{\geq}$ consists of those filtered objects $\rmE$ such that $\pi_i(\rmE_n)=0$ if $i< n$ and such that the canonical maps $\pi_i(\rmE_n)\to \pi_i(\colim \, \rmE)$ are  equivalences for every $i\geq n$. In particular, if $\rmF\in \C$, the graded pieces are given by $\gr^n(\tau_{\geq } \rmF)\simeq \pi_n(\rmF)[n]\in \C^{\heartsuit}[n]$.
\end{construction}

\medskip

\begin{construction}
\mylabel{remark-conditionforcomplete} Let $\C$ be stable presentable $\infty$-category with a left-complete presentable $t$-structure (in the sense of \cite[1.2.1.17, 1.2.1.19]{lurie-ha}) and such that $\C_{\geq 0}$ is stable under countable products (such as $\Spectra$ or $\Mod_k$). Denote by $\Fildescend(\C)$ the full subcategory of $\Fil(\C)$ spanned by those filtered objects $\rmE$ such that for every $i$, $\rmE_i\in \C_{\geq i}$. Then we have

$$
\Fildescend(\C)\subseteq \Filcomplete(\C)
$$

\noindent Indeed, we argue that the limit $(\lim \, \rmE)\in \C_{\geq i}$ for all $i\in \bbZ$, so that $(\lim \, \rmE) \in \bigcap_{i} \C_{\geq i}=\{0\}$ because of left-completeness. To show that $(\lim \, \rmE)\in \C_{\geq i}$ we see that by cofinality we have $$\lim \, \rmE\simeq \lim_{n\geq i} \, (\cdots \to \rmE_{i+3}\to \rmE_{i+2}\to \rmE_{i+1})$$
This sequential limit can be computed using the cotensorization of $\C$ over the $\infty$-category of spaces via the pullback diagrams in $\C$

$$
\xymatrix{
\prod_{n\geq i+1} \rmE_n[-1]\ar[d]\ar[r]&\lim_{n\geq i} \, (\cdots \to \rmE_{i+3}\to \ar[d] \rmE_{i+2}\to \rmE_{i+1})\ar[rr]&&\prod_{n\geq i+1} \rmE_n^{\Deltaone}\ar[d]\\
0\ar[r]&\prod_{n\geq i+1} \rmE_n\ar[rr]&&\prod_{n\geq i+1} \rmE_n\times \rmE_n
}
$$

\noindent and we obtain

$$
\lim \, \rmE\, \simeq \cofiber \, (\prod_{n\geq i+1} \rmE_n[-1]\to \prod_{n\geq i+1} \rmE_n[-1])
$$

\noindent Since, by assumption $\C_{\geq i}$ is stable under countable products, we have $\prod_{n\geq i+1} \rmE_n[-1]\in \C_{\geq i}$. Since $\C_{\geq i}$ is stable under all colimits \cite[1.2.1.6]{lurie-ha} we conclude that $\lim \, \rmE\in \C_{\geq i}$.

\end{construction}

\medskip

\begin{remark}
\mylabel{remarkwhiteheadtoweriscomplete}
\noindent The discussion in \cref{remark-conditionforcomplete} also shows that whenever the $t$-structure on $\C$ is left complete, the Whitehead filtration  $\tau_\geq X$ of \cref{construction-whiteheadtowerfunctor} is complete so that $\tau_{\geq}$ factors as 
$$\tau_{\geq}: \C\to \Fildescend\subseteq \Filcomplete(\C)\subseteq \Fil(\C)$$.
\end{remark}

\medskip

We will also need the notion of graded and filtered categories: 
\medskip

\begin{definition}
\mylabel{FilteredCategory}
We define a \emph{graded category} to be a stable presentable \icategory endowed with a structure of object in $\Mod_{\Spectra^{\bbZ-\mathrm{gr},\otimes}}(\Prl)$. Similarly,  a \emph{filtered category} is an object in $\Mod_{\Fil(\C)^\otimes}(\Prl)$.
\end{definition}
\medskip

We will need to use the point of view on filtrations and gradings given by the Rees construction of Simpson \cite[Lemma 19]{MR1159261} where the following geometric objects play a central role:

\begin{construction}
\mylabel{construction-geometricstackfiltration}
Let $\B\Gm{\Sphere}$ be the classifying stack of the flat multiplicative  abelian group scheme over the sphere spectrum  and $e:\Spec(\Sphere)\to \B\Gm{\Sphere}$ the canonical atlas. Consider also the canonical geometric action of $\Gm{\Sphere}$ on the flat affine line $\affineline{\Sphere}$ of weight 1, ie, $(\lambda, a)\mapsto \lambda^1. a$. Algebraically, this is the action for each the variable $t$ is of weight $(-1)$. Form the quotient stack $[\affineline{\Sphere}/\Gm{\Sphere}]$. This lives canonically as a stack over $\B\Gm{\Sphere}$ via a map 
$$
\pi:[\affineline{\Sphere}/\Gm{\Sphere}]\to \B\Gm{\Sphere}
$$

The inclusion of the zero point $0:\Spec(\Sphere) \to \affineline{\Sphere}$ provides a section of the canonical projection $[\affineline{\Sphere}/\Gm{\Sphere}]\to \B\Gm{\Sphere}$
$$
0:\B\Gm{\Sphere}\to [\affineline{\Sphere}/\Gm{\Sphere}]
$$
The inclusion of $\Gm{\Sphere}$ in $\affineline{\Sphere}$ also passes to the quotient and provides a map
$$
 [\affineline{\Sphere}/\Gm{\Sphere}] \leftarrow \Gm{\Sphere}/\Gm{\Sphere}\simeq\ast:  1
$$

\end{construction}

\medskip

The Rees construction gives us a geometric interpretation of filtered objects and gradings when $\C=\Spectra$  is the \icategory of spectra, in terms of objects over $[\affineline{\Sphere}/\Gm{\Sphere}] $ and $\B \Gm{\Sphere}$.  The proof of the following result appears in \cite{tasos}: 

 \begin{theorem}
 \mylabel{prop:gradingsandstacks}
 There exists symmetric monoidal equivalences
 
 \begin{equation}
 \label{formula-BGmgradedspectra}
  \Spectra^{\bbZ-\mathrm{gr},\otimes}\simeq \Qcoh(\B \Gm{\Sphere})^\otimes
 \end{equation}
 
 \begin{equation}
 \label{formula-filteredstackfilteredspectra}
\mathrm{Rees}:\Fil(\Spectra)^\otimes\simeq \Qcoh([\affineline{\Sphere}/\Gm{\Sphere}])^\otimes
 \end{equation}
 
 \noindent such that the following diagram commutes:
 
 $$
 \xymatrix{
 \Qcoh(\B \Gm{\Sphere})^\otimes \ar[d]_{\sim }& \ar[l]_{0^\ast}  \ar[d]_{\sim }\Qcoh([\affineline{\Sphere}/\Gm{\Sphere}])^\otimes\ar[r]^-{1^\ast}& \Qcoh(\Spec(\Sphere))^\otimes=\Spectra^\otimes \ar@{=}[d]\\
\Spectra^{\bbZ-\mathrm{gr},\otimes} & \ar[l]_{\gr} \Fil(\Spectra)^\otimes\ar[r]^{\colim}& \Spectra^\otimes
 }
 $$
 Moreover, after base change along $\Spec(\bbZ)\to \Spec(\Sphere)$ we recover analogues of these comparisons for filtered and graded objects in $\Mod_\bbZ$ the \icategory derived category of abelian groups and quasi-coherent sheaves on $\B \Gm{\bbZ}$ and $[\affineline{\bbZ}/\Gm{\bbZ}]$
 via the Rees construction (see for instance \cite{MR1492538}).
\end{theorem}

\medskip
In view of the \cref{prop:gradingsandstacks} the following definition becomes natural.

\begin{definition}
\mylabel{definition-filteredstacks}
We define a graded stack to be a stack over $\B\Gm{}$ and a filtered stack to be a stack over $\Filstack $. Let $X\to \Filstack $ be a filtered stack. The  \emph{associated graded} of $X$, denoted $X^{\gr}$, is the base change of $X$ along the map $0: \B\Gm{}\to \Filstack $. By abuse of notation we will also write $X^{\gr}$ to denote the further pullback along the atlas $\ast\to \B \Gm{}$, endowed with its canonical $\Gm{}$-action. 

The \emph{ underlying stack}, $X^u$, is the base-change along $1:\ast\to \Filstack $.
\end{definition}

\begin{definition}
\mylabel{filteredglobalsections}
Let $\pi: X \to \Filstack$ be a (derived) stack over $\Filstack$. We use the notation 
$$\Ofil(X) := \pi_{*} \calO_{X}$$
for the push-forward  of the structure sheaf; this is an object of the $\infty$-category $\Qcoh(\Filstack)$ of filtered $k$-modules. 
\end{definition}

\medskip

\begin{remark}
\mylabel{filteredcategoriesQcohfilteredstack}
Let $X$ be a filtered (resp. graded) stack. Then $\Qcoh(X)$ is a filtered (resp. graded) category in the sense of \cref{FilteredCategory} via the symmetric monoidal pullback along the structure map to $\Filstack$ (resp. $\B \Gm{}$).  
\end{remark}

\begin{remark}
\mylabel{globalsectionsoffilteredarefiltered}
Let $\pi:X\to \Filstack$ be a filtered stack and consider the pullback diagram

\begin{equation}
\label{primitivebasechangediagram}
\xymatrix{
X^{\gr} \ar[d]_{\pi^{\gr}}\ar[r]^{\tilde{0}} & X \ar[d]^{\pi} & X^u \ar[d]^{\pi^u}\ar[l]_{\tilde{1}}\\
\B \Gm{}\ar[r]^{0}& \Filstack & \ar[l]_{1} \ast
}
\end{equation}

In this paper we will need to deal with different situations under which the base change  property for quasi-coherent sheaves of $\structuresheaf_X$-modules along the two pullback diagrams, holds. Namely, the Beck-Chevalley transformations 

$$
0^\ast \pi_\ast \to (\pi^{\gr})_\ast\tilde{0}^\ast\,\,\,\,\,\,\,\,,\,\,\,\,1^\ast \pi_\ast \to (\pi^{u})_\ast \tilde{1}^\ast\,\,\,\,\,\,\,\, \text{ on } \Qcoh
$$

\noindent are equivalences.

\begin{enumerate}
    \item  Following \cite[A.1.3 (1)]{1402.3204} (see \cref{left-complete-t-structure}-c) and  \cite[6.2.5.9]{lurie-sag}) this holds for $1$ whenever $X$ admits a flat hypercover by affines with flat transition maps and $F\in \Qcoh(X)_{<\infty}$ (homologically bounded above). Indeed, in this case the argument in the proof of \cite[A.1.3 (1)]{1402.3204} applies since 1 is an open immersion, ergo flat, ergo of finite tor-dimension.
    \item Also following \cite[A.1.3 (1)]{1402.3204}, the Beck-Chevalley transformation is an isomorphism for $0$ whenever $X$ admits a flat atlas by affines with flat transition maps (see \cref{left-complete-t-structure}-c) and \cite[6.2.5.9]{lurie-sag}) and $F\in \Qcoh(X)$. This is possible since $0$ is an lci closed immersion and therefore of finite tor-dimension and finite.

    \item Whenever the map $\pi$ is of finite cohomological dimension. See different analysis in \cite[A.1.4, A.1.5, A.1.9]{1402.3204}, \cite[9.1.5.7, 9.1.5.8, 9.1.5.3]{lurie-sag}, \cite[6.3.4.1]{lurie-sag} or \cite[Prop. 3.10]{MR2669705}.

\end{enumerate}

Under these assumptions, the global sections $\pi_\ast(\structuresheaf_X)=\Ofil(X)$ admit the structure of an $\Einfinity$-algebra in $\Fil(\Spectra)$ which can be interpreted as a filtration on the $\Einfinity$-algebra $(\pi^u)_\ast \structuresheaf=\cochains(X^u, \structuresheaf)$, whose associated graded is $(\pi^{\gr})_\ast\structuresheaf=\cochains(X^{\gr}, \structuresheaf)$.
By the same mechanics, given a graded stack, $Y\to \B \Gm{}$, the $\Einfinity$-algebra $\cochains(Y, \structuresheaf)$ carries a grading compatible with the $\Einfinity$-structure.
\end{remark}

\medskip

\begin{construction}
\mylabel{construction-filteredassociatedstack}
Let $X$ be a stack with a $\Gm{}$-action and consider the stacky quotient $[X/\Gm{}]$ which lives canonically over $\B \Gm{}$. We defined the filtered stack associated to $X$ with the $\Gm{}$-action to be the fiber product of stacks

$$
\xymatrix{
X^{\Fil}:=X/\Gm{}\times_{\B\Gm{}} \Filstack\ar[d] \ar[r]& [X/\Gm{}]\ar[d]\\
\Filstack \ar[r]& \B\Gm{} 
}
$$

We have cartesian squares

$$
\xymatrix{
(X^{\Fil})^{\gr}=[X/\Gm{}]\ar[r] \ar[d]&X^{\Fil} \ar[d] & \ar[l](X^{\Fil})^{u} \ar[d]\\
\B \Gm{} \ar[r]^{0} & \Filstack & \ar[l]_{1} \ast
}
$$

Moreover, the cartesian diagram

\CSquare{\affineline{};\ast; \Filstack; \B\Gm{};h;e;p;\pi;diagram-remark-pullbackA1A1Gm}

\noindent tells us that the pullback of $X^{\Fil}$ to $\affineline{}$ is isomorphic to $\affineline{}\times X$ and exhibits $X^\Fil$ as the quotient $[(X\times \affineline{})/\Gm{}]$  for the product of the $\Gm{}$-action on $X$ and the canonial action of $\Gm{}$ on $\affineline{}$. Moreover, this also tells us that $(X^{\Fil})^u\simeq X$.  

The filtered stacks obtained from $\Gm{}$-equivariant stacks using the
previous construction are precisely the \emph{split filtered stacks}.

\end{construction}

\medskip
\subsection{Kernel and fixed points of the Frobenius}
\mylabel{subsection-kernelandfixedpointsfrobenius}

We now turn back to Witt vectors and the interpolation between fixed points and kernel of the Frobenius. The starting ingredient is the following natural grading:

\begin{construction}
\mylabel{construction-GmactionWittVectors}
The abelian group of $p$-typical Witt vectors $\Wittp(\rmA)$ carries an action of the underlying multiplicative monoid of $\rmA$: for a given $a\in \rmA^\times$, we define $[a]:\Wittp\to \Wittp$ by the formula $(\lambda_n)_{n\in S}\mapsto (a^{n}. \lambda_n)_{n\in S}$ where $S$ is as in the \cref{guideWitt}. It is an easy exercice to check that in terms of the Ghost coordinates this becomes $(\omega_n)_{n\in S}\mapsto (a^n.\omega_n)_{n\in S}$ which is coordinatewise a map of abelian groups under addition. By the unique characterization of the group structure on $\Wittp$ re-engineered from Ghost coordinates (see the \cref{guideWitt}), the map $[a]:\Wittp\to \Wittp$ is a map of abelian groups. This defines an action of the multiplicative group scheme $\Gm{}$ on the group scheme $\Wittp$ and makes it a graded group scheme.

We consider the quotient stack $[\Wittp/\Gm{}]$  which lives canonically as an abelian group stack over $\B \Gm{}$.
\end{construction}

\medskip
\begin{remark}
\mylabel{remark-Frobeniuscompatiblegrading}
The operation $\Frob_p$ is compatible with the action
of the multiplicative monoid of $\affineline{}$ on $\Wittp$ in the following sense: given  $a\in \rmA$ and denoting by $[a]:\Wittp(\rmA)\to \Wittp(\rmA)$ the action by $a$, we have
$$
\Frob_p\circ [a]= [a^p]\circ \Frob_p
$$
 \end{remark}

\medskip

\begin{construction}
\mylabel{construction-trivialfamilyfiltered}
We let $\Wittp^{\Fil}$ be the output of the \cref{construction-filteredassociatedstack} applied to the action of $\Gm{}$ on $\Wittp$ of the \cref{construction-GmactionWittVectors}. As $[\Wittp/\Gm{}]$ is a group stack over $\B\Gm{}$, it follows that $\Wittp^{\Fil}$ is a group stack over $\Filstack$. Explicitly, 

$$
\Wittp^{\Fil}\simeq [(\Wittp\times \affineline{})/\Gm{}]
$$

\noindent the quotient of trivial family $\Wittp\times \affineline{}$ by the diagonal action of $\Gm{}$.  
\end{construction}

\medskip

We will now explain how to use the trivial family of the \cref{construction-trivialfamilyfiltered} to construct a new family that interpolates between Frobenius fixed points and the kernel.

\medskip

\begin{construction}
\mylabel{construction-interpolationkernelfixedpoints}
Consider the trivial group scheme $\Wittp\times \affineline{}$ over $\affineline{}$. For each $\rmA$ over $\integerslocalp$, consider the endomorphism of abelian groups
$$
\calG_p: \Wittp(\rmA)\times \rmA\to \Wittp(\rmA)\times \rmA
\,\,\,\, \text{ given by }\,\,\,(f, a)\mapsto (\Frob_p(f) - [a^{p-1}](f),\, a)
$$
This is functorial in $\rmA$ and defines a morphism of abelian group schemes over $\affineline{}$
$$
\calG_p: \Wittp\times \affineline{}\to \Wittp\times \affineline{}
$$
\end{construction}

\medskip

\begin{remark}
\mylabel{remark-underlyingandassociatedgradedFilCircle}
The \cref{remark-Frobeniuscompatiblegrading} is equivalent to the statement that  $\calG_p$ is $\Gm{}$-equivariant with respect to the diagonal action of $\Gm{}$  on the source and the twist by $(-)^p: \Gm{}\to \Gm{}$ on the target. This implies that the inclusion
$$\ker \calG_p\subseteq \Wittp\times \affineline{} $$ 
\noindent is $\Gm{}$-equivariant. 
The fiber of $\ker \calG_p$ over $0$ is $\Frobkernel$  (\cref{guideWitt}) and is closed under the $\Gm{}$-action. The fiber over any $\lambda\in \rmA^\times$, $(\ker \calG_p)_\lambda$ is isomorphic to  $(\ker \calG_p)_1\simeq \Frobfixed$ via the isomorphism sending $((\lambda_n)_{n\in S}, \lambda)\mapsto ([\frac{1}{\lambda}]((\lambda_n)_{n\in S}), 1)$. 
\end{remark}

\medskip

\begin{lemma}
\mylabel{lemma-Gpisfpqccover}
The morphism $
\calG_p: \Wittp\times \affineline{}\to \Wittp\times \affineline{} $ is a cover for the fpqc topology\footnote{\textcolor{black}{ie, surjective on points and flat}}. In particular, it is fpqc-locally surjective and we have a short exact sequence of  abelian group-stacks over $\affineline{}$
$$
\xymatrix{
0\ar[r]& \ker \calG_p \ar[r]& \Wittp\times \affineline{}\ar[r]^-{\calG_p}& \Wittp\times \affineline{}\ar[r]& 0
}
$$
In particular, $\ker \calG_p$ is a flat group scheme over $\affineline{}$.
\begin{proof}
As discussed in the \cref{ptypicalWittvectors}, the group scheme $\Wittp$ is the inverse limit of the system $\Wittpm$ of $m$-truncated $p$-typical Witt vectors and each restriction map $\Wittpm\to  \Wittpmone$ is isomorphic to a projection $\Wittpmone\times \affineline{}\to  \Wittpmone$. Each restriction map is a flat surjection between affine schemes, and therefore, an fpqc cover. In this case (see for instance \cite[\href{https://stacks.math.columbia.edu/tag/05UU}{Lemma 05UU}]{stacks-project}), in order to show that $\calG_p$ is an fpqc cover, it is enough to show that each composition

\begin{equation}
\label{mapcompositionwithtruncation}
\xymatrix{
 \Wittp\times \affineline{}\ar[r]^-{\calG_p}& \Wittp\times \affineline{}\ar[r]& \Wittpm\times \affineline{}
}
\end{equation}

\noindent is an fpqc cover. 

As remarked in the \cref{ptypicalWittvectors}, the Frobenius map on $m$-truncated Witt vectors factors as $\Wittpm\to \Wittpmone$. The $\Gm{}$-action on the contrary is defined levelwise  $[x]:\Wittpm\to \Wittpm$. By composing with the truncation maps $[x]:\Wittpm\to \Wittpm\to \Wittpmone$ we obtain a system of maps that after passing to the inverse limit, it recovers the $\Gm{}$-action on $\Wittp$. In this case, the composition (\ref{mapcompositionwithtruncation}) factors as

\begin{equation}
\label{mapcompositionwithtruncation2}
\xymatrix{
 \Wittp\times \affineline{}\ar[d]\ar[r]^-{\calG_p}& \Wittp\times \affineline{}\ar[d]\\
 \Wittpmplus\times \affineline{}\ar[r]&\Wittpm\times \affineline{}
}
\end{equation}

So that (as in \cite[\href{https://stacks.math.columbia.edu/tag/090N}{Lemma 090N}]{stacks-project}) to show that each composition (\ref{mapcompositionwithtruncation}) is an fpqc cover, it is enough to show that each truncated map is an fpqc cover

\begin{equation}
\label{mapcompositionwithtruncation3}
\xymatrix{
 \Wittpmplus\times \affineline{}\ar[r]&\Wittpm\times \affineline{}
}
\end{equation}

This morphism is a map of smooth group schemes that commutes with the projections to $\affineline{}:=\affineline{\integerslocalp}\simeq \affineline{\bbZ}\times \Spec(\integerslocalp)$ and therefore, as now for each $n$ the map is of finite presentation, to check that it is a fpqc cover, it is enough (\textcolor{black}{by a local criterion for flatness \cite[\href{https://stacks.math.columbia.edu/tag/039D}{Lemma 039D}]{stacks-project}, \cite[\href{https://stacks.math.columbia.edu/tag/03NV}{Section 03NV}]{stacks-project}, and since surjectivity in the topological sense can be tested on closed points  \cite[\href{https://stacks.math.columbia.edu/tag/0485}{Section 0485}]{stacks-project} and therefore on fields}) to check that it is so after base change to any field valued point $\Spec K\to \affineline{\bbZ}\times \Spec(\integerslocalp)$. As both projections to $\affineline{\integerslocalp}$ are compatible with the $\Gm{}$-action, it is enough to test the statement for the four different points
$$
 (0, \bbQ), \,\, (1, \bbQ),\,\, (0, \finitefieldp),\,\,(1, \finitefieldp)
$$

\noindent ie, the four maps

$$
\xymatrix{
\Wittpmplus\times \bbQ \ar[r]^-{\Frob_p}&\Wittpm\times \bbQ &&\Wittpmplus\times\bbQ \ar[r]^-{\Frob_p-\Id}&\Wittpm\times \bbQ\\
\Wittpmplus\times \finitefieldp \ar[r]^-{\Frob_p}&\Wittpm\times \finitefieldp &&\Wittpmplus\times \finitefieldp \ar[r]^-{\Frob_p-\Id}&\Wittpm\times \finitefieldp
}
$$

Using the Ghost components for truncated $p$-typical Witt vectors of (\ref{ghostcomponentstruncatedptypical}), the first two maps becomes isomorphic to, respectively, the projection away from the first coordinate and a linear projection

$$
\xymatrix{
\prod_{i\in S_{m+1}} \Ga{\bbQ} \ar[r]^-{}&\prod_{i\in S_{m}} \Ga{\bbQ} &&\prod_{i\in S_{m+1}} \Ga{\bbQ} \ar[r]^-{}&\prod_{i\in S_{m}} \Ga{\bbQ}\\
}
$$
\noindent both being clearly surjective and flat.

 After base-change to $\finitefieldp$, $\Frob_p$ acts by the standard power $p$ Frobenius on $\affineline{}$ (see \cref{guideWitt}) and the morphisms over $\finitefieldp$ become, respectively, the compositions

 $$\Frob_p:\prod_{i\in S_{m+1}} \affineline{}\to \prod_{i\in S_{m+1}}\affineline{}\to \prod_{i\in S_{m}} \affineline{}$$
 
 $$(\lambda_1, \lambda_p, .., \lambda_{p^{m-1}},\lambda_{p^{m}})\mapsto (\lambda_1^p, \lambda_p^p.., \lambda_{p^{m-1}}^p, \lambda_{p^{m}}^p)\mapsto   (\lambda_1^p,.., \lambda_{p^{m-1}}^p)$$

 $$\Frob_p-\Id: \prod_{i\in S_{m+1}} \affineline{}\to \prod_{i\in S_{m+1}} \affineline{}\to \prod_{i\in S_{m}} \affineline{}$$
 
 $$(\lambda_1, \lambda_p,..., \lambda_{p^{m-1}},\lambda_{p^{m}})\mapsto (\lambda_1^p-\lambda_1,.., \lambda_{p^{m-1}}^p-\lambda_{p^{m-1}}, \lambda_{p^{m}}^p-\lambda_{p^{m}})\mapsto   (\lambda_1^p-\lambda_{1},.., \lambda_{p^{m-1}}^p-\lambda_{p^{m-1}})$$

Both projections are fpqc covers. We conclude using the fact that the standard power $p$ Frobenius on $\affineline{}$,  $(-)^p$, is fpqc in our situation (see for instance \cite[\S 4, Exercice 3.13]{MR1917232}) and each $(-)^p-\Id$ is the Artin-Schreier isogeny well known to be an \'etale cover in characteristic $p$.

\end{proof}
\end{lemma}

\medskip

\begin{definition}
\mylabel{definition-Hgroupstack}
We define a filtered stack $\Hgroup \to \Filstack$ to be stack over $\Filstack$ given by the quotient 

$$
\Hgroup:= [(\ker \calG_p)/\Gm{}]\to \Filstack
$$
\end{definition}

\begin{remark}
\mylabel{remark-underlyingandassociatedgradedFilCircle2}
It follows from \cref{remark-underlyingandassociatedgradedFilCircle} that the underlying stack $\Hgroup^u$ is $\Frobfixed$ and the associated graded $\Hgroup^{\gr}$ is the quotient  $[\Frobkernel/\Gm{}]$ under the $\Gm{}$-action of the \cref{construction-GmactionWittVectors} and \cref{remark-underlyingandassociatedgradedFilCircle}.
\end{remark}

\section{The Filtered Circle}
\mylabel{section-filteredcircle}


\subsection{Hopf algebras and comodules}
\mylabel{hopfalgebrasandcomodules}
As discussed in the introduction, a key feature that any reasonable  construction of a ``filtered circle" should satisfy is that it must carry the structure of an abelian group stack. This will manifest itself in the form of a Hopf algebra structure on its cohomology, which on the one hand specializes to the  group structure on $\circle$, and on the other to a strict square-zero multiplication. \\

Before we proceed, we first set straight what we mean by the $\infty$-category of cosimplicial commutative $k$-algebras. Recall that the ordinary category of cosimplicial commutative $k$-algebras is defined to be $\Fun(\Delta, \CAlg_k^{0})$, the category of cosimplicial objects in (discrete) commutative $k$-algebras. By \cite[Théorème 2.1.2]{MR2244263}, this may be equipped with a closed (simplicial) model structure, with weak equivalences maps with induce isomorphisms on cohomology and fibrations being the level-wise surjections.

\begin{definition}
\mylabel{cosimplicialcommalgebra}
We define the \icategory of cosimplicial commutative $k$-algebras $\coSCRings{k}$ to be the localization of the category $\Fun(\Delta, \CAlg^{0}_k)$, equipped with the model category structure described above. 
\end{definition}

There exists a functor $\theta: \coSCRings{k} \to \CAlg_k^{\mathsf{ccn}}$, which arises as the cosimplicial version of the Dold-Kan construction; to be more precise, we remark that the model structure on $\coSCRings{k}$ is right induced by that of the \icategory $\Mod_k^{\Delta}$ of cosimplicial objects in (discrete) $k$-modules via the forgetful functor; this in turn obtains a model category structure from $\Mod_R^{ccn}$ via the adjunction
$$
N: \Mod_{R}^{\Delta} \to \Mod_R^{\mathsf{ccn}} \, \, \, \, \, D: \Mod_R^{\mathsf{ccn}} \to \Mod_{R}^{\Delta}
$$
making it into a Quillen equivalence. By the classical Eilenberg-Zilberg theorem, $N(-)$ commutes with tensor products (see \cite[2.1]{MR2244263}), so the composition with the forgetful functor,  $N \circ U: \coSCRings{k} \to \Mod_K^{\mathsf{ccn}}$ factors through $\CAlg_k^{\mathsf{ccn}}$, giving us $\theta$.  Moreover, this functor preserves homotopy limits as it is the composition of the forgetful functor 
$$
\coSCRings{} \to \Mod_R^{\mathsf{ccn}} 
$$
composed with the functor $N: \Mod_{R}^{\Delta} \to \Mod_R^{\mathsf{ccn}}$ which is a right adjoint functor. 

\begin{construction}
\mylabel{symmetricmonoidalstructurecosimplicialcommutativealgebras}
We will always consider the $\infty$-category $\coSCRings{k}$ equipped with its cocartesian monoidal structure. Explicitly, this is given by taking coproducts between cofibrant replacements in the model category of cosimplicial commutative algebras. In this case coproducts coincide with usual tensor products of cosimplicial commutative algebras.
\end{construction}

\medskip

\begin{construction}
\mylabel{generalizedBarconstruction}
Let $\C$ be a presentable \icategory endowed with the cartesian symmetric monoidal structure. We denote by $\groups(\C):=\Mongrouplike(\C)$ the category of group objects in $\C$. Following \cite[5.2.6.6, 4.1.2.11, 2.4.2.5.]{lurie-ha}, an explicit model is given by category of diagrams $\Fun(\Delta^{\op}, \C)$ satisfying the Segal conditions. The colimit functor 
\begin{equation}
\label{Barconstructioncolimitmodel}
\B: \groups(\C)\to \C_{\ast}
\end{equation}
\noindent lands in the category of pointed objects in $\C$ and admits a right adjoint, sending a pointed object $\ast\to X$ to its nerve. \\
There is a dual version of this given by cogroup objects $\cogroups(\C)$; these will be given by the category of cosimplicial objects $\Fun( \Delta, \C)$  of $\C$ satisfying the relevant Segal conditions in the opposite category. One has an adjunction 
$$
\lim_{\Delta}: \cogroups(\C) \rightleftarrows \C : \mathsf{CoNerve}
$$
\end{construction}

\medskip

\begin{definition}
\mylabel{cosimplicialHopfalgebra}
We define, following \cite[Definition 3.1.4]{MR2244263}, a \emph{$H_{\infty}$-Hopf algebra} to be a cogroup object in the $\infty$-category $\coSCRings{}$ of cosimplicial (commutative) algebras. Similarly, a \emph{graded $H_{\infty}$-Hopf algebra} will be a cogroup object in the category $\coSCRings{}^{\gr}$ of graded cosimplicial commutative rings. 

\begin{equation}
\label{cosimplicialHopfalgebrascategory}
\cosimplicialHopf:=\cogroups(\coSCRings{})
\end{equation}

\begin{equation}
\label{cosimplicialHopfalgebrascategorygraded}
\cosimplicialHopfgraded:=\cogroups(\coSCRings{}^{\gr})
\end{equation}

\end{definition}

\medskip

\begin{remark}
Equivalently, this is the data of a cosimplicial object $\mathsf{H}^{\bullet}$ in $\coSCRings{k}$ with the following properties:
\begin{itemize}
    \item $\mathsf{H}^0 \simeq k$, 
    \item the natural map  $\mathsf{H}^1 \otimes_k ...\otimes_k \mathsf{H}^1 \simeq \mathsf{H}^n$ is an equivalence.
\end{itemize}
As the symmetric monoidal structure on $\coSCRings{k}$ is cocartesian (see \cref{symmetricmonoidalstructurecosimplicialcommutativealgebras}), this is precisely the data of a cogroup object in this \icategory. 
\end{remark}



$$
$$
We now introduce a notion of \emph{homotopy coherent comodules} over a fixed $H_{\infty}$-Hopf algebra $\mathsf{H}^{\bullet}$. 

\begin{definition}
\mylabel{comodulesasQcohBH}
Given $\mathsf{H}= \mathsf{H}^{\bullet}\in \cosimplicialHopf$, we define the \icategory of $\mathsf{H}$-comodules as
$$
\CoMod_{\mathsf{H}}:=\Qcoh(\B \cospec(\mathsf{H}^{\bullet})) \simeq \lim_{\Delta} 
\Qcoh (\cospec(H^\bullet)). 
$$
\end{definition}

\begin{remark}
We explain the idea behind this definition a bit further. Since $\cospec: \coSCRings{k} \to \Stacks{k}$ is fully faithful, a cogroup object $\mathsf{H}^\bullet$ in $\coSCRings{}$ gives rise to a group object $G_\bullet$ in $\Stacks{k}$, which we may think of as the nerve of an affine group stack $G_1 = \cospec \mathsf{H}_1$. Then $\CoMod_{\mathsf{H}}$ may be thought of as the \icategory of $G_1= \cospec \mathsf{H}_1$ representations 
$$
\Qcoh(\B G) \simeq \lim_{\Delta} \Qcoh(G_\bullet)
$$
\end{remark}

\begin{remark}
This notion will be important to us later on, when we identify $\mathsf{CoMod}_{C^*(\mathsf{Ker}, \calO)^{\bullet}}$ with the underlying $\infty$-category of the category of strict dg-comodules over a particular dga.  
\end{remark}

\personal{add comment about comodules in the sense of higher algebra. we defined only commmutative cosimplicial Hopf algebra.}


$$
$$



\subsection{ Affinization and Cohomology of stacks}
\mylabel{FilteredcircleandAffinization}

At this point, after the constructions in the previous section, we have in fact proved \cref{filteredcircletheorem}-(iii). We absorb it in the following definition:

\begin{definition}\mylabel{definition-filteredcircle}
The \emph{filtered circle} (local at $p$) is the filtered stack given by the classifying stack (relatively to $\Filstack$ and \textcolor{black}{with respect to the fpqc-topology}) of the filtered abelian group stack $\Hgroup$:
$$
\Filcircle:= \B \Hgroup \to \Filstack
$$
\end{definition}

\medskip

\begin{construction}
\mylabel{remark-abeliangroupstructurefilteredcircle}
The filtered stack $\Filcircle$,  being the classifying stack of a filtered abelian group stack is again a filtered abelian group stack. In other words, it carries a canonical abelian group structure compatible with the filtration. In particular, we can take its classifying stack 
$$\B\Filcircle\simeq \Kspace(\Hgroup,2)$$

\noindent obtained as the \textcolor{black}{fpqc} quotient $\Filstack/\Filcircle$.

\end{construction}

\medskip

\begin{remark}
\mylabel{flatgroupschemesrelativeflatatlas}
If follows from the \cref{lemma-Gpisfpqccover} and \cref{definition-Hgroupstack} that $\Hgroup$ is a flat  group stack relatively to $\Filstack$ and also, relatively affine. In particular, both $\Filcircle$ and $\B\Filcircle$ are flat group stacks over $\Filstack$. This implies that both admit a flat hypercover by affines with flat transition maps. For $\Filcircle$ this can be seen directly using the bar construction for $\Hgroup$ relatively to 
$\Filstack$

$$
\xymatrix{
 \cdots \,\, \,  \, \,\,\,  \Hgroup^{\times_2}\,\, \ar[dr]\ar@<-0.7ex>[r]_{} \ar@<-.0ex>[r]^{} \ar@<+0.7ex>[r]^{} & \ar[d] \, \,\,\, \, \,\,\,\Hgroup \ar@<-1ex>[r]_{} \ar@<1ex>[r]^{}& \ar[dl]\Filstack \, \,\,\, \ar[r]&\Filcircle\ar[dll]\\
 &\Filstack&&\\
}
$$

\noindent For $\B\Filcircle$, we use the fact that $\Nerve(\Delta^\op)$ is sifted, to exhibit $\B\Filcircle$ as the geometric realization of the diagonal of the double bar construction:

\begin{equation}
\label{diagonalflatatlasbyaffines}
\xymatrix{
 \cdots \,\, \,  \, \,\,\,  \Hgroup^{\times_4}\,\, \ar[dr]\ar@<-0.7ex>[r]_{} \ar@<-.0ex>[r]^{} \ar@<+0.7ex>[r]^{} & \ar[d] \, \,\,\, \, \,\,\,\Hgroup \ar@<-1ex>[r]_{} \ar@<1ex>[r]^{}& \ar[dl]\Filstack \, \,\,\, \ar[r]&\B\Filcircle\ar[dll]\\
 &\Filstack&&\\
}
\end{equation}

\end{remark}

\medskip

\begin{construction}
\mylabel{tstructureQcohBfilteredcircle}
The discussion in \cref{flatgroupschemesrelativeflatatlas} implies that  both $\Qcoh(\Filcircle)$ and $\Qcoh(\B \Filcircle)$ admit left-complete $t$-structures as described in \cref{left-complete-t-structure}-c). Concretely, for  $\Qcoh(\B \Filcircle)$, it is induced by transferring the $t$-structure on $\Qcoh(\Filstack)$ by pullback along the atlas $\Filstack\to \B\Filcircle$. Here the $t$-structure on $\Qcoh(\Filstack)$ is itself induced by the standard $t$-structure on $\Qcoh(\affineline{})$ via the same argument as in \cref{tstructureQcohBfilteredcircle} using the flat simplicial atlas encoding the geometric action of $\Gm{}$ on $\affineline{}$ of weight 1 (see \cref{construction-geometricstackfiltration}). 

Under the equivalence $\Qcoh(\Filstack)\simeq \Fil(\Spectra)$ of \cref{prop:gradingsandstacks} this $t$-structure corresponds to the levelwise $t$-structure on filtered spectra (see \cite[Proof of 6.1]{tasos}).
\end{construction}

\medskip

\begin{remark}
\mylabel{basechangeBfilteredcircle}
The conclusion of the \cref{flatgroupschemesrelativeflatatlas} is also particularly important in the context of the base change property of \cref{globalsectionsoffilteredarefiltered} with $\pi:\B\Filcircle\to \Filstack$

\begin{equation}
\label{primitivebasechangediagram}
\xymatrix{
(\B\Filcircle)^{\gr} \ar[d]_{\pi^{\gr}}\ar[r]^{0_a} & \B\Filcircle \ar[d]^{\pi} & (\B\Filcircle)^\mathrm{u} \ar[d]^{\pi^\mathrm{u}}\ar[l]_{1_a}\\
\B \Gm{}\ar[r]^{0}& \Filstack & \ar[l]_{1} \ast
}
\end{equation}

Indeed, using \cref{flatgroupschemesrelativeflatatlas} we conclude that that base-change holds for $\pi$ in the following contexts:

\begin{itemize}
    \item from the \cref{globalsectionsoffilteredarefiltered}-(ii): the Beck-Chevalley transformation

\begin{equation}
\label{beckchevalley2}
0^\ast\,\, \pi_\ast\,\, \to (\pi^{\mathrm{gr}})_\ast\,\,(0_a)^\ast\,\,
\end{equation}

\noindent is an equivalence for any $\mathrm{F}\in \Qcoh(\B\Filcircle)$.

 \medskip
 
\item from the \cref{globalsectionsoffilteredarefiltered}-(i):  we can only guarantee that the Beck-Chevalley transformation

\begin{equation}
\label{beckchevalley2foropen}
1^\ast\,\, \pi_\ast\,\, \to (\pi^{\mathrm{u}})_\ast\,\,(1_a)^\ast\,\,
\end{equation}

\noindent is an equivalence for homologically bounded above objects, ie, in $\Qcoh(\B\Filcircle)_{<\infty}$ (for the t-structure of \cref{left-complete-t-structure}).
\end{itemize}
\end{remark}

\medskip

One of the claims in \cref{filteredcircletheorem}-(i) is that our filtered circle $\Filcircle$ is related to the topological circle $\circle$ via the notion of \emph{affinization} of \cite{MR2244263}. \emph{Affine stacks} were introduced in \cite[Def. 2.2.4]{MR2244263} (see also \cite{lurie-DAGVIII} where these are called \emph{Coaffine}). Informally, an (higher) stack $X$ is affine if it can be recovered from its cohomology of global sections. More precisely:

\begin{review}[Affine Stacks]\mylabel{reviewaffinestacks}
See \cref{notationcosimplicialandsimplicial}. By \cite[2.2.3]{MR2244263} the \ifunctor $\cospec: \coSCRings{\integerslocalp}^{\op}\to\Stacks{\integerslocalp} $ is fully faithful and admits a left adjoint $\cochainscosimplicial(-, \structuresheaf)$ that enhances the standard $\Einfinity$-algebra structure of cohomology of global sections $\cochains(-, \structuresheaf)$ with a structure of cosimplicial commutative algebra, namely, it provides a lifting in $\inftycatsbig$

$$
\xymatrix{
&\coSCRings{\integerslocalp}\ar[d]^\theta\\
\Stacks{\integerslocalp}\ar[ur]^{\cochainscosimplicial(-, \structuresheaf)}\ar[r]_{\cochains(-, \structuresheaf)}& \CAlg_{\integerslocalp}
}
$$

We say that $X\in \Stacks{\integerslocalp}$ is affine if it lives in the essential image of $\cospec$. More generally, given $X\in \Stacks{\integerslocalp}$, we define its  affinization as the stack $\cospec(\cochainscosimplicial(X, \structuresheaf))$ \cite[2.3.2]{MR2244263}.
\end{review}

\begin{proposition}
\mylabel{filteredcircleisaffine}
Both $\Filcircle$ and $\B \Filcircle$ are relatively affine stacks over $\Filstack$ in the sense of \cite{MR2244263}. By stability of affine stacks under base-change \footnote{see \cite[2.2.7, 2.2.9, Remarque p.49]{\Bertrandaffine}} so are all the stacks $\Filcircleund$, $\Filcirclegr$, $\B\Filcircleund$ and $\B\Filcirclegr$.
\begin{proof}
Let us start by showing that $\Filcircle$ is relatively affine over $\Filstack$. Using the atlas $\affineline{}\to \Filstack$ this is equivalent to showing that $\B\ker \calG_p$ is an affine stack over $\affineline{}= \Spec(\integerslocalp[T])$ . Applying $\B$ to the short exact sequence in the \cref{lemma-Gpisfpqccover}, we get a fiber sequence of fpqc sheaves over $\affineline{}$

$$
\xymatrix{
\B\ker \calG_p\ar[r]\ar[d]& \B\Wittp \times \affineline{}\ar[d]^{\B\calG_p\times \Id}\\
\affineline{}\ar[r]& \B\Wittp \times \affineline{}
}
$$

By \cite[2.2.7]{\Bertrandaffine} the class of affine stacks over any commutative ring is closed under limits. Therefore, to conclude that $\B\ker \calG_p$  is affine it is enough to show that  $\B\Wittp \times \affineline{}$ is affine. But the group scheme $\Wittp$ can be written as a limit $\lim \Wittpm$ where each projection $\Wittpmplus\to \Wittpm$ is a smooth epimorphism of affine groups with fiber $\Ga{}$. We claim that this fact together with the fact with the Witt schemes are truncated, implies that the limit decomposition of $\Wittp$ induces a limit decomposition

\begin{equation}
\label{limitdecompositionB}
    \B\Wittp\simeq \lim \,\,\B\Wittpm.
\end{equation}

Indeed, the Milnor sequences (see for instance  \cite[2.2.9]{MR2840650}) tell us that the  obstructions for the limit decomposition \eqformula{limitdecompositionB} are given by the groups  $\lim^1 \pi_{i}\Map_{fpqc}(X, \Wittpm)$ for $i\geq 0$ and $X$ affine classical. For $i\geq 1$ these groups vanish because the mapping spaces are discrete. For $i\geq 0$ we have  $\pi_{i}\Map_{fpqc}(X, \Wittpm)\simeq \mathsf{H}^i_{fpqc}(X, \Wittpm)$, which vanishes for $i>0$ because
$X$ is affine. \personal{ For $i=0$ this is the Mittag-Leffler condition provided by the (fppf) surjectivity proven above.\\}

In fact, more generally, we also have the same decomposition for the iterated construction

\begin{equation}
\label{limitdecompositionBiterated}
    \B^j\Wittp\simeq \lim\,\, \B^j\Wittpm.
\end{equation}

This follows again because the mapping spaces are discrete and because of the vanishing of the higher cohomology groups

\begin{equation}
\label{vanishinghighercohomologyWitt}
    \pi_0 \Map_{fpqc}(X,\B^j \Wittpm)=\mathsf{H}^j_{fpqc}(X, \Wittpm)=0\,\,\,\, j\geq 1
\end{equation}

\noindent for $X$ affine classical. One can see this by induction using the long exact sequences extracted from the fact  $\Wittpmplus$ is an extension of $\Wittpm$ by $\Ga{}$, 

\begin{equation}
\label{groupextensionsWittrestriction}
 0\to \Ga{}\to \Wittpmplus\to \Wittpm\to 0
   \end{equation}

The result is true for $\Ga{}$ and  as $\Wittp^{(1)}=\Ga{}$, by induction, it is true all for $m$.

\medskip

Finally, knowing \eqformula{limitdecompositionB}, by \cite[2.2.7]{\Bertrandaffine} it becomes enough to show that each $\B\Wittpm$ is affine. But now, each of the group extensions \eqformula{groupextensionsWittrestriction} is classified by a map of group stacks $\Wittpm\to \rmK(\Ga{}, 1)$, which we can write as a map $\B\Wittpm\to \rmK(\Ga{}, 2)$. By definition of this map, we have a pullback square

$$
\xymatrix{
\B \Wittpmplus\ar[r]\ar[d]& \ast \ar[d]\\
\B\Wittpm \ar[r]& \Kspace(\Ga{}, 2)
}
$$

Each $\Kspace(\Ga{}, n)$ is known to be affine \cite[2.2.5]{\Bertrandaffine}. An induction argument concludes the proof.\\

To prove the claim for $\B \Filcircle$ it is enough to show that $\B (\B\ker \calG_p)$ is affine over $\affineline{}$. The argument runs the same, using the iterated formula \eqformula{limitdecompositionBiterated}.

\end{proof}
\end{proposition}

\begin{remark}
\mylabel{KHnareaffine}
The proof of the \cref{filteredcircleisaffine} shows that more generally, the stacks $\Kspace(\Hgroup, n)$ are affine.
\end{remark}

\subsection{The Underlying Stack of $\Filcircle$}
\mylabel{underlyingfilteredcircle}

As we now know, by the \cref{filteredcircleisaffine}, the underlying stack $\Filcircleund$ is affine.  We would like, in order to establish \cref{filteredcircletheorem}-(i), to identify it with the affinization of $\circle$ over $\integerslocalp$.

\begin{construction}
\mylabel{mapS1toBFix}
Recall that $\Filcircle=\B\Hgroup$ (\cref{definition-Hgroupstack}). Let $\bbZ$ denote the constant group scheme with value $\bbZ$. There is a canonical morphism of group schemes $\bbZ \to \Wittp$ given by  $1 \mapsto (1,0,0,...) \in \Wittp(R)$. This Witt vector is fixed by the Frobenius \footnote{Can easily be checked using $\mathrm{Ghost}(1,0,...)=(1,1,1,..)$} so clearly the map factors through $\Hgroup^{u}= \Frobfixed$. By passing to classifying stacks we obtain a morphism of stacks

\begin{equation}
\label{affinizationunderlyingmap}
u:\circle = \B\bbZ \to \Filcircleund = \B \Frobfixed
\end{equation}

\noindent with affine target.
\end{construction}

\medskip

The main result of this section is the following:

\begin{proposition}
\mylabel{propositionaffinizationunderlying}
\noindent The map (\ref{affinizationunderlyingmap}) displays $\Filcircleund=\B\Frobfixed$ as the affinization of $\circle$ over $\Spec \bbZ_{(p)}$. By \cite[Corollaire 2.3.3]{\Bertrandaffine}, this is equivalent to say that (\ref{affinizationunderlyingmap}) induces an equivalence on cosimplicial cochain algebras
\begin{equation}
\label{affinizationunderlyingmapcomplexes}
\cochainscosimplicial( \B \Frobfixed, \structuresheaf) \simeq \cochainscosimplicial(\circle, \integerslocalp)
\end{equation}
\end{proposition}

\medskip

We will establish below in the \cref{localcriterionaffinization} a local criterion for affinization: in order to prove that the map (\ref{affinizationunderlyingmap}) is the affinization of $\circle$ over $\integerslocalp$ it is enough to know that when base-changed to $\bbQ$ and $\finitefieldp$, the maps

\begin{equation}
\label{affinizationunderlyingmap2}
\circle \to \B\Frobfixed_{|_{\bbQ}}:=\B\Frobfixed\times_{\Spec(\integerslocalp)} \Spec(\bbQ)
\end{equation}

\begin{equation}
\label{affinizationunderlyingmap3}
\circle \to \B\Frobfixed_{|_{\finitefieldp}}:=\B\Frobfixed\times_{\Spec(\integerslocalp)} \Spec(\finitefieldp) 
\end{equation}

\noindent are affinizations of $\circle$, respectively, over $\bbQ$ and $\finitefieldp$. 

For the moment let us describe the targets of the maps \eqformula{affinizationunderlyingmap2} and \eqformula{affinizationunderlyingmap3}.  By the  \cref{remark-characteristiczerocase}, we already know that $\B\Frobfixed_{|_{\bbQ}}\simeq \B\Ga{\bbQ}$. It remains to work over $\finitefieldp$:

\begin{lemma} \mylabel{basechangetoFp}
There is an an equivalence 

\begin{equation}
\label{formulabasechangetoFp}
\Filcircleund_{|_{\finitefieldp}}  \simeq \B \padicintegers
\end{equation}

\noindent where $\B \padicintegers$ is the classifying stack of the proconstant group scheme with values in the $p$-adic integers.

\begin{proof}
As discussed in \cref{guideWitt}, for an $\finitefieldp$-algebra $\rmA$ the map $\Frob_p$ on $\Wittp(\rmA)$ takes the form
$$
\Frob_{p}(\lambda_1,\lambda_{p},..\lambda_{p^k},...) = (\lambda_1^p, \lambda_{p}^{p},...\lambda_{p^k}^p,...)
$$

\noindent so that levelwise it coincides with the standard Frobenius of $\rmA$. Now, for each $n$ we have Artin-Schreier-Witt exact sequences \cite[Proposition 3.28]{MR565469}

\begin{equation}
\label{exactsequencefixedpointsArtinSchreier}
\xymatrix{
0\ar[r]& (\bbZ/p^n \bbZ) \ar[r]& \Wittpmfinitefieldp \ar[rr]^-{\Frob_{p} - \Id}&& \Wittpmfinitefieldp \ar[r]& 0
}
\end{equation}
\noindent where we consider  $(\bbZ/p^n \bbZ)$ as the constant-valued group scheme over $\finitefieldp$. By passing to the limit over $n$ we obtain exact sequences

\begin{equation}
\label{exactsequencefixedpointsArtinSchreier33}
\xymatrix{
0\ar[r]& \padicintegers  \ar[r]& \Wittpfinitefieldp \ar[rr]^-{\Frob_{p} - \Id}&& \Wittpfinitefieldp\ar[r]& 0
}
\end{equation}

\noindent from where we can conclude the identification $\Frobfixedfinitefieldp\simeq \padicintegers$.
\end{proof}
\end{lemma}

\medskip

The following is a key computation:

\begin{proposition}\cite[Corollaire 2.5.3 and Lemma 2.5.2 ]{\Bertrandaffine}
\mylabel{affinizationsoverQandFp}
The affinization of $\circle$ over $\bbQ$ is $\B \Ga{}$ and over $\finitefieldp$ is $\B \padicintegers$ .
\end{proposition}

Let us now work our local criterion for affinization over $\integerslocalp$. We start with a simple remark:

\begin{remark} \mylabel{basechangeeeee}
Let $M$ be an object in $\Mod_{\integerslocalp}(\Spectra)$. Suppose that both base changes $M\otimes_{\integerslocalp} \bbQ$ and $M\otimes_{\integerslocalp} \finitefieldp$ are zero. Then $M\simeq 0$. Indeed, as  $\bbQ$ is obtained from $\integerslocalp$ by inverting $p$, $M\otimes_{\integerslocalp} \bbQ$ is obtain via the filtered colimit of the diagram given by multiplication by $p$
\begin{equation}
\label{multiplicationbyp}
\xymatrix{
... \ar[r]^{.p}& M\ar[r]^{.p}& M\ar[r]^{.p}& M\ar[r]^{.p}& ...}
\end{equation} 

At the same time, we know that $\finitefieldp$ is obtained from $\integerslocalp$ via an exact sequence
$$
\xymatrix{
0\ar[r]& \integerslocalp \ar[r]^{.p}& \integerslocalp\ar[r] & \finitefieldp \ar[r]& 0}
$$
In particular, we have a cofiber-fiber sequence

\begin{equation}
\label{multiplicationbypfibersequence}
\xymatrix{
\ar[d]M\simeq M\otimes_{\integerslocalp} \integerslocalp\ar[r]^{.p}& M\simeq M\otimes_{\integerslocalp}\integerslocalp\ar[d]\\
0\ar[r]& M\otimes_{\integerslocalp} \finitefieldp
}
\end{equation}
It follows that $M\otimes_{\integerslocalp} \finitefieldp\simeq 0$ if and only if the multiplication by $p$ is an equivalence of $M$. In that case, the colimit of the diagram (\ref{multiplicationbyp}), meaning $M\otimes_{\integerslocalp} \bbQ$, is  equivalent to $M$. But the assumption $M\otimes_{\integerslocalp} \bbQ\simeq 0$ concludes that $M\simeq 0$. 

In particular, given $f: E \to F$ a morphism of chain complexes over $\bbZ_{(p)}$, if the two base changes to $\bbQ$ and $\finitefieldp$ are equivalences, then so is $f$. 
\end{remark}

\begin{lemma}
\mylabel{flatbasechangecohomologyBFixed}
Consider the pullback diagrams:
$$
\xymatrix{
\B\Ga{\bbQ}\simeq \B\Frobfixed_{|_{\bbQ}}\ar[r]^-{J}\ar[d]^{f_\bbQ}&\B\Frobfixed\ar[d]^f& \B\Frobfixed_{|_{\finitefieldp}}\ar[l]_-{I}\ar[d]^{f_{p}}\simeq \B\padicintegers\\
\Spec(\bbQ)\ar[r]^j&\Spec(\integerslocalp)&\Spec(\finitefieldp)\ar[l]_i
}
$$
Then we have equivalences of cosimplicial commutative $k$-algebras

\begin{equation}
\label{basechanecohomologyFixedFrob}
\cochainscosimplicial(\B\Frobfixed, \structuresheaf)\otimes_{\integerslocalp}\bbQ\simeq \cochainscosimplicial(\B\Frobfixed_{|_{\bbQ}}, \structuresheaf) \,\,\,\, \text{ and }\,\,\, \cochainscosimplicial(\B\Frobfixed, \structuresheaf)\otimes_{\integerslocalp}\finitefieldp\simeq \cochainscosimplicial(\B\Frobfixed_{|_{\finitefieldp}}, \structuresheaf)
\end{equation}

\begin{proof}
To show this lemma one could evoke directly the two cases in the \cref{basechangeBfilteredcircle} and \cref{globalsectionsoffilteredarefiltered}: one deals with $\bbQ$ the other with the lci closed immersion $i$.

We shall unfold the argument in more detail.  Let us start with the rational equivalence. We write $A^\bullet$ for the co-simplicial object $$
[n]\mapsto \C^\ast(\Frobfixed, \structuresheaf)^{\otimes n}
$$
\noindent which is a diagram of discrete and flat modules over $\integerslocalp$ with flat transition maps
$$A^\bullet:\Nerve(\Delta)\to \Mod_{\integerslocalp}^{\leq 0}$$
Let $\mathrm{Tot}(A^\bullet)$ denote the totalization of this co-simplicial object which we can read as the homotopy limit $\lim_{\Delta}A^\bullet$. The fact that the cosimplicial object is levelwise discrete implies that the associated spectral sequence is in the third quadrant and that the homotopy limit is also in $\Mod_{\integerslocalp}^{\leq 0}$ - use the dual of \cite[1.2.4.5]{lurie-ha}. In particular, $\lim_\Delta A^\bullet$ satisfies the homological bounded condition of the \cref{basechangeBfilteredcircle}-(ii). The claim in the lemma is that the comparison map
$$(\lim_\Delta A^\bullet)\otimes^{\mathbb{L}}_{\integerslocalp}\bbQ\to \lim_{\Delta} (A^\bullet\otimes_{\integerslocalp}^{\mathbb{L}}\bbQ)
$$
is an equivalence. 
Therefore, it is enough to show that for every $n\geq 0$ the map induces an isomorphism on cohomology groups
$$\mathrm{H}^n((\lim_\Delta A^\bullet)\otimes^{\mathbb{L}}_{\integerslocalp}\bbQ)\to \mathrm{H}^n(\lim_{\Delta} (A^\bullet\otimes_{\integerslocalp}^{\mathbb{L}}\bbQ))
$$
For the l.h.s. we have a chain of isomorphisms
$$\mathrm{H}^n((\lim_\Delta A^\bullet)\otimes^{\mathbb{L}}_{\integerslocalp}\bbQ)\simeq \mathrm{H}^n(\lim_\Delta A^\bullet)\otimes_{\integerslocalp}\bbQ\simeq  \mathrm{H}^n(\lim_{\Delta_{\leq n+1}} A^\bullet)\otimes_{\integerslocalp}\bbQ$$
\medskip
\noindent the first because $\bbQ$ is discrete flat over $\integerslocalp$; the second follows from \cite[1.2.4.5(5)]{lurie-ha}.
For the r.h.s, we have
$$\mathrm{H}^n(\lim_{\Delta} (A^\bullet\otimes_{\integerslocalp}^{\mathbb{L}}\bbQ))\simeq 
\mathrm{H}^n(\lim_{\Delta_{\leq n+1}} (A^\bullet\otimes_{\integerslocalp}\bbQ))\simeq \mathrm{H}^n([\lim_{\Delta_{\leq n+1}} A^\bullet]\otimes_{\integerslocalp}\bbQ)
$$
\noindent where first we used again \cite[1.2.4.5(5)]{lurie-ha} and in the second isomorphism we used the fact that since the homotopy limit is now finite, it commutes with derived tensor products.
Finally, the comparison between the r.h.s and the l.h.s follows from the isomorphism
$$
\mathrm{H}^n(\lim_{\Delta_{\leq n+1}} A^\bullet)\otimes_{\integerslocalp}\bbQ\simeq \mathrm{H}^n([\lim_{\Delta_{\leq n+1}} A^\bullet]\otimes_{\integerslocalp}\bbQ)
$$
\noindent which holds because $\bbQ$ is discrete and flat over $\integerslocalp$.\\


Let us now show the second equivalence, over $\finitefieldp$. The \cref{lemma-Gpisfpqccover} gives us a short exact sequence of fpqc sheaves of groups 
$$
\xymatrix{
0\ar[r]& \Frobfixed  \ar[r]& \Wittp \ar[r]^-{\Frob_{p} - \Id}& \Wittp \ar[r]& 0
}
$$
\noindent where the map $\Frob_{p} - \Id$ is flat. This implies that $\Frobfixed$ is a flat group scheme over $\Spec(\integerslocalp)$ and that the square

\CSquare{\Frobfixed ;\Wittp; \Spec(\integerslocalp);\Wittp;;\Frob_{p} - \Id;;0;}

\noindent is actually a derived fiber product and so the diagram is

\CSquare{\Frobfixed_{|_{\finitefieldp}} ;\Frobfixed;\Spec(\finitefieldp) ;\Spec(\integerslocalp);;;;;}

The derived base change \cite[2.5.3.3, 2.5.4.5]{lurie-sag} formula applied to the last square tells us that

\begin{equation}
\label{basechanecohomologyFixedFrob22}
\cochainscosimplicial(\Frobfixed, \structuresheaf)\otimes_{\integerslocalp}\finitefieldp\simeq \cochainscosimplicial(\Frobfixed_{|_{\finitefieldp}}, \structuresheaf)
\end{equation}
\noindent where $\cochainscosimplicial(\Frobfixed, \structuresheaf)$ is the Hopf-algebra of functions on the affine group scheme $\Frobfixed$. To show that \formula{basechanecohomologyFixedFrob22} implies \formula{basechanecohomologyFixedFrob} over $\finitefieldp$ we use the description of the classifying stack $\B\Frobfixed$ as the geometric realization of the simplicial object
\SimplicialObjectNoLabels{\Spec(\integerslocalp);\Frobfixed;\Frobfixed\times \Frobfixed}
\noindent which exhibits 
\begin{equation}
\label{limiformulacohomologyBG}
\cochainscosimplicial(\B\Frobfixed, \structuresheaf)\simeq \lim_{[n]\in \Delta}\, \cochainscosimplicial(\Frobfixed, \structuresheaf)^{\otimes_n}
\end{equation}
\medskip

\noindent Here, because $\Frobfixed$ is affine, the Kunneth formulas for the cohomology of its cartesian powers are automatic. The same argument also tells us that

\begin{equation}
\label{limiformulacohomologyBG2}
\cochainscosimplicial(\B\Frobfixed_{|_{\finitefieldp}}, \structuresheaf)\simeq \lim_{[n]\in \Delta}\, \cochainscosimplicial(\Frobfixed_{|_{\finitefieldp}}, \structuresheaf)^{\otimes_n}
\end{equation}

But now we know that, as in the \cref{basechangeeeee}, $\C^\ast(\B\Frobfixed, \structuresheaf)\otimes_{\integerslocalp}\finitefieldp$ is the cofiber of multiplication by $p$

\begin{equation}
\label{limitsquare}
\xymatrix{
\ar[d]\cochainscosimplicial(\B\Frobfixed, \structuresheaf)\ar[r]^{.p}& \cochainscosimplicial(\B\Frobfixed, \structuresheaf)\ar[d]\\
0\ar[r]& \cochainscosimplicial(\B\Frobfixed, \structuresheaf)\otimes_{\integerslocalp}\finitefieldp
}
\end{equation}

As multiplication by $p$ is actually happening levelwise in \eqformula{limiformulacohomologyBG}, we deduce that the square \eqformula{limitsquare} is obtained from the squares

\begin{equation}
\label{limitsquare2}
\xymatrix{
\ar[d]\cochainscosimplicial(\Frobfixed, \structuresheaf)^{\otimes_n}\ar[r]^{.p}& \cochainscosimplicial(\Frobfixed, \structuresheaf)^{\otimes_n}\ar[d]\\
0\ar[r]& \cochainscosimplicial(\Frobfixed, \structuresheaf)^{\otimes_n}\otimes_{\integerslocalp}\finitefieldp
}
\end{equation}

\noindent by passing to the limit in $\Delta$. Now, \formula{basechanecohomologyFixedFrob} over $\finitefieldp$  follows from the comparison \eqformula{basechanecohomologyFixedFrob22} applied to each entry of the cartesian square \eqformula{limitsquare2}, upon passing to the limit and using \formula{limiformulacohomologyBG2}. 

\end{proof}
\end{lemma}

In fact we can show something more general than \cref{flatbasechangecohomologyBFixed} that will be useful to us later:

\begin{proposition}
\mylabel{lemma-Bfixfinitecohdimension}
The structure map $\B \Frobfixed \to \Spec \, \integerslocalp$ is of finite cohomological dimension. In particular, by \cite[A.1.5, A.1.6,A.1.9]{1402.3204} (see also \cite[9.1.5.6]{lurie-sag}), it verifies base-change against any map $Y\to \Spec \, \integerslocalp$.
\end{proposition}

\medskip

The \cref{lemma-Bfixfinitecohdimension} follows from the following more general analysis:

\begin{lemma}
\mylabel{mainlemmafinitecohodimension}
Let $G$ be an affine group scheme over $\integerslocalp$. Assume that \textcolor{black}{$G$ is flat} and that both base change maps 

$$f_{\bbQ}:\B G_{|_{\bbQ}}\to \Spec\, \bbQ \,\,\, \text{ and }\,\,\, f_{\finitefieldp}:\B G_{|_{\finitefieldp}}\to \Spec\, \finitefieldp \,\,\,$$ 

\noindent are of finite cohomological dimension $d$. Then $f:\B G\to \Spec \integerslocalp$ is of finite cohomological dimension $\textcolor{black}{d+1}$.
\begin{proof}
To check that $f$ is of finite cohomological dimension it is enough (see \cite[A.1.4]{1402.3204}) to see that $f_\ast$ sends $ \Qcoh(\B G)^{\heartsuit}$ to $\Mod_{\integerslocalp}^{\geq -d}$ for some $d>0$. Let us now consider the pullback squares

\begin{equation}
    \label{diagramfinitecohomologicaldimension}
    \xymatrix{
\B G_{|_{\bbQ}}\ar[r]^-{J}\ar[d]^{f_\bbQ}&\B G \ar[d]^f& \B G{|_{\finitefieldp}}\ar[l]_-{I}\ar[d]^{f_{p}}\\
\Spec(\bbQ)\ar[r]^j&\Spec(\integerslocalp)&\Spec(\finitefieldp)\ar[l]_i
}
\end{equation}

Let $M\in \Qcoh(\B G)^{\heartsuit}$. In this case we can use the cases in \cite[A.1.3]{1402.3204}:

\begin{itemize}
    \item one to conclude that tbe Beck-Chevalley transformation

$$
i^\ast\, f_\ast\, M \to (f_p)_\ast\, I^\ast \, M
$$

\noindent is an equivalence since $i$ is an lci closed immersion and therefore finite of finite Tor-dimension. \\

\medskip

\item the other to conclude that the Beck-Chevalley transformation

$$
j^\ast\, f_\ast\, M \to (f_\bbQ)_\ast\, J^\ast \, M
$$

\noindent is also an equivalence since $M$ is in the heart and therefore homologically bounded above.
\end{itemize}

\medskip

Since pullbacks are right t-exact, it follows that $$I^\ast( M)\in \Qcoh(\B G_{|_{\finitefieldp}})_{\geq 0}\,\, \text{ and }\,\,\, J^\ast( M)\in \Qcoh(\B G_{|_{\bbQ}})_{\geq 0}$$

Now we use our assumption that both $\B G_{|_{\bbQ}}$ and $\B G_{|_{\finitefieldp}}$ are of finite cohomological dimension $d$, combined with \cite[A.1.6]{1402.3204}, to deduce that

$$(f_p)_\ast I^\ast( M)\in \Mod_{\finitefieldp}^{\geq -d}\,\, \text{ and }\,\,\, (f_\bbQ)_\ast J^\ast( M)\in \Mod_\bbQ^{\geq -d}$$

Our final goal is to show that $f_\ast M$ is in $\Mod_{\integerslocalp}^{\geq -(d+1)}$. Using the fact that the Beck-Chevalley transformations above are equivalences, we are therefore reduced to prove the following statement: let $N\in \Mod_{\integerslocalp}$ and assume that

$$i^\ast( N)\in \Mod_{\finitefieldp}^{\geq -d}\,\, \text{ and }\,\,\,j^\ast( N)\in \Mod_\bbQ^{\geq -d}$$

Then $N\in \Mod_{\integerslocalp}^{\geq -(d+1)}$, ie, $\pi_i(N)=0$ for $i<-(d+1)$ (homological notation). To see this we use the cofiber sequence \cref{multiplicationbypfibersequence}

\begin{equation}
\xymatrix{
\ar[d]N\ar[r]^{.p}& M\ar[d]\\
0\ar[r]& N\otimes_{\integerslocalp} \finitefieldp
}
\end{equation}

\noindent and the induced long exact sequence

\begin{equation}
\label{longexactsequencemultiplicationbyp}
\cdots \to \pi_i(N)\to \pi_i(N)\to \pi_i(N\otimes_{\integerslocalp}\finitefieldp)\to \pi_{i-1}(N)\to \cdots 
\end{equation}

But by assumption we have $\pi_i(N\otimes_{\integerslocalp}\finitefieldp)=0$ for $i<-d$. In particular, for $i\leq -d-2$ the long exact sequence \eqref{longexactsequencemultiplicationbyp} gives isomorphisms

\begin{equation}
\label{multiplicationbypisomorphismbelowdplusone}
\xymatrix{\pi_{i}(N)\ar[r]^{.p}_{\sim}& \pi_{i}(N)}
\end{equation}

Finally, using the description  $\bbQ=\integerslocalp[p^{-1}]$ and the fact that this localization is flat over $\integerslocalp$, we find

\begin{equation}
\label{multiplicationbypisomorphismbelowdplusone2}
\pi_i(N)\otimes_{\integerslocalp}\bbQ\simeq \pi_i(N\otimes_{\integerslocalp}\bbQ)
\end{equation}

for all $i$ and by assumption we have

\begin{equation}
\label{multiplicationbypisomorphismbelowdplusone3}
\pi_i(N)\otimes_{\integerslocalp}\bbQ\simeq 0
\end{equation}

for $i<-d$. At the same time, for each $i$ we can write

\begin{equation}
\label{multiplicationbypisomorphismbelowdplusone4}
\xymatrix{
\pi_i(N)\otimes_{\integerslocalp}\bbQ\ar[r]^-{\sim} &\colim\, (\pi_i(N)\ar[r]_-{.p}& \pi_i(N)\ar[r]_-{.p} & \cdots)
}
\end{equation}

So that for $i\leq -d-2$, the combination of \eqref{multiplicationbypisomorphismbelowdplusone} and \eqref{multiplicationbypisomorphismbelowdplusone3} gives

\begin{equation}
\label{multiplicationbypisomorphismbelowdplusone5}
\xymatrix{
0\simeq\pi_i(N)\otimes_{\integerslocalp}\bbQ\ar[r]^{\sim} &\colim\, (\pi_i(N)\ar[r]_-{.p}^-{\sim}& \pi_i(N)\ar[r]_-{.p}^-{\sim} & \cdots)\ar[r]^-{\sim}& \pi_i(N),
}
\end{equation}

concluding the proof.

\end{proof}
\end{lemma}

\medskip
\begin{proof}[Proof of \cref{lemma-Bfixfinitecohdimension}]
By \cref{mainlemmafinitecohodimension}, to prove the statement in  \cref{lemma-Bfixfinitecohdimension}, we only need to show that both $\B \Ga{\bbQ}\simeq \B \Frobfixed\times_{\Spec \, \integerslocalp}\Spec\, \bbQ \to \Spec\, \bbQ$ and $\B \padicintegers \simeq \B \Frobfixed\times_{\Spec \, \integerslocalp}\Spec\, \finitefieldp \to \Spec\, \finitefieldp$ are maps of finite cohomological dimension.  For $\B \Ga{}$ over a field of characteristic zero, this is well-known \cite[9.1.5.4]{lurie-sag}.
More generally, we can use the fact that if $H$ is an affine group scheme over a field $k$ and $\rmM$ is an object in the heart of the category of $\structuresheaf(H)$-comodules, then
\begin{equation}
\label{formulachampsaffinescohomologyaffinegroup}
\mathsf{H}^i(\B H, \rmM)\simeq \Map_{\Qcoh(\B H)}(\structuresheaf, \rmM[i])\simeq \Map_{\Stacks{k}/\B H}(\B H, \Kspace(\rmM, i))\simeq \Ext^{i}_{H-\mathrm{rep}^{\heartsuit}}(k, \rmM)
\end{equation}
\noindent where the $\Ext$-groups are computed in the classic abelian category of $\structuresheaf(H)$-comodules. See for instance \cite[Lemme 1.5.1]{MR2244263}.
We can now use this to show that $\B \padicintegers$ is of cohomological dimension 1 over $\finitefieldp$. See \cite[Chapter XIII, \S 1, Corollary]{MR554237} and \cite[\S 3.2, Example]{MR1324577}.
\end{proof}

\begin{remark}
\mylabel{singularcochainscircleuniversalcoefficients}
 For a fixed discrete commutative ring $R$, the $\Einfinity$ (resp. cosimplicial) algebra of singular cochains $\cochains(\circle, R)$ (resp. $\cochainscosimplicial(\circle, R)$) can be defined as the cotensor $R^{\circle}$ in the \icategory  $\CAlg(\Mod_R)$ (resp. $\coSCRings{R}$) which can be explicitely described as the limit of the constant diagram with value $R$, $\lim_{\circle}R$. Because $\circle$ has a finite model as a simplicial set, this limit is finite. Moreover, these two constructions coincide under the functor $\theta^{\mathsf{ccn}}$. It follows that for any map of rings $R\to R'$, the derived base change $\cochainscosimplicial(\circle, R)\otimes_R R'\to \cochainscosimplicial(\circle, R')$ is an equivalence of algebras. In particular, we have equivalences

$$
\cochainscosimplicial(\circle, \integerslocalp)\otimes_{\integerslocalp}\finitefieldp\simeq \cochainscosimplicial(\circle,\finitefieldp)\,\,\,\,\,\text{ and }\,\,\,\,\cochainscosimplicial(\circle, \integerslocalp)\otimes_{\integerslocalp}\bbQ\simeq \cochainscosimplicial(\circle,\bbQ)
$$
\end{remark}

We are now ready to prove our local criterion for affinization:

\begin{lemma}
\mylabel{localcriterionaffinization}
If the two maps (\ref{affinizationunderlyingmap2}) and (\ref{affinizationunderlyingmap3}) are affinizations of $\circle$, respetively over $\bbQ$ and $\finitefieldp$, then the map (\ref{affinizationunderlyingmap}) is an affinization over $\integerslocalp$.
\begin{proof}
The combination of the \cref{singularcochainscircleuniversalcoefficients} and the \cref{flatbasechangecohomologyBFixed} with \cite[Corollaire 2.3.3]{\Bertrandaffine} tells us that the statement in the lemma is equivalent to the following:  to deduce that the map (\ref{affinizationunderlyingmapcomplexes}) is an equivalence, it is enough to check that both maps
\begin{equation}
\label{affinizationunderlyingmapcomplexes2}
\cochainscosimplicial(  \B\Frobfixed, \structuresheaf)\otimes_{\integerslocalp} \bbQ \to  \cochainscosimplicial(\circle, \bbQ)\,\,\, \text{ and }\,\,\,\cochainscosimplicial(  \B\Frobfixed, \structuresheaf)\otimes_{\integerslocalp} \finitefieldp \to  \cochainscosimplicial(\circle, \finitefieldp)
\end{equation}

\noindent are equivalences. Formulated this way, the lemma is immediate from the \cref{basechangeeeee}.

\end{proof}
\end{lemma}

This concludes the proof of \cref{propositionaffinizationunderlying}.

\begin{remark}
\mylabel{KFixnaffinization}
The results of this section, combined with the \cref{KHnareaffine}, show that, more generally, the stacks $\Kspace(\Frobfixed, n)$ are the affinization of $\Kspace(\bbZ, n)$ over $\integerslocalp$. This was left open in \cite{\Bertrandaffine}.
\end{remark}

\subsection{The associated graded of $\Filcircle$}
\mylabel{associatedfilteredcircle}

Our next order of business is to prove \cref{filteredcircletheorem}-(ii) concerning the associated graded $\Filcirclegr \simeq \B \Frobkernel$. As in the previous section,  we reduce the problem to computations over $\bbQ$ and $\mathbb{F}_{p}$.

\begin{lemma}
\mylabel{flatbasechangecohomologyBKernel}
We have canonical equivalences of commutative cosimplicial algebras

\begin{equation}
\label{basechanecohomologyKernFrob}
\cochainscosimplicial(\B\Frobkernel, \structuresheaf)\otimes_{\integerslocalp}\bbQ\simeq \cochainscosimplicial(\B\Frobkernel_{|_{\bbQ}}, \structuresheaf) \,\,\, \text{ and }\,\,\,\cochainscosimplicial(\B\Frobkernel, \structuresheaf)\otimes_{\integerslocalp}\finitefieldp\simeq \cochainscosimplicial(\B\Frobkernel_{|_{\finitefieldp}}, \structuresheaf)
\end{equation}
\begin{proof}
As in the proof of \cref{flatbasechangecohomologyBFixed}, we use base change together with the observation that the exact sequence
$$
\xymatrix{
0\ar[r]& \Frobkernel  \ar[r]& \Wittp \ar[r]^-{\Frob_{p}}& \Wittp \ar[r]& 0
}
$$
exhibits $\Frobkernel$ as a flat group scheme over $\integerslocalp$. From here the proof goes as in \cref{flatbasechangecohomologyBFixed}. 
\end{proof}
\end{lemma}

\medskip

\begin{construction}
\mylabel{maptocohomologyBker}
Consider the composition 

\begin{equation}
\label{Ghostkernelprojectionfirstfactor}
    \xymatrix{
   \Frobkernel\subseteq \Wittp \ar[rr]^{\mathrm{Ghost}} && \prod_{n\in S} \Ga{}\ar[r]^-{\mathrm{proj}_{1\in S}}&\Ga{}
    }
\end{equation}

\noindent and the induced map

\begin{equation}\label{maptoBGafromKernel}u:\B\Frobkernel\to \B\Ga{}\end{equation}

\noindent By definition of $\B\Ga{}$, the map $u$ \eqformula{maptoBGafromKernel} corresponds to an element $u\in \mathsf{H}^1(\cochains(\B\Frobkernel, \structuresheaf))$, $u:\integerslocalp[-1]\to \cochains(\B\Frobkernel, \structuresheaf)$. One can check using explicit formulas for the Ghost map (see \cref{guideWitt}) that the composition \eqformula{Ghostkernelprojectionfirstfactor} is compatible with the $\Gm{}$-actions, where on the l.h.s we have the action of the \cref{construction-GmactionWittVectors} and \cref{remark-underlyingandassociatedgradedFilCircle} and on the r.h.s we have the standard $\Gm{}$-action on $\Ga{}$. In particular, \eqformula{maptoBGafromKernel} is $\Gm{}$-equivariant and the element $u$ is realized as a map in  $\Mod_{\integerslocalp}^{\bbZ-\gr}$; $\integerslocalp[-1]$ is concentrated in weight $1$ by definition. At the same time we consider the canonical element $1:\integerslocalp\to \C^\ast(\B\Frobkernel, \structuresheaf)$ in $\mathsf{H}^0(\C^\ast(\B\Frobkernel, \structuresheaf))$. Because the structure map $\B\Frobkernel\to \Spec(\integerslocalp)$ is $\Gm{}$-equivariant for the trivial action on the target, $1$ also defines a graded map, with $\integerslocalp$ sitting in weight $0$.

The sum of the graded maps $u$ and $1$ give us a map 

\begin{equation}\label{maptoBGafromKernel2}\integerslocalp\oplus \integerslocalp[-1]\to  \cochains(\B\Frobkernel, \structuresheaf) \,\,\, \text{ in } \,\, \Mod_{\integerslocalp}^{\bbZ-\gr} \end{equation}

\noindent This map becomes an equivalence after base change to $\bbQ$ ( \cref{remark-characteristiczerocase}).

\end{construction}

\medskip

\begin{proposition}
\mylabel{underlyingcomplexBKern} The map of graded complexes \eqformula{maptoBGafromKernel2} is an equivalence after tensoring with $\finitefieldp$. By the \cref{flatbasechangecohomologyBKernel} and the \cref{basechangeeeee}, it is also an equivalence over $\integerslocalp$. In particular, the grading on the complex $\cochains(\B\Frobkernel, \structuresheaf)$ coincides with the cohomological grading.
\end{proposition}

We will establish the proof of \cref{underlyingcomplexBKern} by computing the underlying complex of global sections of the structure sheaf of $\B\Frobkernel_{|_{\finitefieldp}}$.

\begin{remark}
\mylabel{formulacohomologyBkern}
Notice that $\Frobkernel_{|_{\finitefieldp}}=\Spec(\cochainscosimplicial(\Frobkernel_{|_{\finitefieldp}}, \structuresheaf)) $ is an affine abelian group scheme over $\finitefieldp$. Therefore, both $\B\Frobkernel$ and $\Frobkernel$ are graded and \textcolor{black}{because these are abelian groups} we have an equivalence of cosimplicial graded Hopf  algebras

\begin{equation}
\label{formulabarBKer}
\cochainscosimplicial(\B\Frobkernel_{|_{\finitefieldp}}, \structuresheaf)\simeq \lim_{\Delta} \, \cochainscosimplicial(\Frobkernel_{|_{\finitefieldp}}, \structuresheaf)^{\otimes_n}
\end{equation}

\end{remark}

In order to understand the underlying complex of $\C^\ast_\Delta(\Frobkernel_{|_{\finitefieldp}}, \structuresheaf)$ we will characterize its category of representations as an Hopf algebra. 




\medskip


\begin{construction}
\mylabel{prostructureBker}
The group scheme $\Frobkernel_{|_{\finitefieldp}}$ has a natural pro-group structure induced from the decompositon $\Wittp\simeq \lim \, \Wittpm$ by defining $\Frobkernel_{|_{\finitefieldp}}$ to be  the kernel of the exact sequence

$$
\xymatrix{
0\ar[r]& \Frobkernelm \ar[r]&\ (\Wittpm)_{|_{\finitefieldp}}\ar[r]^{\Frob_p}& (\Wittpm)_{|_{\finitefieldp}}\ar[r]&0
}
$$
We obtain $\Frobkernel_{|_{\finitefieldp}}\simeq \lim \, \Frobkernelm$ and therefore a colimit of Hopf algebras
\begin{equation}
\label{colimitHopfalgebras1}
    \cochainscosimplicial(\Frobkernel_{|_{\finitefieldp}}, \structuresheaf)\simeq \colim_m \,\cochainscosimplicial(\Frobkernelm, \structuresheaf)
\end{equation}
\end{construction}

\medskip

\begin{definition}
\mylabel{progroupalphap}
  Let us denote by $\alphapm$ the affine scheme over $\finitefieldp$ given by $\Spec(\finitefieldp[T]/(T^{p^{m}}))$. Its functor of points is given by $R\mapsto \{r\in R: r^{p^{m}}=0\}$ classifying $p^m$-roots of zero. This is an abelian affine group scheme under the additive law over $\finitefieldp$.
   \end{definition}

   \begin{lemma}
   \mylabel{cartierdualityformula}
\noindent  For each $m\geq 1$, the Cartier dual of the group scheme $\alpha_{p^{m}}$ is the algebraic group $\Frobkernelm$. In particular, it follows from Cartier duality that we have an equivalence of strict abelian 1-categories
  \begin{equation}
  \label{comodulesandmodulesCartierduality}
  \CoMod^{\heartsuit}_{\cochainscosimplicial(\Frobkernelm,\structuresheaf)}\simeq \Mod^{\heartsuit}_{\finitefieldp[T]/(T^{p^{m}})}
  \end{equation}

  \begin{proof}
  This is \cite[II.10.3, Remark]{MR0213365}\footnote{Formula $L_{m,n}^D=L_{n,m}$} (see also \cite[III \S 4]{MR883960}). See also \cite{MR519769} for the equivalence of categories.
  \end{proof}
   \end{lemma}

\medskip

\begin{remark}
\mylabel{alphapandmup}
 If we forget the group structures, the affine scheme $\alphapm$ is isomorphic to the underlying scheme of $\mupm=\Spec(\finitefieldp[U]/(U^{p^m}-1))$ of $p^{m}$-roots of unity under the change of coordinates $T\mapsto (U-1)$. This induces an equivalence of strict abelian 1-categories

\begin{equation}
\label{comodulesandmodulesCartierduality2}
\Mod^{\heartsuit}_{\finitefieldp[U]/(U^{p^m}-1)}\simeq \Mod^{\heartsuit}_{\finitefieldp[T]/(T^{p^{m}})}
\end{equation}

\noindent Furthermore, we know that $\mupm$ is Cartier dual to the group scheme $\bbZ/p^{m}\bbZ$ \cite[I.2.12, Lemma 2.15]{MR0213365} and this gives us an equivalence of strict abelian 1-categories

\begin{equation}
\label{comodulesandmodulesCartierduality3}
 \Mod^{\heartsuit}_{\finitefieldp[T]/(T^{p^{m}})}\simeq \CoMod^{\heartsuit}_{\cochainscosimplicial(\bbZ/p^{m}\bbZ, \structuresheaf)}
\end{equation}

\noindent Notice moreover that the isomorphisms of schemes $\mupm\simeq \alphapm$ are compatible for different $m$'s under inclusions.
\end{remark}

 \begin{proof}[\textbf{Proof of \cref{underlyingcomplexBKern}}]
 \mylabel{ProofofcohomologyBKern}
 Composing the equivalences \eqformula{comodulesandmodulesCartierduality}, \eqformula{comodulesandmodulesCartierduality2} and \eqformula{comodulesandmodulesCartierduality3}, the Barr-Beck theorem for strict abelian 1-categories gives us an equivalence of \textcolor{black}{coalgebras} \

 \begin{equation}
     \label{equivalencecoalgebraslevelm}
     \cochainscosimplicial(\Frobkernelm, \structuresheaf)\simeq \cochainscosimplicial(\bbZ/p^{m}\bbZ, \structuresheaf)
 \end{equation}

\noindent  Because Cartier duality is functorial, these equivalences are now compatible under the restriction maps and therefore, the equivalence extends to the filtered colimit of coalgebras
 
  \begin{equation}
     \label{equivalencecoalgebras}
     \cochainscosimplicial(\Frobkernel_{|_{\finitefieldp}}, \structuresheaf)\simeq \colim_m \,  \cochainscosimplicial(\Frobkernelm, \structuresheaf)\simeq  \colim_m \, \cochainscosimplicial(\bbZ/p^{m}\bbZ, \structuresheaf)
 \end{equation}

 But now the r.h.s is by definition the coalgebra of the group scheme $\padicintegers$, and we get an equivalence of coalgebras
   \begin{equation}
     \label{equivalencecoalgebras}
     \cochainscosimplicial(\Frobkernel_{|_{\finitefieldp}}, \structuresheaf)\simeq  \cochainscosimplicial(\padicintegers, \structuresheaf)
 \end{equation}
 
 Finally, using the \formula{formulabarBKer}, we find an equivalence of \textcolor{black}{cosimplicial coalgebras}

\begin{equation}
\label{formulabarBKer2}
\cochainscosimplicial(\B\Frobkernel_{|_{\finitefieldp}}, \structuresheaf)\simeq \lim_{\Delta} \, \cochainscosimplicial(\Frobkernel_{|_{\finitefieldp}}, \structuresheaf)^{\otimes_n}\simeq \lim_{\Delta} \, \cochainscosimplicial( \padicintegers, \structuresheaf)^{\otimes_n}\simeq \cochainscosimplicial(\B\padicintegers, \structuresheaf)
\end{equation}
 
Here, the last equivalence uses the fact that $\padicintegers$ is an affine group scheme over $\finitefieldp$ \footnote[1]{It is affine of infinite type,  being a projective limit of finite constant groups schemes which are affine.}.\\

After applying the co-dual Dold-Kan construction (see \cref{notationcosimplicialandsimplicial}) and using the \cref{affinizationsoverQandFp} we deduce an equivalence in $\Mod_{
 \finitefieldp}$

 \begin{equation}
\label{formulabarBKer3}
\cochains(\B\Frobkernel_{|_{\finitefieldp}}, \structuresheaf)\simeq \finitefieldp\oplus \finitefieldp[-1]
\end{equation}

\noindent which one can check is implemented by the underlying map of \eqformula{maptoBGafromKernel2}.

 \end{proof}

\medskip

As in \cref{lemma-Bfixfinitecohdimension} we have

\begin{lemma}
\mylabel{lemma-BKerfinitecohdimension}
The structure map $\B \Frobkernel \to \Spec \, \integerslocalp$ is of finite cohomology dimension. In particular, by \cite[B.15, B.16]{1402.3204}, it verifies base-change against any map $Y\to \Spec \, \integerslocalp$.
\begin{proof}
As in the  proof of the \cref{lemma-Bfixfinitecohdimension}, it is now enough to argue that both $\B\Frobkernel_{|_{\finitefieldp}}\to \Spec\, \finitefieldp$ and $\B\Ga{\bbQ}\to \Spec\, \bbQ$ are of finite cohomological dimension. The case of $\B\Ga{\bbQ}$ has already been discussed. It remains to discuss $\B\Frobkernel_{|_{\finitefieldp}}$. The case of 
$\Frobkernel_{|_{\finitefieldp}}$ follows from the equivalence of coalgebras \eqref{equivalencecoalgebras} which implies that the categories of comodules are equivalent.
\end{proof}
\end{lemma}

\medskip

The combination of \cref{lemma-Bfixfinitecohdimension} and \cref{lemma-BKerfinitecohdimension} also allow us to conclude that

\begin{corollary}
\mylabel{Filteredcirclefinitecohomologicaldimension}
The structure map $\Filcircle\to \Filstack$ is of finite cohomolodical dimension. In particular it verifies the base change formula against any map.
\begin{proof}
The argument boils down to an open-closed complement reduction as in the proof of \cref{lemma-Bfixfinitecohdimension}.
Namely, we argue that to show that $\Filcircle\to \Filstack$ is of finite cohomological dimension, it is enough to have both fibers
$$
\xymatrix{
\B \Frobfixed \ar[d]\ar[r]& \Filcircle\ar[d]& \ar[l]\B \Frobkernel/ \B \Gm{}\ar[d]\\
\Spec\, \integerslocalp \ar[r]^j& \Filstack & \ar[l]_i \B\Gm{}
}
$$

\noindent of finite cohomological dimension (which we now know to be true thanks to \cref{lemma-Bfixfinitecohdimension} and \cref{lemma-BKerfinitecohdimension}).

To confirm the claim, we use the fact that the $t$-structure on $\Filstack$, is determined by its flat atlas $\affineline{}\to \Filstack$ as discussed in the \cref{left-complete-t-structure}-c). Therefore, by flat base change, we are reduced to discuss the same situation over the open-closed pair obtained by pullback

$$
\xymatrix{
\Gm{}\ar[r]^j \ar[d] &\ar[d] \affineline{}& \ar[l]_i 0\ar[d]\\
[\Gm{}/\Gm{}]\simeq \Spec \, \integerslocalp\ar[r]& \Filstack& \ar[l] \B \Gm{}}
$$

\noindent But here, the proof runs exactly as for the pair $\Spec\,\bbQ\subseteq \Spec \, \integerslocalp \supseteq \Spec\, \finitefieldp$ in the proof \cref{lemma-Bfixfinitecohdimension} since the situation amounts to the inversion of the coordinate $t$ in $\affineline{}$ (for the open immesion) and the quotient by $t$ (for the closed).

\end{proof}
\end{corollary}

\medskip

We now discuss both the $\Einfinity$ and cosimplicial algebra structure on $\C^\ast(\B\Frobkernel, \structuresheaf)$. We start by the $\Einfinity$-structure:

\begin{lemma}
\mylabel{uniqueEinfinityalgebra}
Let $\rmM:=\integerslocalp\oplus \integerslocalp[-1]\in \Mod_{\integerslocalp}^{\bbZ-\gr}$ denote the graded complex of the \cref{maptocohomologyBker}. Then, the space of $\Einfinity$-algebra structures on $\rmM$ compatible with the grading is equivalent to the set of  classical commutative graded algebra structures on its cohomology $\mathsf{H}^\ast(\rmM)$. In particular, it is homotopically discrete.
\begin{proof}
The data of an $\Einfinity$-algebra structure on a given object $\rmM\in \Mod_{\integerslocalp}$ is the data of a map of \ioperads from $\Einfinity$ to the \ioperad of endomorphisms of $\rmM$, for which we shall write $\End(\rmM)^\otimes$. In our case,  the object $\rmM:= \integerslocalp\oplus \integerslocalp[-1]$ is endowed with a structure of object in $\Mod_{\integerslocalp}^{\bbZ-\gr}$ of the \cref{maptocohomologyBker}, so in fact, we are interested in the \ioperad of endomorphisms of $\rmM$ in this \icategory $\End_{\gr}(\rmM)^\otimes$. Its space of $n$-ary operations is given by the mapping space $\Map_{\Mod_{\integerslocalp}^{\bbZ-\gr}}(\rmM^{\otimes_n},\rmM)$ where the powers $\rmM^{\otimes_n}$ are taken with respect to the graded tensor product. It follows from the formula for the Day convolution (see \cref{construction-Filteredobjects} and the references to \cite{lurie-K}) that for every $n\geq 1$ the piece of weight 0 in $\rmM^{\otimes_n}$ is $\integerslocalp$ and the piece in weight $1$ is given by $\bigoplus_{i=1}^n\integerslocalp[-1]$. In particular, we get 

$$
\End_{\gr}(\rmM)^{\otimes}(n)\simeq \Map_{\Mod_{\integerslocalp}}(\integerslocalp, \integerslocalp)\times \Map_{\Mod_{\integerslocalp}}(\bigoplus_{i=1}^n\integerslocalp[-1], \integerslocalp[-1])\simeq \bigoplus_{i=0}^n\integerslocalp
$$

\noindent which is a discrete space.  This implies that the space of maps of \ioperads $$\Map_{\inftyoperads}(\Einfinity, \End_{\gr}(\rmM))^{\otimes}$$ is discrete. But more is true: consider the cohomology $\mathsf{H}^\ast(\rmM)$ as a (classical) graded $\integerslocalp$-vector space and $\End_{\gr,\mathsf{cl}}(\mathsf{H}^\ast(\rmM))^{\otimes})$ its classical operad of (graded) endomorphisms. As  $\mathsf{H}^\ast$ is lax monoidal, we get a map 

$$ \Map_{\inftyoperads}(\Einfinity, \End_{\gr}(\rmM)^{\otimes})\to \Map_{\inftyoperads}(\Einfinity,\End_{\gr,\mathsf{cl}}(\mathsf{H}^\ast(\rmM))^{\otimes})$$

The computation above applied to the classical graded version shows that this map is actually an equivalence of spaces. 

\end{proof}
\end{lemma}

\medskip

\begin{remark}
\mylabel{uniqueEinfinityalgebracoconnective}
The argument in the proof of the \cref{uniqueEinfinityalgebra} also shows that there exists a unique commutative $\Einfinity$-algebra structure on the object $\rmM:=\integerslocalp\oplus\integerslocalp[-1]$ seen as an object of $\Mod_{\integerslocalp}^{\leq 0, \gr}$ with the symmetric monoidal structure of \cref{notationcosimplicialandsimplicial}.
\end{remark}

We now discuss the co-simplicial multiplicative structure.

\begin{notation}
\mylabel{gradedversionscosimplicialandcoconnectivealgebras}
Let us consider graded versions $\coSCRings{\integerslocalp}^{\gr}$ and $\CAlg_{\integerslocalp}^{\mathsf{ccn},\gr}:=\CAlg(\Mod_{\integerslocalp}^{\leq 0, \gr})$ of respectively, cosimplicial and $\Einfinity$-algebras and $\theta^{\mathsf{ccn}}:\coSCRings{\integerslocalp}^{\gr}\to \CAlg_{\integerslocalp}^{\mathsf{ccn},\gr}$ denote the graded version of the dual Dold-Kan construction (see \cref{notationcosimplicialandsimplicial}).
\end{notation}

\medskip

\begin{notation}
\mylabel{splitsquarezeronotation}
We denote by $\splitlocalp$ the trivial square zero extension structure on the complex $\cohcirclelocalp$ ($\eta$ of degree 1 cohomological) as an object in $\coSCRings{\integerslocalp}$.  We will use the same notation for its underlying $\Einfinity$-algebra under the functor $\theta$ of \cref{notationcosimplicialandsimplicial}.
\end{notation}

\medskip

\begin{corollary}[of \cref{uniqueEinfinityalgebra} and \cref{uniqueEinfinityalgebracoconnective}]
\mylabel{equivalenceasEinfinity}
 The map \eqformula{maptoBGafromKernel2} extends as an equivalence of graded $\Einfinity$-algebras in $\Mod_{\integerslocalp}^{\leq 0}$ compatible with the augmentations to $\integerslocalp$:
\begin{equation}
\label{equivalencesasEinfinityalgebras}
\splitlocalp\to \cochains(\B\Frobkernel, \structuresheaf)
\end{equation}
In particular, $\cochains(\B\Frobkernel, \structuresheaf)$ is \textcolor{black}{formal} as a graded $\Einfinity$-algebra.
\end{corollary}

\medskip

We now claim that  \eqformula{equivalencesasEinfinityalgebras} actually 
lifts via $\theta^{\mathsf{ccn}}$ to an equivalence of graded cosimplicial commutative algebras. In order to prove this we use the limit formula \eqformula{formulabarBKer} to reduce to an argument about the discrete graded abelian group scheme $\Frobkernel$. We will need some preliminaries:

\medskip

\begin{construction}
\mylabel{constructioncogroupscobar}
We consider cogroup objects in the categories $\C= \coSCRings{}^{\gr}$ and $\C=\CAlg_{}^{\mathsf{ccn},\gr}$ as in the \cref{generalizedBarconstruction}. Namely,

\begin{equation}
\label{Barconstructioncolimitmodel22}
\xymatrix{  \cogroups(\coSCRings{\integerslocalp}^{\gr})\ar@<+1ex>[r]^-{\lim_{\Delta}}& \ar@<+1ex>[l]^-{\mathsf{CoNerve}} (\coSCRings{\integerslocalp}^{\gr})_{./\integerslocalp}& \cogroups(\CAlg_{\integerslocalp}^{\mathsf{ccn},\gr})\ar@<+1ex>[r]^-{\lim_{\Delta}}& \ar@<+1ex>[l]^-{\mathsf{CoNerve}} (\CAlg_{\integerslocalp}^{\mathsf{ccn},\gr})_{./\integerslocalp}}
\end{equation}

\noindent Both adjunctions commute with $\theta^{\mathsf{ccn}}:\coSCRings{\integerslocalp}^{\gr}\to \CAlg^{\mathsf{cnn},\gr}_{\integerslocalp}$ (that $\theta^{\mathsf{ccn}}$ commutes with $\mathsf{CoNerve}$  follows from the fact that it commutes with tensor products)   (\cref{notationcosimplicialandsimplicial}):

\begin{equation}
\label{Barconstructioncolimitmodel33}
\xymatrix{\ar[d]^{\theta^{\mathsf{ccn}}} \cogroups(\coSCRings{\integerslocalp}^{\gr})\ar@<+1ex>[r]^-{\lim_{\Delta}}& \ar@<+1ex>[l]^-{\mathsf{CoNerve}} (\coSCRings{\integerslocalp}^{\gr})_{./\integerslocalp} \ar[d]^{\theta^{\mathsf{ccn}}}\\ \cogroups(\CAlg_{\integerslocalp}^{\mathsf{ccn},\gr})\ar@<+1ex>[r]^-{\lim_{\Delta}}& \ar@<+1ex>[l]^-{\mathsf{CoNerve}} (\CAlg_{\integerslocalp}^{\mathsf{ccn},\gr})_{./\integerslocalp}}
\end{equation}

\end{construction}

\medskip

\medskip

Our main computation is the following:

\begin{theorem}
\mylabel{theorem-splitsquarezero}
The equivalence of graded $\Einfinity$-algebras of the \cref{equivalenceasEinfinity} can be promoted to an equivalence of graded commutative cosimplicial algebras
\begin{equation}
\label{formulacohomologyBkernel}
\splitlocalp\simeq \cochainscosimplicial(\B\Frobkernel, \calO) 
\end{equation}
In particular, as $\B\Frobkernel$ is an affine stack (\cref{filteredcircleisaffine}) we can write 
\begin{equation}
\label{equivalencestacksBKersplit}
\B \Frobkernel \simeq \Spec^{\Delta}(\splitlocalp)
    \end{equation}

\begin{proof}
By construction, the cosimplicial object $\cochainscosimplicial(\Frobkernel, \structuresheaf)^{\bullet}$ in the formula \eqformula{formulabarBKer} defines an object in $\cogroups(\coSCRings{}^{\gr})$ and in the terminology of the \cref{constructioncogroupscobar} the equivalence \eqformula{formulabarBKer} reads as an equivalence in $\coSCRings{./\integerslocalp}^{\gr}$:

$$\cochainscosimplicial(\B \Frobkernel, \structuresheaf)\simeq \lim_{\Delta} [\cochainscosimplicial(\Frobkernel, \structuresheaf)^{\bullet}]$$

\noindent Because the stack $\B\Frobkernel$ is affine (\cref{filteredcircleisaffine}) we know that $
\B \Frobkernel \simeq \Spec^{\Delta}(\cochainscosimplicial(\B \Frobkernel, \structuresheaf))$. Now, using the fact $\cospec$ is a right adjoint, we deduce an equivalence in $\cogroups(\coSCRings{}^{\gr})$

$$\cochainscosimplicial(\Frobkernel, \structuresheaf)^{\bullet}\simeq \mathsf{coNerve}[\cochainscosimplicial(\B \Frobkernel, \structuresheaf)]$$

\noindent where we see $\cochainscosimplicial(\B \Frobkernel, \structuresheaf)$  augmented over $\integerslocalp$ via the atlas. We deduce that

$$
\cochainscosimplicial(\B \Frobkernel, \structuresheaf)\simeq \lim_{\Delta}\circ\,\, \mathsf{coNerve} \,\,\,[\cochainscosimplicial(\B \Frobkernel, \structuresheaf)]
$$

\noindent in $\coSCRings{./\integerslocalp}^{\gr}$. 

The commutativity of the diagram \eqformula{Barconstructioncolimitmodel33} tells us that a similar formula holds for the graded coconnective $\Einfinity$-version. Therefore, the equivalence of the \cref{equivalenceasEinfinity} tells us that a similar formula holds for the graded coconnective $\Einfinity$-version of $\splitlocalp$, ie, an equivalence in $\CAlg_{./\integerslocalp}^{\mathsf{ccn},\gr}$

$$
\theta^\mathsf{ccn}(\splitlocalp) \simeq \lim_{\Delta}\circ\,\, \mathsf{coNerve} \,\,\,[\theta^\mathsf{ccn}(\splitlocalp)]
$$

In particular, we have an equivalence in $\cogroups(\CAlg^{\mathsf{ccn},\gr})$

\begin{equation}
\label{equationdiscreteEinfinity}
\mathsf{coNerve} \,\,\,[\theta^\mathsf{ccn}((\splitlocalp)] \simeq \mathsf{coNerve} \,\,\,[\theta^\mathsf{ccn}((\cochainscosimplicial(\B \Frobkernel, \structuresheaf))]\simeq \theta^\mathsf{ccn}((\cochainscosimplicial(\Frobkernel, \structuresheaf)^{\bullet})
\end{equation}

\medskip

\noindent Finally, we remark that $\Frobkernel$ is a classical affine scheme, so that its global sections are discrete in the co-simplicial direction. In particular, as the functor $\theta^\mathsf{ccn}:\coSCRings{\integerslocalp}^{\gr}\to \CAlg_{\integerslocalp}^{\mathsf{ccn},\gr}$ induces an equivalence on discrete algebras, the equivalence \eqformula{equationdiscreteEinfinity} lifts to an equivalence in $\cogroups(\coSCRings{}^{\gr})$

\begin{equation}
\label{equationdiscretecosimplicial2}
\mathsf{coNerve} \,\,\,[\splitlocalp] \simeq \mathsf{coNerve} \,\,\,[\cochainscosimplicial(\B \Frobkernel, \structuresheaf)]
\end{equation}

\medskip

\noindent and therefore an equivalence in $\coSCRings{./\integerslocalp}^{\gr}$

\begin{equation}
\label{equationdiscretecosimplicial3}
\splitlocalp \simeq \lim_{\Delta}\circ\,\,\mathsf{coNerve} \,\,\,[\splitlocalp] \simeq \lim_{\Delta}\circ\,\,\mathsf{coNerve} \,\,\,[\cochainscosimplicial(\B \Frobkernel, \structuresheaf)]\simeq \cochainscosimplicial(\B \Frobkernel, \structuresheaf)
\end{equation}

\end{proof}
\end{theorem}

\medskip

\begin{corollary}
\mylabel{corollaryequivalenceasHopfalgebras}
The stack $\Spec^\Delta(\splitlocalp)$ admits a unique $\Einfinity$ group structure compatible with the grading and the natural base point
$* \to\Spec^\Delta(\splitlocalp)$ induced from the augmentation
$\Spec^\Delta(\splitlocalp) \to \integerslocalp$ . In particular the equivalence of \eqformula{equivalencestacksBKersplit} is compatible with the group structures.

\begin{proof}
We already know that $\Spec^\Delta(\splitlocalp)$ is of the form
$\B\Frobkernel$. The corollary thus follows from the observation that, for 
any sheaf of abelian groups $A$, over any Grothendieck site, 
the stack $\B A$ carries a unique $\Einfinity$-group structure (up to
an equivalence) compatible with its natural base point $*\to \B A$. 
Indeed, suppose that $BA$ is endowed with a
$\Einfinity$-group structure with unit $* \rightarrow \B A$. 
The stack $\Omega_*\B A\simeq A$ comes equipped with an induced 
$\Einfinity$-group structure 
which is now compatible with the abelian group structure given on $A$.
By the Eckmann-Hilton argument we know that 
these two group structures on $A$ must be equal. We have an
adjunction morphism of
of stacks $\B(\Omega_*\B A) \longrightarrow \B A$, which also compatible with $\Einfinity$-group structures on both sides. This shows that any
$\Einfinity$-group structure on $\B A$ compatible with the
natural base point is canonically equivalent to the standard $\Einfinity$-group
structure induced by the abelian group structure of $A$ itself. 

The corollary is then obtained by considering $A=\Frobkernel$, as a sheaf 
over the big site of affine schemes over the stack $\B\Gm{}$.
\end{proof}
\end{corollary}

\subsection{Quasi-coherent sheaves on the filtered circle}
\mylabel{section-quasicoherentsheavesfilteredcircle}

To conclude this section we describe the categories of representations of $\Frobkernel$ and $\Frobfixed$. More precisely, we understand how the filtered category $\Qcoh(\Filcircle)$ endows a filtration on $\Qcoh(\B \Frobfixed)$ with associated graded given by $\Qcoh(\B \Frobkernel)$ and we conclude that as categories, this filtration in fact is split.

\begin{remark}
\mylabel{filtrationsplitsasE1algebras}
If follows as a consequence of \cref{theorem-splitsquarezero} and \cref{propositionaffinizationunderlying} that the filtration on  $\cochains(\Filcircle,\structuresheaf)$ splits when regarded as an $\Eone$-algebra. That same is not true when regarded as an $\Einfinity$-algebra.
\end{remark}

The main result of this section is the following:

\begin{lemma}
\mylabel{lemmarepGcompactlygenerated}
Let $G$ be a flat affine group scheme over $\integerslocalp$ and assume that:
\medskip
\begin{enumerate}
    \item $\B G$ is of finite cohomological dimension (in the sense of the \cref{mainlemmafinitecohodimension});
    \medskip
    \item Both base changes $G_{|_{\bbQ}}$ and $G_{|_{\finitefieldp}}$ are unipotent groups.
    \medskip
\end{enumerate}
Then the category $\Qcoh(\B G)$ is compactly generated by the tensor unit $\structuresheaf_{\B G}$ and in particular, we have
\begin{equation}
\label{equivalencecompactgenerator}
\Qcoh(\B G)\simeq \Mod_{\cochains(\B G, \structuresheaf)}
\end{equation}
\begin{proof}
The equivalence of categories follows directly from compact generation by \cite[7.1.2.1]{lurie-ha}. To prove the claim that the tensor unit is a compact generator we use  \cite[9.1.5.3 -(1)-(2)]{lurie-sag} to show that $\structuresheaf_{\B G}$ is a compact object. To show that it is a generator we have to show that if $\Map_{\Qcoh(\B G)}(\structuresheaf, E)\simeq 0$ then $E$ vanishes. If we denote by $f:\B G\to \Spec \integerslocalp$ the projection, this is equivalent to show that if $f_\ast (E)$ vanishes then $E$ has to vanish. But via the \cref{basechangeeeee}, $f_\ast (E)$ vanishes if and only if both base changes $f_\ast (E)\otimes_{\integerslocalp}\finitefieldp$ and $f_\ast(E)\otimes_{\integerslocalp}\bbQ$ vanish. But since $f$ is assumed to be of finite cohomological dimension, by \cite[A.1.5]{1402.3204} (see also the \cref{globalsectionsoffilteredarefiltered}-(iii)), we have

$$
f_\ast (E)\otimes_{\integerslocalp}\finitefieldp\simeq (f_p)_\ast I^\ast(E)\simeq \bbR\Hom( \structuresheaf_{\B G_{|_{\finitefieldp}}}, I^\ast(E))
$$

\noindent and 

$$
f_\ast (E)\otimes_{\bbQ}\finitefieldp\simeq (f_\bbQ)_\ast J^\ast (E)\simeq \bbR\Hom(\structuresheaf_{\B G_{|_{\bbQ}}}, J^\ast(E))
$$

\medskip

\noindent (with the notations of \eqref{diagramfinitecohomologicaldimension}). 
But, since we are assuming the base changes to fields to be unipotent groups, this means by definition that the tensor units $\structuresheaf_{\B G_{|_{\finitefieldp}}}$ and $\structuresheaf_{\B G_{|_{\bbQ}}}$ given by the trivial representations, are generators.  This implies now that both $J^\ast (E)$ and $I^\ast(E)$ vanish. But now we can again use the \cref{basechangeeeee} to conclude that E vanishes. Indeed, it sufficies, if $e:\Spec \integerslocalp\to \B G$ denotes the atlas, and since the functor $e^\ast:\Qcoh(\B G)\to \Mod_{\integerslocalp}$ is conservative, it suffices to use the commutativivity of the diagram

$$
\xymatrix{
\Spec(\bbQ)\ar[d]\ar[r]^j & \Spec \integerslocalp \ar[d]^{e}& \ar[d] \ar[l]_i \Spec (\finitefieldp\\
\B G_{|_{\bbQ}}\ar[r]^{J}& \B G& \ar[l]_{I} \B G_{|_{\finitefieldp}}
}
$$

\noindent and the induced commutative diagram of pullbacks

$$
\xymatrix{
\Mod_{\bbQ } &\ar[l]_{j^\ast } \Mod_{\integerslocalp} \ar[r]^{i^\ast }& \Mod_{\finitefieldp}\\
\Qcoh(\B G_{|_{\bbQ}})\ar[u]& \ar[u]^{e^\ast}\ar[l]_{j^\ast }\ar[r]^{I^\ast} \Qcoh(\B G) & \ar[u]\Qcoh( \B G_{|_{\finitefieldp}})
}
$$

\end{proof}
\end{lemma}

\begin{corollary}
\mylabel{representationsFixandKer}
We have canonical equivalences of $\infty$-categories
\begin{equation}
    \label{quasicohBFix}
    \Qcoh(\B\Frobfixed)\simeq \Mod_{\cochains(\B \Frobfixed, \structuresheaf)}
\end{equation}
\begin{equation}
    \label{quasicohBKer}
    \Qcoh(\B\Frobkernel)\simeq \Mod_{\cochains(\B \Frobkernel, \structuresheaf)}
\end{equation}
\noindent Moreover, since by the \cref{filtrationsplitsasE1algebras}, $\cochains(\B \Frobfixed, \structuresheaf)$ and $\cochains(\B \Frobkernel, \structuresheaf)$ are equivalent as $\Eone$-algebras we conclude

\begin{equation}
    \label{quasicohBKerequivalentquasicohBFix}
    \Qcoh(\B \Frobfixed)\simeq \Qcoh(\B \Frobkernel)
\end{equation}

\begin{proof}
The equivalence \eqformula{quasicohBFix} is a consequence of \cref{lemmarepGcompactlygenerated} and \cref{lemma-Bfixfinitecohdimension} for $G=\Frobfixed$. The equivalence \eqformula{quasicohBKer} is a consequence of \cref{lemmarepGcompactlygenerated} and \cref{lemma-BKerfinitecohdimension} for $G=\Frobkernel$.
\end{proof}
\end{corollary}

\section{Representations of the Filtered Circle}
\mylabel{section-RepresentationsFilteredCircle}

Our goal in this section is to describe the category $\Qcoh(\B \Filcircle)$ of representations of our filtered $\Filcircle$. For this purpose turn our attention to the classifying stack $\B \Filcircle$ (\cref{remark-abeliangroupstructurefilteredcircle}); note that by \cref{filteredcircleisaffine} this is an affine stack relatively to $\Filstack$.

\subsection{The cohomology of $\B \Filcircle$}
\mylabel{subsection-cohomologyofBS1Fil}

By the \cref{globalsectionsoffilteredarefiltered}, the cohomology ring $\cochains(\B \Filcircle, \calO)$ admits a natural structure as a filtered $\Einfinity$-algebra.

\medskip

As we shall see, the $\infty$-category of  $\Filcircle$-representations will coincide with that of mixed complexes; but this identification will not preserve the relevant symmetric monoidal structures.  Our first evidence of this is the following 

\begin{proposition}
\mylabel{groupcohomologyfilteredcircle}
There is an equivalence of filtered $\Eone$-algebras 
$$
\cochains(\B \Filcircle, \calO) \simeq \integerslocalp[u]
$$
where $u$ sits in degree 2 (cohomological) and has a split filtration induced by the canonical grading for which $u$ is of weight $-1$. 
\end{proposition}

The following construction is a key ingredient in the proof of \cref{groupcohomologyfilteredcircle}:

\begin{construction}
\mylabel{elementu}
We denote by $\integerslocalp(-1)$, the copy of $\integerslocalp$ seen as a graded module pure of weight $(-1)$. This is an invertible object with respect to the tensor product of graded objects (\cref{construction-Filteredobjects}) and we denote by $\structuresheaf(-1)$ the line bundle given by its pullback to $\Filstack$ along the map $q:\Filstack\to \B \Gm{}$. We construct a morphism  

\begin{equation}
    \label{morphismdegree2weight1}
    u:\structuresheaf(-1)[-2]\to \cochains(\B \Filcircle, \calO)
\end{equation}

\medskip

\noindent in the $\infty$-category $\Qcoh(\Filstack)$, as follows: Consider again the composition $\Wittp\to \Ga{}$ of  \eqformula{Ghostkernelprojectionfirstfactor}. As discussed in the \cref{maptocohomologyBker}, this is a map of groups compatible with the $\Gm{}$-actions. This induces a map of filtered abelian group stacks

\begin{equation}
    \label{mapdegree1}
   \Hgroup:= [(\ker \calG_p)/\Gm{}]\to \Wittp^{\Fil}\underbrace{:=}_{\cref{construction-trivialfamilyfiltered}} [(\Wittp\times\affineline{})/\Gm{}]\to [(\Ga{}\times \affineline{})/\Gm{}]=:\Ga{}(1)
\end{equation}
\medskip

\noindent where $\Ga{}(1)$ can also be described as $\Ga{}(1)=\Spec(\Sym(\structuresheaf(-1))$. Geometrically, $\Ga{}(1)$  corresponds to the filtered group scheme with filtration induced by the geometric action pure weight 1 action of $\Gm{}$ on $\Ga{}$ (see \cref{construction-geometricstackfiltration}). We can now take $\B^2$ to obtain a morphism of stacks $\B \Filcircle\to \B^2 \Ga{}(1)$. We define $u \in \mathsf{H}^2(\B\Filcircle, \calO(1))$ to be the resulting class induced by pullback in cohomology, incarnated as a map \eqformula{morphismdegree2weight1}.

\medskip

Since by definition $\cochains(\B\Filcircle,\structuresheaf)\simeq \pi_\ast(\structuresheaf)$ where $\pi:\B\Filcircle\to \Filstack$, by adjunction the map $u$ can also read as a map

\begin{equation}
\label{uinS1rep}
\pi^\ast(\structuresheaf(-1))[-2]\to \structuresheaf
\end{equation}

\noindent in $\Qcoh(\B\Filcircle)$.\\ 

\medskip

As $\cochains(\B \Filcircle, \structuresheaf)$ may be viewed as a filtered $\Einfinity$-algebra, and in particular a filtered $\Eone$-algebra, there is a morphism of filtered $\Eone$-algebras induced by the universal property

\begin{equation}
    \label{cohomologyBS1filtered}
\integerslocalp[u] \to \cochains(\B \Filcircle, \calO) \,\,\,\, \text{ in }\,\,\, \Alg_{\Eone}(\Qcoh([\affineline{\integerslocalp}/\Gm{\integerslocalp}]))
\end{equation}

\medskip

\noindent where $\integerslocalp[u]$ is the free $\Eone$-algebra on $\structuresheaf(-1)[-2]$. 
\end{construction}

\medskip

\begin{remark}
\mylabel{Zusplits}
Notice in particular that since the filtration splits for $\integerslocalp[u]$ as a filtered $\Eone$-algebra (see \cref{construction-filteredassociatedstack}), one has an equivalence of the underlying $\Eone$-algebras

\begin{equation}
    \label{splitfiltrationu}
    (\integerslocalp[u])^{\mathsf{und}} \simeq (\integerslocalp[u])^{\mathsf{ass}-\gr}
\end{equation}

\medskip

\noindent where on the r.h.s we forget the grading.

\medskip
\end{remark}

\medskip

\begin{proof}[Proof of \cref{groupcohomologyfilteredcircle}:] We now show that the map of $\Eone$-algebras constructed in \eqformula{cohomologyBS1filtered} is an equivalence. 
 As in the \cref{lemma-Gpisfpqccover} (see also \cref{mapoffiltratinsisoiffassgradandunderiso}), it will be enough to show that \eqformula{cohomologyBS1filtered} is an equivalence after base-change to the field-valued points

$$(0, \bbQ), (1, \bbQ), (0, \finitefieldp), (1, \finitefieldp).$$ 

\medskip
\noindent It will then be enough to test the underlying maps of complexes, forgetting the grading and the algebra structures.\\

\noindent Let us first deal with the associated-graded, ie, the base change of \eqformula{cohomologyBS1filtered} along $0$. As a consequence of the base change formula \eqref{beckchevalley2} in the \cref{basechangeBfilteredcircle} we obtain

\begin{equation}
\label{cohomologyBS1filteredassgraded}
    \xymatrix{
    (\integerslocalp[u])^{\mathsf{ass}-\gr}\ar[r]\ar[dr]&\cochains(\B \Filcircle, \calO)^{\mathsf{ass}-\gr}\ar[d]_{\sim}^{\eqref{beckchevalley2}}\\
    &\cochains(\B^2 \Frobkernel, \calO)
    }\,\,\,\text{ in }\,\,\,\Alg_{\Eone}(\Qcoh([\B\Gm{\integerslocalp}]))
\end{equation}

\medskip

\noindent  First we test the base change of \eqref{cohomologyBS1filteredassgraded} to $\bbQ$. In this case we get

\begin{equation}
  \label{cohomologyBS1filteredrational2}
    \xymatrix{
     (\integerslocalp[u])^{\mathsf{ass}-\gr}\otimes_{\integerslocalp} \bbQ \ar[rr]\ar[ddd]^{\sim}_{\integerslocalp[u]\,\, \text{free filtered } \Eone-algebra}&&\cochains(\B^2 \Frobkernel, \calO)\otimes_{\integerslocalp} \bbQ\ar[d]_{\sim}^{\text{ as in \cref{flatbasechangecohomologyBKernel} using diagonal res.} }\\
    &&\cochains(\B^2 \Frobkernel_\bbQ, \calO)\ar[d]_{\sim}^{\text{\cref{remark-characteristiczerocase}}}\\
   &&\cochains(\B^2 \Ga{\bbQ}, \calO)\ar[d]_{\sim}^{\text{\cref{affinizationsoverQandFp}}}\\
    \bbQ[u]^{\mathrm{und}}\ar[rr] && \cochains(\Kspace(\bbZ,2),\bbQ) \\
     }
\end{equation}

\noindent Finally, the bottom map obtained by clockwise composition around the diagram \eqref{cohomologyBS1filteredrational2}, is an equivalence: indeed, after passing to cohomology groups we get the isomorphism of commutative graded algebras 

\begin{equation}
    \label{isomorphismincohomologyBcircleQ}
    \bbQ[u]\to \mathsf{H}^\ast(\Kspace(\bbZ,2),\bbQ)
\end{equation}

\medskip

\noindent Now we test the base change of \eqref{cohomologyBS1filteredassgraded} to $\finitefieldp$. We get

\begin{equation}
  \label{point0Fp}
    \xymatrix{
     (\integerslocalp[u])^{\mathsf{ass}-\gr}\otimes_{\integerslocalp} \finitefieldp \ar[rr]\ar[d]^{\sim}_{\integerslocalp[u]\,\, \text{free filtered } \Eone-algebra}&&\cochains(\B^2 \Frobkernel, \calO)\otimes_{\integerslocalp} \finitefieldp\ar[d]_{\sim}^{\text{ as in \cref{flatbasechangecohomologyBKernel} using diagonal res.} }\\
    \finitefieldp[u]^{\mathrm{und}}\ar[rr] && \cochains(\B^2 \Frobkernel_{|_{\finitefieldp}}, \calO)\\
     }
\end{equation}

\medskip

\noindent and we have to explain why the bottom map obtained by clockwise composition in \eqref{point0Fp} is an equivalence. Thanks to  \eqformula{formulabarBKer2}, we know that  $\cochains( \B\Frobkernel_{|_{\finitefieldp}}, \calO)$ and  $\cochains(\B \padicintegers,\structuresheaf )$ are equivalent as coalgebras, and in particular, as complexes. Moreover, as both $\B\Frobkernel_{|_{\finitefieldp}}$ and $\B \padicintegers$ are affine stacks over $\finitefieldp$ (resp. \cref{filteredcircleisaffine} and \cref{affinizationsoverQandFp}), we deduce that 

\begin{equation}
\cochains( \B\Frobkernel_{|_{\finitefieldp}}^{\times^n}, \calO)\simeq \cochains( \B\Frobkernel_{|_{\finitefieldp}}, \calO)^{\otimes^n}\,\,\,\,\,\,\,\,\, \text{ and }\,\,\,\,\,\cochains(\B \padicintegers^{\times^n},\structuresheaf )\simeq \cochains(\B \padicintegers,\structuresheaf )^{\otimes^n}
\end{equation}

\medskip

\noindent But then, by passing to the limit, we obtain

\begin{equation}
\cochains(\B^2 \Frobkernel_{|_{\finitefieldp}}, \calO)=\cochains(\Kspace( \Frobkernel_{|_{\finitefieldp}},2), \calO) \simeq \cochains(\Kspace(\padicintegers, 2),\structuresheaf )
\end{equation}

\medskip

\noindent Finally, we get

\begin{equation}
    \label{cohomologyBS1filteredfinite1}
(\finitefieldp[u])^{\mathsf{und}}=(\finitefieldp[u])^{\mathsf{ass}-\gr} \to  \cochains(\Kspace( \Frobkernel_{|_{\finitefieldp}},2), \calO) \simeq \cochains(\Kspace(\padicintegers, 2),\structuresheaf ) \underbrace{\simeq}_{\text{\cref{affinizationsoverQandFp}}} \cochains(\Kspace(\bbZ, 2), \finitefieldp) 
\end{equation}

\noindent and it is now clear that \eqformula{cohomologyBS1filteredfinite1} produces an isomorphism after passing to cohomology 

\begin{equation}
\label{isomorphismincohomologyBcircleFp}
    \finitefieldp[u] \simeq \mathsf{H}^{*}(\Kspace(\bbZ,2), \finitefieldp)
\end{equation}

\vspace{1cm}

\noindent Let us now deal with the underlying objects, ie, the base change of \eqformula{cohomologyBS1filtered} along $1$. The Beck-Chevalley transformation of \eqref{beckchevalley2foropen} in the \cref{basechangeBfilteredcircle} a priori only gives the possibility of a composition

\begin{equation}
    \label{cohomologyBS1filteredunder}
    \xymatrix{
(\integerslocalp[u])^{\mathsf{und}} \ar[r]\ar[dr] &\ar[d]^{\text{\eqref{beckchevalley2foropen}}}\cochains(\B\Filcircle, \structuresheaf)^\mathrm{und}\\
&\cochains((\B\Filcircle)^{\mathrm{und}}, \calO)\simeq \cochains(\B^2 \Frobfixed, \calO)} \,\,\,\, \text{ in }\,\,\, \Alg_{\Eone}(\Mod_{\integerslocalp})
\end{equation}

\medskip

But, since the structure sheaf of $\B\Filcircle$ is in the heart of the t-structure (see \cref{left-complete-t-structure}), the map \eqref{beckchevalley2foropen} is in fact an equivalence.\\

\medskip

\noindent As before, first we test the base change of \eqref{cohomologyBS1filteredunder} to $\bbQ$. We get

\begin{equation}
    \xymatrix{
     (\integerslocalp[u])^{\mathsf{und}}\otimes_{\integerslocalp} \bbQ \ar[rr]\ar[ddd]^{\sim}_{\integerslocalp[u]\,\, \text{free filtered } \Eone-algebra}&&\cochains(\B^2 \Frobfixed, \calO)\otimes_{\integerslocalp} \bbQ\ar[d]_{\sim}^{\text{ as in \cref{flatbasechangecohomologyBFixed} using res. of \eqref{diagonalflatatlasbyaffines}} }\\
    &&\cochains(\B^2 \Frobfixed_\bbQ, \calO)\ar[d]_{\sim}^{\text{\cref{remark-characteristiczerocase}}}\\
   &&\cochains(\B^2 \Ga{\bbQ}, \calO)\ar[d]_{\sim}^{\text{\cref{affinizationsoverQandFp}}}\\
    \bbQ[u]^{\mathrm{und}}\ar[rr] && \cochains(\Kspace(\bbZ,2),\bbQ) \\
     }
\end{equation}

\noindent After passing to cohomology we recover the isomorphism in \eqref{isomorphismincohomologyBcircleQ}.

\medskip

\noindent Finally, concerning the base change to $\finitefieldp$, we get:

\begin{equation}
    \label{cohomologyBS1filteredfinite2}
(\finitefieldp[u])^{\mathsf{und}}\to \cochains(\B^2 \Frobfixed_{|_{\finitefieldp}}, \calO) \underbrace{\simeq}_{\eqref{formulabasechangetoFp}} \cochains(\Kspace(\padicintegers,2),\structuresheaf) \underbrace{\simeq}_{\text{\cref{affinizationsoverQandFp}}} \cochains(\Kspace(\bbZ, 2), \finitefieldp) 
\end{equation} 

\noindent which induces recovers the isomorphism \eqref{isomorphismincohomologyBcircleFp} after passing to cohomology.\\

\medskip


\end{proof}

\subsection{Representations of  $\Filcircle$}
\mylabel{subsection-representationsofS1Fil}

Before we proceed with the ramifications of \cref{groupcohomologyfilteredcircle}, we recall in more detail the notion of mixed complexes. 

\begin{definition}
\mylabel{strictlambdamodules}
Let $\Lambda := \integerslocalp[\epsilon]=\mathsf{H}_{*}(\circle, \integerslocalp)$ be the graded associative dg-algebra freely generated with $\epsilon$ in degree 1 (homological) and $\epsilon^{2}=0$. The dga $\Lambda$ is graded, with $\epsilon$ sitting in weight $1$.  Alternatively, $\Lambda$ can be characterized as the dual graded Hopf algebra of $\mathsf{H}^\ast(\B \Frobkernel, \structuresheaf)$ \footnote{Notice that $\mathsf{H}^\ast(\B \Frobkernel, \structuresheaf)\simeq \cochains(\B \Frobkernel, \structuresheaf)$  by the formality in \cref{equivalenceasEinfinity} of graded $\Einfinity$-algebras}.

By giving $\Lambda$ the split filtration corresponding to its grading, we may consider $\Lambda$ as a filtered $\Eone$-algebra. 
\end{definition}

\begin{construction}
\mylabel{strictlambdamodules-symmetricmonoidalstructure}
We let $\Mod_{\Lambda}$ denote the $\infty$-category associated to \emph{strict} graded $\Lambda$-modules: concretely, this is obtained by localizing the ordinary symmetric monoidal category of (strict) graded $\Lambda$-modules with respect to quasi-isomorphisms. This carries a symmetric monoidal structure $\otimes_k$ due to the fact that $\Lambda$ is a strict graded dg-Hopf algebra. As a category  of modules over a filtered algebra it acquires the structure of a \emph{filtered} $\infty$-category, which we will denote as  $(\Mod_{\Lambda})_{\Fil}$.

Together with the symmetric monoidal structures, it becomes a sheaf of $\Einfinity$-categories over $\Filstack$. 
\end{construction}


It turns out that the induced filtration on the categories of representations 
$\Qcoh(\B\Filcircle)$ is split, when no tensor structure is involved (see \cref{splitfiltration2}). In proving the HKR theorem we will need a finer result which describes
the symmetric monoidal structures on
$\Qcoh(\B\Filcircleund)$ and $\Qcoh(\B\Filcirclegr)$. The following is the main result of this section:

\medskip


\medskip

\begin{proposition}\mylabel{proprepmonoidal}

\begin{enumerate}

    \item The pullback along the map of $\integerslocalp$-stacks
    $u:\circle \longrightarrow \Filcircleund$ of \eqformula{affinizationunderlyingmap}
    induces a symmetric monoidal equivalence

    \begin{equation}
    \label{equivalencerepresentationsfilteredcircleunderlying}
    \xymatrix{
    u^\ast: \Qcoh(\B\Filcircleund)\ar[r]_-{\sim}& \Qcoh(\B \circle) 
    }
    \end{equation}
    
     \medskip

    \item There exist a natural 
    symmetric monoidal equivalence
    
    \medskip
    
    \begin{equation}
    \label{equivalencerepresentationsfilteredcirclegraded}
    \Mod_\Lambda^{\otimes} \simeq \Qcoh(\B\Filcirclegr)^{\otimes}
    \end{equation}
    
     \medskip

  \noindent  compatible with the graduations on both sides. Here the symmetric monoidal monoidal structure on the l.h.s is the one of \cref{strictlambdamodules-symmetricmonoidalstructure}. In particular, an action of $\Filcirclegr$ is given by a strict square zero differential.

\end{enumerate}

\end{proposition}

\medskip

\begin{proof}[Proof of \cref{proprepmonoidal}:]
We start with the proof of (i).  To show that  \eqref{equivalencerepresentationsfilteredcircleunderlying} is an equivalence, we use presentation of both $\B\Filcircleund$ and $\B \circle$ as the geometric realizations of simplicial diagrams 

$$
[n]\mapsto \Filcircleundntimes\simeq  \B \Frobfixed^{\times_n}
$$
\noindent respectively,

$$
[n]\mapsto \circletimesn
$$

\medskip

Since \eqref{equivalencerepresentationsfilteredcircleunderlying} is induced by a map of groups $\circle\to \B\Frobfixed$, to show that it is fully faithful, it is enough to check that levelwise the induced pullback functors

\begin{equation}
 \label{covidmillion}
u_n^\ast:  \Qcoh(\B \Frobfixed^{\times_n})\to \Qcoh(\circletimesn)
\end{equation}

\noindent are fully faithful.

Using the equivalence \eqref{quasicohBFix}, combined with the \cref{propositionaffinizationunderlying}, we therefore reduced to show that the pullback functors

\begin{equation}
 \label{covid2million}
u_n^\ast:  \Qcoh(\B \Frobfixed^{\times_n})\underbrace{\simeq}_{\eqref{quasicohBFix}} \Mod_{\cochains(\B \Frobfixed, \structuresheaf)^{\otimes_n}}\underbrace{\simeq}_{\ref{propositionaffinizationunderlying}}
\Mod_{\cochains(\circle,\integerslocalp)^{\otimes_n}}\to \Qcoh(\circletimesn)
\end{equation}

\medskip

\noindent are fully faithful for every $n$. But this follows from the fact that $k$ is a compact object in $\Qcoh(\circletimesn)\simeq \Fun(\circletimesn, \Mod_{k})$ (under this equivalence $k$ denotes with constant diagram with value $k$) since the spaces $\circletimesn$ are finite CW-complexes.


To check this, we observe that the functor $u_n^\ast$ admits a right adjoint
$$
(u_n)_\ast: \Fun(\circletimesn, \Mod_k)\to \Mod_{\cochains(\circletimesn, \integerslocalp)}
$$
\noindent given by taking the homotopy limit of diagrams. The fact that $k$ is compact says that this functor commutes with all colimits. So does $u_n^\ast$. Therefore, in order to show that the unit of this adjuntion 
\begin{equation}
\label{covid21}
N\to (u_n)_\ast u_n^\ast(N)
\end{equation}
\noindent is an equivalence for every $N\in \Mod_{\cochains(\circletimesn, \integerslocalp)}$ we can use the fact that the category $\Mod_{\cochains(\circletimesn, \integerslocalp)}$ is generated under colimits by $\cochains(\circletimesn, \integerslocalp)$ and that $u_n^\ast$ preserves colimits. Using these two facts, it is enough to check that the unit map \eqref{covid21} is an equivalence when $N=\cochains(\circletimesn, \integerslocalp)$. 
\begin{equation}
\label{covid212}
\cochains(\circletimesn, \integerslocalp)\to (u_n)_\ast u_n^\ast(\cochains(\circletimesn, \integerslocalp))
\end{equation}
But finally, since $u_n^\ast$ is monoidal, we get that $u_n^\ast(\cochains(\circletimesn, \integerslocalp))$ is the constant diagram with values $k=\integerslocalp$ and therefore we get that \eqref{covid212} is an equivalence. This concludes the proof that the $u_n^\ast$ are fully faithful.
\medskip

To show that \eqref{equivalencerepresentationsfilteredcircleunderlying} is essentially surjective we observe that since $u^\ast$ is a limit of fully faithful functors on each degree $n$

$$
\lim_{\Delta} \Qcoh(\B \Frobfixed^{\times_n})\to \lim_{\Delta} \Qcoh(\circletimesn)
$$

\noindent an object $(E_n)$ in the r.h.s is in the essential image of the limit functor if and only if levelwise each $E_n$ is in the essential image of the functor $u_n^\ast$. For $n=0$, $u_0^\ast$ is an equivalence, and for $n>0$, each $E_n$ is obtained from $E_0$ by pullback from $\Spec (\integerslocalp)$.

\medskip

\vspace{1cm}

We now address the proof of (ii). The proof is essentially the same as 
for (1), except that we have to produce a symmetric monoidal $\infty$-functor
$$\Mod_\Lambda \longrightarrow \Qcoh(\B^2 \Frobkernel).$$
\noindent where $\Lambda$ is seen as a graded algebra as in \cref{strictlambdamodules}.
Such a functor is obtained as follows. 
We can consider  $\mathsf{H}^*(\B\Frobkernel,\mathcal{O})$ as a strict
commutative dg-Hopf algebra. It is a consequence of \cref{corollaryequivalenceasHopfalgebras} that $\mathsf{H}^*(\B\Frobkernel,\mathcal{O})$ and $\mathsf{H}^\ast(\circle, \integerslocalp)$ are isomorphic as graded Hopf algebras. 

We consider a strict symmetric monoidal category of comodules in chain complexes $\CoMod_{\mathsf{H}^*(\B\Frobkernel,\mathcal{O})}^{\mathsf{strict}}$ defined as the limit internal to strict symmetric monoidal 1-categories

\begin{equation}
\label{strictcomodulesdefinition}
\CoMod_{\mathsf{H}^*(\B\Frobkernel,\mathcal{O})}^{\mathsf{strict}}:= \lim_{\Delta}\Mod_{\mathsf{H}^*(\B\Frobkernel,\mathcal{O})^{\otimes_n}}^{\mathsf{strict}}
\end{equation}

\noindent It follows from the \cref{strictlambdamodules} that $\Lambda$, is the dual graded Hopf algebra of $\mathsf{H}^*(\B\Frobkernel,\mathcal{O})$. It follows in this case that we have an equivalence of strict symmetric monoidal 1-categories of comodules (resp. modules) on chain complexes

\begin{equation}
\label{strictcomodulesdefinition}
\CoMod_{\mathsf{H}^*(\B\Frobkernel,\mathcal{O})}^{\mathsf{strict}, \otimes}\simeq \Mod_{\Lambda}^{\mathsf{strict}, \otimes}
\end{equation}

\noindent given by the identity on objects. It can now be localized on both sides along quasi-isomorphism to produce an equivalence of symmetric monoidal $\infty$-categories.

\begin{equation}
\label{strictcomodulesdefinition}
\CoMod_{\mathsf{H}^*(\B\Frobkernel,\mathcal{O})}^{\mathsf{strict}, \otimes}[q.iso^{-1}]\simeq \Mod_{\Lambda}^{\mathsf{strict}, \otimes}[q.iso^{-1}] \underset{\cref{strictlambdamodules-symmetricmonoidalstructure}}{=:} \Mod_{\Lambda}^\otimes
\end{equation}

\noindent We now claim that we are able to construct a symmetric monoidal functor

\begin{equation}
\label{strictcomodulestoinfinitycomodulesfunctor}
\CoMod_{\mathsf{H}^*(\B\Frobkernel,\mathcal{O})}^{\mathsf{strict}, \otimes}[q.iso^{-1}]\to \Qcoh(\B^2\Frobkernel)
\end{equation}

\noindent Indeed, after inverting quasi-isomorphisms we get a naturally defined symmetric monoidal $\infty$-functor

\begin{equation}
\xymatrix{
\Mod_\Lambda^{\otimes}\simeq (\lim_{n\in \Delta}\Mod^{\mathsf{strict}}_{\mathsf{H}^*(\B\Frobkernel,\mathcal{O})^{\otimes n}})[q.iso^{-1}] \ar[r]&
\lim_{n\in \Delta}
(\Mod^{\mathsf{strict}}_{\mathsf{H}^*(\B\Frobkernel,\mathcal{O})^{\otimes_n}}[q.iso^{-1}])\ar[d]^{\sim}\\
&\lim_{n\in \Delta}
\Mod_{\mathsf{H}^*(\B\Frobkernel,\mathcal{O})^{\otimes_n}}
}
\end{equation}

\noindent  On the other hand, the symmetric monoidal $\infty$-category $\Qcoh(\B^2\Frobkernel)$ is obtained
as a limit of symmetric monoidal \icategories (by descent from $\B\Frobkernel$ to $\B^2 \Frobkernel$)

\begin{equation}
    \label{descriptionQcohBBker}
    \Qcoh(\B^2 \Frobkernel) \simeq \lim_{n\in \Delta}
\Qcoh( \B\Frobkernel^{\times_n})
\end{equation}

\medskip

\noindent We are therefore reduced to the construction of symmetric monoidal functors

\begin{equation}
    \label{stricttoinfinity}
    \Mod_{\mathsf{H}^*(\B\Frobkernel,\mathcal{O})^n}\to \Qcoh(\B\Frobkernel^{\times_n})
\end{equation}

\noindent compatible with the transition maps in $\Delta$. But this follows again because of formality (\cref{equivalenceasEinfinity}) as $$\mathsf{H}^*(\B\Frobkernel,\mathcal{O})^n\simeq \cochains(\B\Frobkernel,\mathcal{O})^n$$
\noindent and there is always a symmetric monoidal functor from  $$\Mod_{\cochains(\B\Frobkernel,\mathcal{O})^n}\to \Qcoh(\B\Frobkernel^{\times_n})$$
\noindent from modules over global sections to quasi-coherent sheaves, which in fact we showed is an equivalence in \eqref{quasicohBKer}.\\

We obtain in this manner the desired symmetric monoidal \ifunctor \eqformula{strictcomodulestoinfinitycomodulesfunctor}. The fact that \eqformula{strictcomodulestoinfinitycomodulesfunctor} is an equivalence on the underlying categories is a consequence of \cref{groupcohomologyfilteredcircle} together with the following \cref{newlemmacovid}, which implies that the functor sends $\Lambda$ to the compact generator of $\Qcoh(\B^2 \Frobkernel)$. By construction this equivalence is clearly compatible with the action of 
$\mathbb{G}_m$ on both sides.

\end{proof}

\begin{lemma}
\mylabel{newlemmacovid}
Let $B$ be a strict (bi)-commutative dg Hopf algebra which is strictly dualizable as a $k$-module. Set  $G = \Spec(B)$ and let 
$$
\Psi: \CoMod_{B}^{\mathsf{strict}, \otimes}[q.iso^{-1}]\to \Qcoh(\B G) \simeq  \lim_{\Mod_{B^{ \otimes \bullet}}}
$$
be the functor as constructed above, between strict and homotopically coherent $B$-comodules. Then $\Psi(B)$ is a compact generator for the right hand side above. 
\end{lemma}

\begin{proof}
We show that the functor 

$$
F \mapsto \Map_{\Qcoh(\B G)}(\Psi(B),-),
$$
corepresented by $\Psi(B)$ coincides with the forgetful functor 
$$
\lim\, \Mod_{B^{n}} \to \Mod_{k};
$$
this itself is just the projection to the zeroth component in the cosimplicial diagram of stable \icategories above. For this we view $\Qcoh(\B G) $ as a cosimplicial \icategory. By construction, $\Psi(B)$ is represented, in each cosimplicial degree, by the object $B^{\wedge} \otimes B^{\otimes  n}$, where $B^{\wedge}$ is the $k$-linear dual of $B$. Moreover, for any $(F_n)\in \lim_{\Mod_{B^{n}}}$ there will be equivalences

$$
\Map_{\Mod_B^{\otimes  n}}(B^{\wedge} \otimes B^{\otimes  n}, F^n) \simeq  \Map_{\Mod_k}(B^{\wedge}, F^n) \simeq \Map_{\Mod_k}(k, B \otimes_{k} F^n) \simeq B \otimes_k F^n
$$
The limit of all the $B \otimes_{k} F^n$ is the standard cofree 
resolution of $F$ as $B$-comodule. Therefore, this colimit is naturally equivalent
to the underlying $k$-module of $F^{\bullet}$, ie. $F^{0}$. Moreover, as this is just the data of the projection of $F$ to the zeroth level of the cosimplicial diagram, this computes the left-adjoint forgetful functor 
$$
\lim\,\Mod_{B^\bullet} \to \Mod_{B^0} \simeq \Mod_{k}.
$$
The fact that this functor is conservative and commutes with colimits implies
that $\Psi(B)$ is a compact generator. 
\end{proof}

\medskip

\begin{remark}
\mylabel{splitfiltration2}
Using arguments similar to the ones of \cref{proprepmonoidal} and the main computation of \cref{groupcohomologyfilteredcircle} one can show that there exists an equivalence of filtered $\infty$-categories
$$
(\Mod_{\Lambda})_{\Fil} \simeq \Qcoh(\B \Filcircle)
$$
between filtered mixed complexes (\cref{strictlambdamodules-symmetricmonoidalstructure}) and $\Filcircle$-representations. As
we do not use this equivalence it in the sequel we do not provide its proof. This equivalence means that the filtration on the underlying $\infty$-category $\Qcoh(\B \Filcircle)$ canonically splits, since the filtration on filtered mixed complexes comes from a grading. A key point however, is that, although this is an equivalence of $\infty$-categories; the symmetric monoidal structure is not preserved. Hence this is not a multiplicative splitting. Over characteristic zero this can be promoted to a symmetric monoidal equivalence, for more see \cref{background} and \cite{MR2862069}. 
\end{remark}

\medskip

\section{The HKR Theorem}
\mylabel{HKRsection}

We are now ready to prove the results listed in \cref{corollaryHKRpositive}
for a general simplicial commutative 
$k$-algebra $A\in \SCRings{k}$ (and more generally for derived
schemes or stacks by gluing) where $k$ is a fixed commutative 
$\integerslocalp$-algebra. 

The filtered circle $\Filcircle$ sits over 
$\integerslocalp$ and we pull-it back over 
$\Spec\, k$ to make it a filtered group stack over $k$. We will continue to denote it by $\Filcircle$. \\

\subsection{The filtered loop stack}
\mylabel{filteredloopstack}

\begin{definition}
\mylabel{filteredloopspace}
Let $X=\Spec\, A$ be an affine derived scheme over $k$. The \emph{filtered loop space} of $X$ (over $k$) is
defined as the relative mapping stack
$$\filteredloops X:=\Map_{/\Filstack}(
\Filcircle, X\times \Filstack).$$
This is a derived scheme over $\Filstack$
equipped with a canonical action of the group
stack $\Filcircle$. As in \cref{filteredglobalsections} we denote by $$\Ofil(\filteredloops X):= p_\ast \structuresheaf_{\filteredloops X}$$
\noindent the object in $\Qcoh(\Filstack)$.
\end{definition}

\medskip

\begin{proposition}
\mylabel{filteredloopsrepresentablyrelative}
Let $X=\Spec\, A$ be an affine derived scheme over $k$ and $\filteredloops X$ its filtered loop space. Then, the derived stack $\filteredloops X$ is representable by an affine derived scheme relative to $\Filstack$.
\begin{proof}
We start noticing that the truncation of $\filteredloops X$ is simply the truncation of $X\times \Filstack$, 
and thus the truncation is a relativitely affine scheme. In order
to prove that $\filteredloops X$ is relativitely affine
it remains to check that 
$\filteredloops X$ has a global cotangent complex, 
admits an obstruction theory and is nil-complete 
(see \cite{MR2394633} Appendix C). But this follows easily from the fact that it is a mapping derived stack
from $\B \Hgroup$, the classifying stack of a filtered group scheme.
\end{proof}
\end{proposition}

\medskip

\begin{theorem}
\mylabel{thmhkr-stackversion}
Let $X=\Spec\, A$ be an affine derived scheme over $k$. Then, the filtered loop stack exhibits a filtration on the usual derived loop stack $\loops X$ whose associated graded in the shifted tangent stack $\mathrm{T}[-1]X$. More precisely, we have:

\medskip

\begin{enumerate}[(1)]
\item The composition with the natural map $\eqformula{affinizationunderlyingmap}:\circle \longrightarrow \Filcircleund_k=\B \Frobfixed_k$. 
    induces an isomorphism of derived schemes:
    \begin{equation}
\label{HHformulaisomappingspaces}
(\filteredloops X)^{\mathrm{u}}:=\Map_k(\Filcircleund_k,X) \longrightarrow \Map_k(\circle,X)=:\loops X
\end{equation}

\medskip

\item There is a natural equivalence of graded derived schemes

\begin{equation}
    \label{equationFiltcircleSymL}
(\filteredloops X)^{\mathrm{gr}}:=\Map_k(\Filcirclegr_k,X) \simeq \Spec\, \Sym^{\Delta}_A(\cotangent_{A/k}[1])=:\mathrm{T}[-1]X
 \end{equation}
 
 \noindent where the action of $\Gm{}$  on $\mathrm{T}[-1]X$ is the geometric action of weight (-1), ie, $\cotangent_{A/k}[1]$ is in weight 1.
 \end{enumerate}
 \begin{proof}
  We first analyse the statement (1). For this purpose we show that the map \eqformula{HHformulaisomappingspaces} is an equivalence as functor of points, ie, that for every simplicial commutative algebra $B$ the induced map

    \begin{equation}
\label{HHformulaisomappingspaces2}
(\filteredloops X)^{\mathrm{u}}(B):=\Map_k(\Filcircleund_k\times \Spec\, B,X) \longrightarrow (\loops X)(B):=\Map_k(\circle\times \Spec\, B,X)
\end{equation}

\noindent is an equivalence. To show this we use the fact that every derived affine scheme can be written as a limit of copies of the affine line $\affineline{k}$ (see \cite[4.1.9]{lurie-structuredspaces}). In this case, we are reduced to showing the statement for $X=\affineline{k}$ and use the fact that for any derived stack $F$, the mapping space $\Map_k(F, \affineline{k})$ is the Dold-Kan construction of the complex of global sections on $F$, $\structuresheaf(F)$. In this case, on the l.h.s we have

\begin{equation}
\label{covid1111}
\Map_k(\Filcircleund_k\times \Spec\, B,\affineline{k})\simeq  \DK(\structuresheaf(\B \Frobfixed\times \Spec \, B))
\end{equation}

\noindent and on the r.h.s we have 

\medskip

\begin{equation}
\label{covid1112}
\Map_k(\circle\times \Spec\, B,X)\simeq \DK( \structuresheaf(\circle \times \Spec \, B))
\end{equation}

It remains to show that the induced pullback map

$$
\structuresheaf(\B \Frobfixed\times \Spec \, B)\to \structuresheaf(\circle \times \Spec \, B)
$$

\noindent is a quasi-isomorphism of complexes.  In this case, we can use the fact that $\B \Frobfixed$ is of finite cohomological dimension (\cref{lemma-Bfixfinitecohdimension}) to apply the base change formula \cite[A.1.5-(2)]{1402.3204}\footnote{See also \cite[9.1.5.6-(c) and 9.1.5.7-(c)]{lurie-sag}} to establish an quasi-isomorphism between the underlying complexes

\begin{equation}
\label{covid11113}
\structuresheaf(\B \Frobfixed\times \Spec \, B)\simeq \structuresheaf(\B\Frobfixed)\otimes  B
\end{equation}\footnote{ If $i:\Spec\, B \to \Spec\, k$ and $\pi:\B \Frobfixed\to \Spec\, k$, with projections $p_1:\B \Frobfixed\times_{\Spec\, k} \Spec\, B \to \B \Frobfixed$, $p_2:\B \Frobfixed\times_{\Spec\, k} \Spec\, B \to \Spec \, B$, then the r.h.s is $i^\ast \pi_\ast$ and the l.h.s is $(p_2)_\ast \, (p_1)^\ast $.}

\medskip

\medskip

In the same way, the finite cohomological dimension for $\circle$ with the base change formula, gives

\begin{equation}
\label{covid11113}
\structuresheaf(\circle \times \Spec \, B)\simeq \structuresheaf(\circle)\otimes  B
\end{equation}

Finally, the fact that $\B\Frobfixed$ is the affinization of $\circle$ (\cref{propositionaffinizationunderlying}) concludes the proof.\\

 We now deal with statement (2). We start constructing the map by noticing the existence of the following commutative square of graded affine stacks
$$\xymatrix{
\Spec\, \dualnumbers{k} \ar[r] \ar[d] & \ar[d] \Spec\, k \\
\Spec\,k \ar[r] & \Filcirclegr_k.
}$$

\noindent with $\epsilon$ in degree $0$ and weigth 1. Such a commutative square of stacks 
is given by an element of  $\Frobkernel(\dualnumbers{k})$ interpreted as a map of graded affine schemes

$$
\Spec \dualnumbers{k}\to \Frobkernel\simeq \Omega \, \Filcirclegr
$$

\medskip

\noindent which simply is the $p$-typical Witt vector 

\begin{equation}
    \label{covid11115}
    \underline{\epsilon}:=(\epsilon, 0,0,0,...)
\end{equation}

\noindent whose Ghost components are $(\epsilon, \epsilon^p, \epsilon^{p^2}, ..)$. As explained in the \cref{guideWitt}, the action of $\Frob_p$ on $\Wittp$ is determined by what it does on the Ghost coordinates, in this case, $(\epsilon, \epsilon^p, \epsilon^{p^2}, ..)\mapsto ( \epsilon^p, \epsilon^{p^2}, ..)$. Reverse-engineering the Ghost cordinates we obtain $\Frob_{p}(\epsilon, 0,0,0,...)=(\epsilon^p, 0,0,..)$. The fact that $\epsilon^p=0$ tells us that the $p$-typical Witt vector $\underline{\epsilon}$ is in the kernel of  $\Frob_p$.

\noindent This square induces a commutative diagram of graded derived affine schemes,
obtained by mapping to $X$

\begin{equation}
\label{covid11116}
\xymatrix{
\Spec \,\Sym_A(\cotangent_A)  & \ar[l] X \\
X \ar[u] & \Map_k(\Filcirclegr_k,X). \ar[u] \ar[l]
}
\end{equation}

\medskip

\noindent This in turn produces a natural morphism of derived stacks

\begin{equation}
\label{covid11114}
\Map_k(\Filcirclegr_k,X) \longrightarrow 
X\times_{\Spec\, \Sym_A(\cotangent_A)}X \simeq \Spec\, 
\Sym_A(\cotangent_{A}[1])
\end{equation}

\medskip

\noindent To check its an equivalence we proceed as before by reducing to $X=\affineline{k}$ using the fact that the construction $A\mapsto  \Sym_A(\cotangent_A[1])$ is compatible with colimits (using the fact that the cotangent complex is left adjoint) and that $\B\Frobkernel$ is of finite cohomological dimension (\cref{lemma-BKerfinitecohdimension}): indeed, for $X=\affineline{k}$ the l.h.s of \eqref{covid11114} we find the functor of points

$$
B\mapsto  \DK( \structuresheaf(\B \Frobkernel \times \Spec \, B))\simeq \DK(\structuresheaf(\B \Frobkernel)\otimes B)\simeq \DK(B\oplus B[-1])
$$

\noindent where the last equivalence follows from \cref{theorem-splitsquarezero}. At the same time, on the r.h.s we find

$$
B\mapsto \DK(B\times_{B[\epsilon]} B)\simeq \DK(B\oplus B[-1])
$$

\noindent where $B[\epsilon]$ is with $\epsilon$ in degree $0$. By the choice of the canonical element \eqref{covid11115} in  $\Frobkernel(\dualnumbers{k})$ used to produce the diagram \eqref{covid11116}, we conclude that the map from the r.h.s to the l.h.s
sends the generator in degree $-1$ to a generator in degree $-1$.

\medskip

\personal{In fact what we uncarefully proved is that \eqref{covid11116} is cocartesian inside affine stacks. }

 \end{proof}
\end{theorem}

\medskip

\medskip

\subsection{The action of the filtered circle on filtered loops and its fixed points}

\begin{construction}
\mylabel{filteredglobalsectionscircleaction}
Consider the filtered loops stack $p:\filteredloops X\to \Filstack$ as in \cref{filteredloopspace}. Since the action of $\Filcircle$ is compatible with the filtration, the map $p$ is $\Filcircle$-equivariant and descends to the quotients

$$
a: \filteredloops X/\Filcircle\to \Filstack/\Filcircle \simeq \B \Filcircle
$$

\noindent and we have a pullback diagram of filtered stacks
$$
\xymatrix{
\ar[d]^{p}\filteredloops X \ar[r]^-{\rho}& \filteredloops X/\Filcircle\ar[d]^{a}\\
\Filstack\ar[r]^-{e}& \B \Filcircle
}$$

\noindent where $\rho$ and $e$ are the canonical atlases. Since $a$ is an \textcolor{black}{affine map} (\cref{filteredloopsrepresentablyrelative}), we can use \cite[9.1.5.8, 9.1.5.3, 9.1.5.7]{lurie-sag} (see the \cref{globalsectionsoffilteredarefiltered}-(iii)) to deduce the base change property for the diagram

\begin{equation}
    \label{basechangeflatatlasaction}
\xymatrix{
\ar[d]^{p_\ast}\Qcoh(\filteredloops X)& \ar[l]^-{\rho^\ast} \Qcoh(\filteredloops X)^{\Filcircle-eq}\ar[d]^{a_\ast}\\
\Qcoh(\Filstack)&\ar[l]^-{e^\ast}\Qcoh( \B \Filcircle)
}\end{equation}

\noindent where $\rho^\ast$ and $e^\ast$ corresponding to the functors forgetting the action. We find that that the object $\Ofil(\filteredloops X)= p_\ast \structuresheaf_{\filteredloops X}$ of \cref{filteredloopspace} carries a canonical action of $\Filcircle$ \personal{(In fact this can be seen explicitely from the fact that $a_\ast$ is limit functor induced from the simplicial diagrams encoding the actions)}. We will sometimes write $\Ofil(\filteredloops X)$ to denote its lift as an object in $\Qcoh(\B\Filcircle)$. The result of this construction can be exhibited as an $\infty$-functor which we will suggestively\footnote{This choice of notation will become clear with the \cref{thmhkr}-(a) below.} denote as 

\begin{equation}
    \label{functorHHfiltered}
    \HHFil:\SCRings{k}\to \Qcoh(\B\Filcircle)\, \,\,\,, \,\,\, \HHFil(A):= \Ofil(\filteredloops \Spec A)
\end{equation}

\noindent As a consequence of  \cref{proprepmonoidal}-(i), the associated underlying object is an $\Einfinity$-algebra with a compatible $\Filcircleund$-action, and by \cref{proprepmonoidal}-(ii), the
associated graded is a graded mixed $\Einfinity$-algebra with a compatible $\Filcirclegr$-action. 
\end{construction}

 \medskip

\begin{construction}
\mylabel{filteredglobalsectionsfixedpoints}
Consider the filtered loops stack $p:\filteredloops X\to \Filstack$ as in \cref{filteredloopspace}. We want to understand the operation of extracting fixed points of $\Filcircle$-action on the filtered object $\Ofil(\filteredloops X)\in \Qcoh(\B \Filcircle)$. Concretely this operation is given by pushforward along the structure map $\pi:\B \Filcircle\to \Filstack$. 

We want to understand how this is compatible with taking associated graded and underlying objects. To that end, we start by considering the following commutative diagram of pullback squares

\begin{equation}
\label{keydiagramobstructionbasechange}
\xymatrix{
&(\filteredloops X)^{\mathrm{gr}}\ar[rr]^{\Tilde{0}}_{}\ar[dl]_{\rho^{\mathrm{gr}}}\ar[dd]^(.3){p^{\mathrm{gr}}}&&\ar[dl]^\rho \ar[dd]^(.3){p}\filteredloops X&&\ar[ll]_{\Tilde{1}}^{}\ar[dd]^(.3){p}(\filteredloops X)^{\mathrm{u}}\ar[dl]^{\rho^{\mathrm{u}}}\\
 (\filteredloops X)^{\mathrm{gr}}/(\Filcircle)^{\mathrm{gr}}\ar[dd]^(.3){a^{\mathrm{gr}}}\ar[rr]^(0.7){\tilde{0}_a}&&\filteredloops X/\Filcircle \ar[dd]^(.3){a}&&(\filteredloops X)^{\mathrm{u}}/(\Filcircle)^{\mathrm{u}}\ar[ll]_(0.3){\tilde{1}_a}\ar[dd]^(.3){a^{\mathrm{u}}}&\\
&\B \Gm{}\ar[dl]^{e^\mathrm{gr}}\ar[rr]^(0.3){0}&&\Filstack\ar[dl]^{e}&&\ast\ar[dl]^{e^{\mathrm{u}}}\ar[ll]_(0.3){1}\\
\B \Filcirclegr\ar@{}[urrr]_-{}\ar[d]^{\pi^{\mathrm{gr}}}\ar[rr]^{0_a} \ar@{}[uuur]_-{}&&\ar@{}[ur]^(0.3){(j)}\B \Filcircle\ar[d]^{\pi}&&\B \Filcircleund\ar@{}[uur]^>{}\ar[d]^{\pi^{\mathrm{u}}}\ar[ll]_{1_a}&\\
\B \Gm{} \ar@{}[urr]^-{}\ar[rr]^{0}&&\Filstack \ar@{}[urr]^-{}&&\ast\ar[ll]_{1}&\\
}
\end{equation}

\noindent and observe that the Beck-Chevalley transformations

\begin{equation}
    \label{beckchevalley1}
(0_a)^\ast \,\, a_\ast \,\, \structuresheaf\to (a^{\mathrm{gr}})_\ast \,\, (\tilde{0}_a)^\ast\,\, \structuresheaf
\end{equation}

\noindent and

\begin{equation}
    \label{beckchevalley1foropen}
(1_a)^\ast \,\, a_\ast \,\, \structuresheaf\to (a^{\mathrm{u}})_\ast \,\, (\tilde{1}_a)^\ast\,\, \structuresheaf
\end{equation}

\medskip

\noindent are equivalences. As in the \cref{filteredglobalsectionscircleaction}, this follows because the map $a$ is an \textcolor{black}{affine map}, and we can use \cite[9.1.5.8, 9.1.5.3, 9.1.5.7]{lurie-sag} (see the \cref{globalsectionsoffilteredarefiltered}-(iii)) to deduce the base change property against all maps.

\noindent Finally, using \cref{basechangeBfilteredcircle} we conclude that that base-change holds for $\pi$ in the context of \cref{globalsectionsoffilteredarefiltered}-(i) and \cref{globalsectionsoffilteredarefiltered}-(ii), namely:

\begin{itemize}
    \item  The combination of \eqref{beckchevalley1}
 and \eqref{beckchevalley2} implies that the Beck-Chevalley transformation
 
 \begin{equation}
     \label{beckchevalley3}
     ([\Ofil(\filteredloops X)]^{\mathrm{h}\Filcircle})^{\mathrm{gr}}\simeq 0^\ast\,\, \pi_\ast\,\, a_\ast\,\,\structuresheaf\,\,\to (\pi^{\mathrm{gr}})_\ast\,\, (0_a)^\ast\,\, a_\ast\,\, \structuresheaf\simeq [(\Ofil(\filteredloops X))^{\mathrm{gr}}]^{\mathrm{h}\Filcirclegr}
 \end{equation}

\noindent is an equivalence, or that, in order words, taking fixed points commutes with taking the associated graded.\\

\item The failure of \eqref{beckchevalley2foropen} to be an equivalence in general tells that we cannot guarantee that the natural map induced by \eqref{beckchevalley1foropen}

 \begin{equation}
     \label{beckchevalley3foropen}
     ([\Ofil(\filteredloops X)]^{\mathrm{h}\Filcircle})^{\mathrm{u}}\simeq 1^\ast\,\, \pi_\ast\,\, a_\ast\,\,\structuresheaf\,\,\to (\pi^{\mathrm{u}})_\ast\,\, (1_a)^\ast\,\, a_\ast\,\, \structuresheaf\simeq [(\Ofil(\filteredloops X))^{\mathrm{u}}]^{\mathrm{h}\Filcircleund}
 \end{equation}

\noindent is an equivalence in general, or, in other words, that taking the underlying object of the fixed points of the filtration is the fixed points of the original underlying object. Notice that via the actions of $\B\circle\simeq\B \Filcircleund$ of \cref{proprepmonoidal}-(i), and the base change \eqref{beckchevalley1foropen}, the r.h.s of \eqref{beckchevalley3foropen} recovers the definition of negative cyclic homology

$$[(\Ofil(\filteredloops X))^{\mathrm{u}}]^{\mathrm{h}\Filcircleund}\simeq [\HH( X)]^{\mathrm{h}\circle} \simeq \HCmin(X)$$.

\medskip

\noindent However, since $1_a$ is an open immersion, the \cref{globalsectionsoffilteredarefiltered}-(i) guarantees that that the Beck-Chevaley transformation \eqformula{beckchevalley3foropen} is an equivalence for objects in $\Qcoh(\B\Filcircle)^{<\infty}$, ie, the commutativity of 

\begin{equation}
    \label{okforboundedabove}
    \xymatrix{
    \Qcoh(\B\Filcircle)^{<\infty}\ar[r]^-{(-)^{\mathrm{u}}}\ar[d]^{(-)^{\mathrm{h}\Filcircle}}& \Qcoh(\B \circle)\ar[d]^{(-)^{\mathrm{h}\circle}}\\
    \Qcoh(\Filstack)\ar[r]^-{(-)^{\mathrm{u}}} &\Mod_{k}}
\end{equation}

\noindent is established.

\end{itemize}

\end{construction}

\medskip

\noindent Due to the failure of \eqref{beckchevalley3foropen} to be an equivalence in general, we introduce the following definition:

\begin{definition}
\mylabel{definition-filteredHCminus}
Consider the filtered loop space as in \cref{filteredloopspace}. We define a filtered version of negative cyclic homology as the filtered object 

\begin{equation}
\label{formula-HCminusFil}
\HCminFil(X):=  [\Ofil(\filteredloops X)]^{\mathrm{h}\Filcircle}\simeq  1^\ast\,\, \pi_\ast\,\, a_\ast\,\,\structuresheaf\,\,\in \,\,\Qcoh(\Filstack)
\end{equation}

\medskip

\noindent By the \cref{filteredglobalsectionsfixedpoints}, the underlying object of this filtration comes canonically equipped with a map to usual negative Hochschild homology

\begin{equation}
\label{comparisonmapfilteredHCminus}
\HCminFil(X)^{\mathrm{u}}\to \HCmin(X)
\end{equation}

\end{definition}

\medskip

\subsection{De Rham algebra functor}

\begin{construction}
\mylabel{DRfunctor}
Recall that any $A\in \SCRings{k}$
possesses a derived de Rham algebra $\DR(A/k)$. 
For us this is a graded mixed $\Einfinity$-algebra
constructed as follows. When $A$ is smooth
over $k$, $\DR(A/k)$ simply is the strictly commutative dg-algebra $\Sym_A(\Omega^1_{A/k})$ 
endowed with its de Rham differential. This
is a strictly commutative monoid inside the 
strict category of graded mixed complexes, and thus, by the \cref{proprepmonoidal}-(ii),
can be considered as a $\Einfinity$-algebra object
inside $\Qcoh(\B\Filcirclegr)$, the symmetric monoidal $\infty$-category of graded mixed complexes (here over $k$).

This produces an $\infty$-functor from smooth 
algebras over $k$, and thus from polynomial $k$-algebras, to graded mixed $\Einfinity$-algebras. 
By left Kan extension we get the de Rham functor

$$\DR : \SCRings{k} \longrightarrow \CAlg(\Qcoh(\B\Filcirclegr)).$$
\end{construction}

\medskip

\begin{remark}
\mylabel{upgradedDR}
We will see below that the functor $\DR$ as defined above with values in $\Einfinity$-algebras in $\Qcoh(\B\Filcirclegr)$ can be refined to take in values  "simplicial commutative  mixed graded algebras". More precisely, what this means is that we can exhibit it as a functor with valued in derived affine schemes over $\Filcirclegr$, which contains slightly more information. \end{remark}

\medskip

\begin{construction}
\mylabel{derivedderrhamcomplex}
\noindent The functor of global sections produces a symmetric
lax-monoidal $\infty$-functor
$$\Qcoh(\B\Filcirclegr) \longrightarrow \Qcoh(\B\Gm{}).$$
\noindent which formally corresponds to taking homotopy fixed points with respect to the action of the graded group stack $\Filcirclegr$. We will denote it by $(-)^{\mathsf{h}\Filcirclegr}$. The composition
$$
\xymatrix{\SCRings{k}\ar[r]^-{\DR}& \Qcoh(\B\Filcirclegr)\ar[r]& \Qcoh(\B\Gm{})}
$$

\noindent sends $A$ to the derived de Rham 
complex of $A$

$$\mathbb{L}\widehat{\DR}(A/k):= \DR(A/k)^{\mathsf{h}\Filcirclegr} \in \CAlg(\Qcoh(\B\Gm{})).$$

\noindent This is a graded $\Einfinity$-algebra whose piece of 
weight $i$ will be denoted by $\mathbb{L}\widehat{\DR}^{\,\,\geq i}(A/k)$. One can provide an explicit description of these graded pieces. Indeed, by means of the equivalence $\Qcoh(\B\Filcirclegr)\simeq \Mod_\Lambda$ of \cref{proprepmonoidal}-(ii), the functor $(-)^{\mathrm{h}\Filcirclegr}$ can be identified with the functor 

$$
\Mod_\Lambda\to \Mod_k^{\bbZ-\gr}\,\,\,\,,\,\,\,\,\rmM\mapsto \bigoplus_{i\in \bbZ} \RHom_{\Mod_\Lambda}(\integerslocalp(i), \rmM)
$$

\noindent where the graded object $\integerslocalp(i)$ is seen as a graded $\Lambda$-module with a trivial action of the element $\epsilon$. In \cite[Prop. 1.3]{1111.3209} the authors prescribe an explicit resolution of $\integerslocalp(i)$ as a $\Lambda$-module, given by $Q(i):= \underset{n\geq 0}{\bigoplus} \Lambda(i+n)[2n]$ so that the graded piece of weight $i$ in $(\rmM)^{\mathrm{h}\Filcirclegr}$ is given by $\underset{n\geq 0}{\prod}\gr^{i+n}{\rmM}[-2n]$. In the case of the de Rham complex, we obtain for $i\geq 0$ that  $\mathbb{L}\widehat{\DR}^{\geq i}(A/k)$ should be understood as the complex 
$\prod_{n\geq 0}(\wedge^{i+n}\mathbb{L}_{A/k}[i-n])\simeq \prod_{q \geq i}(\wedge^{q}\mathbb{L}_{A/k}[2i-q])  $ endowed with the total differential $d+\dderham$, 
where $d$ is the cohomological differential induced
from $A$ and $\dderham$ is the de Rham differential. See for instance \cite[\S 5]{MR3285853} or \cite[\S 1.2]{MR3090262}. For $i\leq 0$ the graded pieces of weight $i$ are all equivalent to a shifted copy of the piece of weight $0$, namely,  $\mathbb{L}\widehat{\DR}^{\,\,\geq 0}(A/k)[2i]$, thus reflecting the 2-periodic shape of the cohomology of the graded circle of \cref{groupcohomologyfilteredcircle}.
\end{construction}

\medskip

\begin{remark}
\mylabel{truncatedderhamcomplex}

\noindent When $A\in \SCRings{}$ is a smooth discrete algebra, for $i\geq 0$, the graded pieces $\mathbb{L}\widehat{\DR}^{\,\,\geq i}(A/k)$ of \cref{derivedderrhamcomplex} are given by $\Omega_A^{\geq i}[2i]$ where  $\Omega_A^{\geq i}$ is the chain complex 

$$
\cdots \to 0\to \Omega_A^i\to \Omega_A^{i+1}\to \Omega_A^{i+2}\to\cdots
$$

\noindent where $\Omega_A^i$ is in homological degree $-i$ and the maps are given by the de Rham differential $d_R$. In this case when $i\leq 0$, we have

\begin{equation}
\mylabel{negativegradedpieces}
\gr^i(\DR(A/k)^{\mathsf{h}\Filcirclegr})\simeq \, \Omega_A^{\geq 0}[2i]
\end{equation}
\end{remark}

\medskip

\subsection{Main theorem}

We are now ready to  prove the main results in this paper.

\medskip

\begin{theorem}\mylabel{thmhkr}
Let $X=\Spec\, A$ be an affine derived scheme over $k$ and $\filteredloops X$ its filtered loop space as in \cref{filteredloopspace}.  Then, the cohomology $\Einfinity$-algebra
    $\Ofil(\filteredloops X)$
    endowed with its natural $\Filcircle$-action 
    is such that:
\medskip

\begin{enumerate}[(a)]
        \item Its underlying object is naturally
        equivalent to $\HH(A/k)$, the Hochschild homology of $A$ over $k$, together with its
        natural $\circle$-action.
        \medskip
        
        \item Its associated graded is naturally 
        equivalent to $\DR(A/k)$ as a graded 
        mixed $\Einfinity$-algebra. 
        \medskip
        
        \item Being compatible with the action of the filtered circle, the filtration descends to fixed points in the sense of \cref{definition-filteredHCminus} and makes $\HCminFil(A)$ a filtered algebra, whose
        
        \begin{itemize}
            \item associated graded pieces are  the truncated complete derived de Rham complexes
    $\mathbb{L}\widehat{\DR}^{\,\,\geq p}(A/k)$;
    \item underlying object comes equipped with a canonical map to usual negative cyclic homology \eqref{comparisonmapfilteredHCminus}:$\,\,\HCminFil(X)^{\mathrm{u}}\to \HCmin(X)$

        \end{itemize}

  \item  If $A$ is discrete and smooth, then the above map \eqref{comparisonmapfilteredHCminus} is an equivalence, hence it endows $\HCmin(X)$  with an exhaustive filtration in the sense of \cite[Definition 5.1]{MR3949030}. 
\end{enumerate}
\medskip

\begin{proof}
The claim (a) is direct a re-interpretation of the base change \eqref{beckchevalley1foropen} established in the \cref{filteredglobalsectionsfixedpoints}.\\

\medskip

To prove (b) we have to 
compare the two $\infty$-functors
$$\SCRings{k} \longrightarrow \CAlg(\Qcoh(\B\Filcirclegr_k))$$
The first one given by $\DR(-/k)$, the second
given by 
$$A \mapsto \structuresheaf(\Map_k(\Filcirclegr_k,\Spec\, A))$$

\noindent For this, we first notice that if we forget 
the action of the group $\Filcirclegr$ (i.e.
the mixed structure), then these two $\infty$-functors
are equivalent and given by $A \mapsto \Sym_A(\cotangent_A[1])$. As the functor $A \mapsto 
\cotangent_A$ is obtained by extension under sifted colimits 
from polynomial rings (see \cref{notationcosimplicialandsimplicial}), this shows that the same is true for the two functors to be shown to be equivalent. In other words, we can restrict these to the category of polynomial $k$-algebras.

We then observe that for any 
polynomial $k$-algebra $A$ the space
of graded mixed structures on the graded $\Einfinity$-algebra
$\Sym_A(\Omega_{A/k}^1[1])$ is a discrete space. Indeed, this follows from the fact that the space
of graded $\Einfinity$-endomorphisms is itself discrete, because
the weight grading coincide with the cohomological grading (as in the proof of \cref{uniqueEinfinityalgebra}).
As a consequence, in order to show that 
the two above $\infty$-functors are equivalent 
it is enough to show that for a fixed
polynomial $k$-algebra $A$, the natural 
isomorphism of graded algebras
$$\DR(A/k) \simeq \mathcal{O}(\filteredloops^{\gr}(X))$$
intertwine the two graded mixed structures. We can even be more precise, the compatible graded mixed
structures on the graded $\Einfinity$-algebra 
$\DR(A/k)$ form a discrete space which embeds
into the set of $k$-linear derivations $A \longrightarrow \Omega_A^1$. 

As a result, we are reduced to proving that, by the 
above identification, the differential obtain
from the $\Filcirclegr$-action on the right hand side
$$d : \pi_0(\mathcal{O}(\filteredloops^{\gr}(X))\simeq A \longrightarrow 
\pi_1(\mathcal{O}(\filteredloops^{\gr}(X)))\simeq 
\Omega_{A/k}^1$$
is indeed equal to the standard de Rham differential.
For this, we can of course 
assume that $k=\integerslocalp$, as the general
case would be obtained by base change. But in this
case all complexes involved are torsion free; one may
then simply base change to $\bbQ$ to check the mixed structure above is the de Rham
differential. But the result is well known in characteristic zero (see 
\cite{MR2862069}).  \\

Parts (c) is a direct consequence of the base change formula \eqformula{beckchevalley3} in  Construction \cref{filteredglobalsectionsfixedpoints}, \cref{derivedderrhamcomplex} and point (b).\\

 It remains to prove (d). Recall that the map \eqref{comparisonmapfilteredHCminus} is a Beck-Chevalley transformation of \eqformula{beckchevalley3foropen}. To conclude that it is an equivalence, following the commutativity of \eqformula{okforboundedabove}, it is enough to show that when $A$ is smooth and discrete, $\HHFil(A)\in \Qcoh(\B \Filcircle)_{<\infty}$. By the \cref{tstructureQcohBfilteredcircle}, this means that as an object $\HHFil(A)\in\Fil(\Mod_k)$, each level of the filtration, $\HHFil(A)_i$ is in $\Mod_k^{<\infty}$. We show that this is indeed the case for $A$ smooth and discrete. By (a) we know that the underlying object of $\HHFil(A)$ is $\HH(A)$. By (b), in the smooth discrete case the associated graded is $\bigoplus_{i\geq 0} \Omega^i_A [i]$ where the $\Omega^i_A$ are concentrated in a single degree, zero. Moreover, if $A$ is smooth of finite dimension $d$, by the classical HKR theorem \cite{MR0142598} we have an identification of abelian groups

$$
 \pi_i(\HH(A)) \simeq  
 \begin{cases}
    \Omega^i_A, & \text{for } 0\leq i\leq d\\
    0, & \text{ for } i<0 \text{ and } i> d 
    \end{cases}
$$

Therefore,  we have  

$$
\gr^i(\HHFil(A))\simeq 0 \text{ for } i<0 \text{ and } i> d 
$$

Using the cofiber sequences

$$
\xymatrix{
\HHFil(A)_{i+1}\ar[r]\ar[d]& \HHFil(A)_{i}\ar[d]\\
0\ar[r]& \gr^ i\, \HHFil(A)
}
$$

\noindent  we have

$$
\HHFil(A)_0\, \simeq \HH(A) \,\,\,\, \text{ , }\,\,\, \lim\, \HHFil(A) \, \simeq \HHFil(A)_{d+1}\simeq 0
$$

\medskip

\noindent and $\HHFil(A)_{i}\in \Mod_k^{\leq d}$. This concludes the proof.

\end{proof}
\end{theorem}

\medskip

\subsection{Comparison with the filtration of Antieau }
\mylabel{section-comparison-Antieu}

\noindent In this section we compare the filtered object $\HCminFil(A)$ (\cref{definition-filteredHCminus} and \cref{thmhkr}) with the filtered object $\HCminFilAntieu(A)$ constructed by Antieau in \cite[Thm 1.1]{1808.05246} which we will revise below. Our comparison neglects the compatibility with the algebra structures and focuses only on  underlying filtered complexes. 

\medskip
\noindent The strategy consists of two main steps:
\medskip

\begin{enumerate}
    \item We compare the two filtrations on smooth polynomial algebras (\cref{comparisonBenresult1})
    \medskip
    
    \item We use the fact that $\SCRings{k}$ is the sifted completion of the discrete category of polynomial algebras $\Nerve(\Poly{k})$ (see \cite[25.1.1.5]{lurie-sag} and \cite[5.5.9.3]{lurie-htt}) to Kan extend the comparison to all simplicial commutative algebras (\cref{finalcomparisonBen}).

\medskip

\end{enumerate}

\medskip

\noindent Let us start by setting some terminology and collecting some facts.

\medskip

\begin{definition}
\mylabel{completemodulesS1action} We say that an object $\rmE\in \Qcoh(\B\Filcircle)$ is \emph{complete} if its underlying filtered module obtained by pullback along the canonical atlas $\Filstack\to \B\Filcircle$
$$
\Qcoh(\B\Filcircle)\to \Qcoh(\Filstack)\simeq \Fil(\Mod_k)
$$
is complete (in the sense of \cref{definition-completefiltration}). We denote by $\widehat{\Qcoh(\B\Filcircle)}$ the full subcategory of complete filtered modules with an $\Filcircle$-action. Moreover, as in \cref{remark-conditionforcomplete}, we will denote by $\Qcoh(\B\Filcircle)^{\dagger}$ the full subcategory of $\widehat{\Qcoh(\B\Filcircle)}$  spanned by those representations whose underlying filtered object is in $\Fildescend(\Mod_k)$. In particular, $\structuresheaf\in \Qcoh(\B\Filcircle)^{\dagger}$ by definition of the tensor unit in filtered objects (see \cref{construction-Filteredobjects}).
\end{definition}

\medskip

\begin{remark}
\mylabel{remark-smoothcasefiltrationisPostnikov}
The proof of \cref{thmhkr}-(d) shows that when $A$ is smooth discrete, we have a canonical equivalence  in $\Fil(\Mod_k)$, $\HHFil(A)\simeq \tau_\geq \HH(A)$, where $\tau_\geq $ is the Whitehead tower of \cref{construction-whiteheadtowerfunctor} for the standard $t$-structure on $\Mod_k$.  Indeed, this follows from the discussion in loc.cit and the characterization of the essential image of $\tau_{\geq}$.
In particular, since the $t$-structure on $\Mod_k$ is left complete \cite[7.1.1.13]{lurie-ha}, by the \cref{remark-conditionforcomplete}, the filtration $\HHFil(A)$ is complete and  $\HHFil(A)\in \Qcoh(\B\Filcircle)^{\dagger}$.
\end{remark}

\medskip

\begin{proposition}
\mylabel{prop-HHcommutessiftedcolimits}
The $\infty$-functor $\HHFil:\SCRings{k}\to \Qcoh(\B\Filcircle)$ 
preserves sifted colimits. In particular it is determined by its restriction 
\begin{equation}
\label{formula-HHfilrestrictedpolynomial}
\Nerve(\Poly{k})\subseteq\SCRings{k}\to \Qcoh(\B\Filcircle)
\end{equation}
\begin{proof}
Since the functor extracting the underlying filtered module $\Qcoh(\B\Filcircle)\to \Qcoh(\Filstack)$  preserves all colimits and is conservative, to conclude the proof of the proposition it is enough to look at the composition
$$
\xymatrix{
\SCRings{k}\ar[r]^{\Spec}_{\sim}& \dAff^{\op}\ar[rr]^-{\Map(\Filcircle,-)}&&(\dSt^{\mathrm{rel-aff}}_{/\Filstack_k})^{\op}\ar[r]^{\structuresheaf^{\Fil}}& \Qcoh(\Filstack)
}$$

\noindent where $(\dSt^{\mathrm{rel-aff}}_{/\Filstack_k})^{\op}$ denotes the category of derived stacks over $\Filstack$ that are relatively affine (\cref{filteredloopsrepresentablyrelative}). \\
The functor $\Spec$, being an equivalence, sends sifted colimits to sifted limits of affine schemes.  By definition (and Yoneda), the functor $\Map(\Filcircle,-)$ commutes with all limits. Finally, since the derived loop spaces are relatively affine over $\Filstack$, by base-change along the atlas $\affineline{k}\to \Filstack$ one deduces that the functor of (relative) global sections also commutes with all limits. The last claim in the proposition follows the description of $\SCRings{k}$ as a completion of $\Nerve(\Poly{k})$ under sifted colimits.
\end{proof}
\end{proposition}

\medskip

\begin{remark}
\mylabel{Fildescendstablecolimits}
\noindent The inclusion $\Fildescend(\Mod_k)\subseteq \Fil(\Mod_k)$ is stable under all colimits. Indeed, since colimits in $\Fil(\Mod_k)$ are computed levelwise, this amounts to the fact that the full  subcategory $\Mod_k^{\geq i}\subseteq\Mod_k$ of $i$-connected objects is stable under all colimits (see \cite[ 1.2.1.6]{lurie-ha}). This implies that the inclusion $\Fildescend(\Mod_k)\subseteq \Filcomplete(\Mod_k)$ of \cref{remark-conditionforcomplete} also commutes with colimits, where on the r.h.s colimits are computed using the completion functor.

\noindent The same argument shows that both inclusions  $\Qcoh(\B\Filcircle)^{\dagger}\subseteq \widehat{\Qcoh(\B\Filcircle)}$ and $\Qcoh(\B\Filcircle)^{\dagger}\subseteq \Qcoh(\B\Filcircle)$ commute with all colimits.
\end{remark}

\medskip

\begin{proposition}
\mylabel{prop-HHfilterediscomplete}
The $\infty$-functor $\HHFil:\SCRings{k}\to \Qcoh(\B\Filcircle)$ has image in the subcategory $\Qcoh(\B\Filcircle)^{\dagger}$. Moreover, the factorization

\begin{equation}
\label{functor-HHfil-complete}
    \HHFil: \SCRings{k}\to \Qcoh(\B \Filcircle)^{\dagger}
\end{equation}

\noindent is compatible with sifted colimits.

\begin{proof}
This is now a consequence of the fact $\SCRings{k}$ is the sifted completion of $\Nerve(\Poly{k})$, the \cref{prop-HHcommutessiftedcolimits}, the fact that the restriction of $\HHFil$  along $\Nerve(\Poly{k})\subseteq \SCRings{k}$ factors through $\Qcoh(\B \Filcircle)^{\dagger}$ (\cref{remark-smoothcasefiltrationisPostnikov}) and finally the fact the inclusion $\Qcoh(\B\Filcircle)^{\dagger}\subseteq \Qcoh(\B\Filcircle)$ commute with all colimits (\cref{Fildescendstablecolimits}).

\end{proof}
\end{proposition}

\medskip

\begin{remark}
\mylabel{Whiteheadtowerrepcirclecommute}
We denote by $\tau^{\Fil}_{\geq}$ the Whitehead tower functor for the left complete $t$-structure on $\Qcoh(\B\Filcircle) $ of \cref{tstructureQcohBfilteredcircle}. The argument used in loc. cit. also explains the existence of  left complete $t$-structures both on $\Qcoh(\B \Filcircleund)$ and $\Qcoh(\B \Filcirclegr)$:

\medskip

\begin{enumerate}

\item Since the map $1_a:\B\Filcircleund\to \B\Filcircle$ in \eqformula{keydiagramobstructionbasechange} is an open immersion, the pullback $1_a^\ast:\Qcoh(\B\Filcircle)\to \Qcoh(\B \Filcircleund) $ is $t$-exact. The characterization of the two $t$-structures via the atlas guarantees that under the equivalence of \cref{proprepmonoidal}-(i)
$$\Qcoh(\B \Filcircleund)\simeq \Qcoh(\B\circle)\simeq \Fun(\B\circle, \Mod_k)$$

\noindent  the $t$-structure on $\Qcoh(\B \Filcircleund)$ becomes the $t$-structure on $\Fun(\B\circle, \Mod_k)$ induced from the \emph{standard} t-structure on $\Mod_k$ via the forgetful functor $\Fun(\B\circle, \Mod_k)\to \Mod_k$. Therefore, the associated Whitehead tower functors of \cref{construction-whiteheadtowerfunctor} commute:

\begin{equation}
    \label{whiteheadtowercommute}
    \xymatrix{
    \Qcoh(\B\Filcircle)\ar[d]_{\tau^{\Fil}_{\geq}}\ar[r]^{1_a^\ast}& \Fun(\B\circle, \Mod_k)\ar[d]_{\tau_{\geq}^{\mathsf{std}}}\\
    \Fil(\Qcoh(\B\Filcircle))\ar[r]^-{1_a^\ast\circ -}&\Fil(\Fun(\B\circle, \Mod_k))
    }
\end{equation}
\medskip

\noindent Notice that forgetting the circle actions and under the equivalence of \cref{prop:gradingsandstacks}, the discussion in \cref{tstructureQcohBfilteredcircle} shows that the $t$-exactness of $1_a^\ast$ amounts to the $t$-exactness of the $\colim$ functor
$
\colim: \Fun(\Nerve(\bbZ^\op), \Mod_k)\to \Mod_k
$ where on the l.h.s we have the levelwise standard $t$-structure of \cite[Proposition 1.4.3.6]{lurie-ha}. In this form, the statement amounts to the fact that homology groups commute with filtered colimits (see \cite[Proposition 1.3.5.21]{lurie-ha}).

\medskip

\vspace{1cm}

\item The same argument of \cref{tstructureQcohBfilteredcircle} using the atlas, shows that under the equivalence $\Qcoh(\B\Gm{})\simeq \Mod_k^{\bbZ-\gr}$ of \cref{prop:gradingsandstacks}, the $t$-structure on $\Qcoh(\B\Filcirclegr)$ is the levelwise $t$-structure. In particular, since the pullback along the map $0_a:\B\Filcirclegr\to \B \Filcircle$ of \eqformula{keydiagramobstructionbasechange} is right $t$-exact (\cref{left-complete-t-structure}-b)), implying that the canonical maps

\begin{equation}
    \label{compatibilitywhiteheadgraded}
    \xymatrix{
  \tau^{\gr}_{\geq i}(\, (0_a)^*\circ \tau^{\Fil}_{\geq i}\,)\ar[r]&   (\,(0_a)^*\circ \tau^{\Fil}_{\geq i}\,)
    }
\end{equation}

\noindent are equivalences.


\end{enumerate}
\end{remark}

\medskip

\begin{remark}
\mylabel{postnikovbifilteredcircle}
If $\rmE\in \Qcoh(\B\Filcircle)^{\dagger}$ then for each $n\in \bbZ$ we have $\tau_{\geq n}^{\Fil}\rmE \in \Qcoh(\B\Filcircle)^{\dagger}$. Indeed, as explained in the \cref{tstructureQcohBfilteredcircle}, the underlying filtered object of $\tau_{\geq n}^{\Fil}\rmE$ is the sequence

$$
\cdots \to \tau_{\geq n}\rmE_{3} \to \tau_{\geq n}\rmE_{2}\to \tau_{\geq n}\rmE_{1}\to \tau_{\geq n}\rmE_{0}\to \cdots
$$

\noindent where $\tau_{\geq n}$ is the truncation in $\Mod_k$. In particular, the fact that $\rmE_i\in \Mod_k^{\geq i}$ implies, as explained in the \cref{remark-conditionforcomplete}, that $\tau_{\geq n}\rmE_i\in \Mod_k^{\geq i}$ so that $\tau_{\geq n}^{\Fil}\rmE \in \Qcoh(\B\Filcircle)^{\dagger}$. In particular, since the Whitehead tower provides a complete filtered object (see the conclusion of \cref{remark-conditionforcomplete}) we find a commutative diagram

$$
\xymatrix{
\Qcoh(\B\Filcircle)^{\dagger}\ar@{^{(}->}[r]\ar@{-->}[d]& \Qcoh(\B\Filcircle)\ar[d]^{\tau_{\geq }^{\Fil}}\\
\Filcomplete(\Qcoh(\B\Filcircle)^{\dagger})\ar@{^{(}->}[r]&\Filcomplete (\Qcoh(\B\Filcircle))
}
$$

\end{remark}

\medskip

\begin{remark}
\mylabel{remark-fixedpointspreservescomplete} The construction of homotopy fixed points $(-)^{\mathrm{h}\Filcircle}$ preserves complete filtrations. Indeed, taking fixed points corresponds to the pushforward functor along the projection $\B\Filcircle\to \Filstack$:
$$(-)^{\mathrm{h}\Filcircle}:\Qcoh(\B\Filcircle)\to \Qcoh(\Filstack)\simeq \Fil(\Mod_k)$$
By construction this functor preserves limits and therefore by the \cref{completemodulesS1action}, it preserves complete filtrations. In particular, we have a commutative diagram

$$
\xymatrix{
\widehat{\Qcoh(\B\Filcircle)}\ar[rr]^-{(-)^{\mathrm{h}\Filcircle}}\ar@{^{(}->}[d]&& \Filcomplete(\Mod_k)\ar@{^{(}->}[d]\\
\Qcoh(\B\Filcircle)\ar[rr]^-{(-)^{\mathrm{h}\Filcircle}}&& \Fil(\Mod_k)
}
$$

\end{remark}

\medskip

\vspace{1cm}

We finally have all the ingredients to construct the commutative diagram that will allow us to compare our construction to that of Antieau \cite{1808.05246} in the smooth polynomial case:

\begin{construction}
\mylabel{comparisonBenpoly}
Notice that restriction to polynomial algebras
$$
\HHFil:\Nerve(\Poly{k})\subseteq \SCRings{k}\to \Qcoh(\B\Filcircle)^{\dagger}
$$
 takes values in $\Qcoh(\B\Filcircle)^{\dagger, <\infty}$ (see the proof of \cref{thmhkr}-(d)). Therefore, the combination of \cref{Whiteheadtowerrepcirclecommute}-(i), \cref{postnikovbifilteredcircle}, \cref{remark-fixedpointspreservescomplete} and the commutativity of \eqformula{okforboundedabove} allows us to establish the commutativity of 
 
 \begin{equation}
     \label{comparisonBensmoothdiagram1}
     \xymatrix{
     \Nerve(\Poly{k})\ar[d]^{\HHFil}\ar[drr]^{\HH}&&\\
     \Qcoh(\B\Filcircle)^{\dagger, <\infty}\ar[rr]^{(-)^{\mathrm{u}}}\ar[d]^{\tau_{\geq }^{\Fil}}&& \Qcoh(\B\circle)^{<\infty}\ar[d]^{\tau_{\geq}^{\mathrm{std}}}\\
     \Filcomplete[\,\,\Qcoh(\B\Filcircle)^{\dagger, <\infty}\,\,]\ar[d]^{(-)^{\mathrm{h}\Filcircle}_{lvl}}\ar[rr]^-{(-)^{\mathrm{u}}_{lvl}}&& \Filcomplete[\,\,\Qcoh(\B \circle)\,\,]\ar[d]^{(-)^{\mathrm{h}\circle}_{lvl}}\\
     \Filcomplete[\,\,\Filcomplete(\Mod_k)\,\,]\ar[rr]^{(-)^{\mathrm{u}}_{lvl}}&& \Filcomplete[\,\,\Mod_k\,\,]
     }
\end{equation}

\noindent Here the functor $(-)^{\mathrm{h}\circle}$ applied levelwise preserves complete filtrations because it is a right adjoint and therefore commutes with all limits.

\end{construction}

\medskip

\begin{notation}
\mylabel{Bennotation}
As in \cite[\S 3]{1808.05246} we denote the composition $(-)^{\mathrm{h}\circle}_{lvl}\,\circ \,\tau_{\geq}^{\mathrm{std}}\,\circ\, \HH$ by $\mathrm{F}_{\mathrm{HKR}}^{\bullet}\HCminus{-}$.
\end{notation}

\medskip

\begin{notation}
\mylabel{notationfilteredobjectsconstantfromlevel0} We shall denote by $\Fil^{\mathrm{const\geq 0}}(\C)$ the full subcategory of $\Fil(\C)$ spanned by those filtered objects $\rmE$ such that the maps $\rmE_0\to \rmE_{-1}\to \rmE_{-2}\to \cdots$ are equivalences. In particular, if $\rmE\in \Fil^{\mathrm{const\geq 0}}(\C)$, its colimit is $\rmE_0$. Moreover, for any functor $F:\C\to \D$, by functoriality $\Fil^{\mathrm{const\geq 0}}(\C)$ is sent to $\Fil^{\mathrm{const\geq 0}}(\D)$ and is compatible with the extraction of underlying objects.
\end{notation}

\medskip

\begin{remark}
\mylabel{remarkcolimitpostnikovrecoveroriginal}

\noindent We can complete the diagram \eqformula{comparisonBensmoothdiagram1} to the left. As explained in the \cref{remark-smoothcasefiltrationisPostnikov}, when $A$ is a polynomial algebra, as filtered objects we have  $\HHFil(A)\simeq \tau_\geq \, \HH(A)$. In particular, the computations in \cref{thmhkr}-(d) and the definition of the t-structure in $\Qcoh(\B\Filcircle)$ show that the restriction $\HHFil:\Nerve(\Poly{k})\to \Qcoh(\B\Filcircle)^{\dagger,<\infty}$ factors through the full subcategory $\Qcoh(\B\Filcircle)^{\dagger,<\infty}_{\geq 0}$. 

\noindent Notice that if $\rmE\in \Qcoh(\B\Filcircle)^{\dagger,<\infty}_{\geq 0}$ then the associated Whitehead tower $\tau^{\Fil}_{\geq}\, \rmE$ is a $\bbZ^{\op}$-diagram which is constant starting from level $0$. In particular, it follows that

$$\colim_{i\in \bbZ^{\op}}\, (\tau^\Fil_{\geq i}\, \rmE)= \rmE$$

\noindent In other words, we have commutativity for the diagram

 \begin{equation}
     \label{comparisonBensmoothdiagramleft1}
     \xymatrix{
     &\Nerve(\Poly{k})\ar[d]^{\HHFil}\\
     &\Qcoh(\B\Filcircle)^{\dagger, <\infty}_{\geq 0}\ar@{=}[dl]\ar[d]^{\tau_{\geq }^{\Fil}} \\
    \Qcoh(\B\Filcircle)^{\dagger, <\infty}_{\geq 0}\,\, &\Filcomplete^{\mathrm{const}\geq 0}[\,\,\Qcoh(\B\Filcircle)^{\dagger, <\infty}_{\geq 0}\,\,]\ar[l]_-{\colim_{}}
     }
\end{equation}

\noindent where the colimit is taken with respect to the Whitehead filtration.

\end{remark}

\medskip

\begin{remark}
\mylabel{colimitpostnikovfixedpoints}
We observe that we can extend the commutative diagram \eqformula{comparisonBensmoothdiagramleft1}  to the extraction of fixed points for the circle action. More precisely, the \cref{notationfilteredobjectsconstantfromlevel0} guarantees the commutativity of the diagram

\begin{equation}
     \label{comparisonBensmoothdiagramleft2}
     \xymatrix{
     \Qcoh(\B\Filcircle)^{\dagger, <\infty}_{\geq 0}\ar[d]^{(-)^{\mathrm{h}\Filcircle}}&\ar[l]_-{\colim_{}}\Filcomplete^{\mathrm{const}\geq 0}[\,\,\Qcoh(\B\Filcircle)^{\dagger, <\infty}_{\geq 0}\,\,]\ar[d]^{(-)^{\mathrm{h}\Filcircle}_{levelwise}}\\
     \Filcomplete(\Mod_k)&\ar[l]_-{\colim_{}}\Filcomplete^{\mathrm{const}\geq 0}[\,\,\Filcomplete(\Mod_k)\,\,]
     }
\end{equation}

\noindent where the bottom map is the colimit with respect to the exterior filtration on the bottom right corner. 

\end{remark}

\medskip

We now recall the construction of the Beilinson $t$-structure on filtered objects:

\begin{theorem}(Beilison $t$-structure \cite[Thm 5.4]{MR3949030})
\mylabel{Beilinsont-structure}
The category of filtered objects $\Fil(\Mod_k)$ admits a t-structure $t_B$ where $\Fil(\Mod_k)_{\geq_B 0}$ is the full subcategory spanned by those filtered objects $\rmE$ such that $\gr^n\rmE\in \Mod_k^{\geq -n}$. If $\tau^B_{\geq n}\rmE$ denotes the truncation, its associated graded pieces are given by 
\begin{equation}
\label{gradedpiecesbeilisontruncation}
    \gr^i\tau^B_{\geq n}\rmE\simeq \tau_{\geq n-i}\, \gr^i(\rmE)
\end{equation}
\noindent where on the r.h.s we have the standard t-structure on $\Mod_k$. The heart of this t-structure $\Fil(\Mod_k)^{\heartsuit_B}$ is equivalent to the abelian category of chain complexes via the construction sending a filtered complex $\rmE\in \Fil(\Mod_k)^{\heartsuit_B}$ to the chain complex

\begin{equation}
\label{formulapibeilison}
[\cdots\to \underbrace{\mathrm{H}_n(\gr^0\rmE)}_{0}\to \underbrace{\mathrm{H}_{n-1}(\gr^1\rmE)}_{-1}\to \underbrace{\mathrm{H}_{n-2}(\gr^2\rmE)}_{-2}\to \cdots]
\end{equation}

\noindent and the maps are the boundary maps of the cofiber sequences

$$
\gr^{i+1}(\rmE)= \rmE_{i+1}/ \rmE_{i+2} \to \rmE_i/ \rmE_{i+2} \to \rmE_i/ \rmE_{i+1}=\gr^i(\rmE)
$$

\end{theorem}

\medskip

\begin{notation}
\mylabel{whiteheadBeilisontower}
We denote by $\tau^B_{\geq}: \Fil(\Mod_k)\to \Fil(\Fil(\Mod_k)$ the Whitehead tower for the Beilinson $t$-structure.

Notice that since the Beilinson $t$-structure is not left complete (see \cite[5.3]{MR3949030}), the Beilison-Whitehead tower is not automatically complete. However, by \cite[Lemma 3.2]{1808.05246}, if $\rmE\in \Filcomplete$ then each truncation $\tau^B_{\geq n}\, \rmE$ is also complete and since complete filtered objects are stable under limits ( \cref{remark-completemodulesleftorthogonal}) so is the filtered object $\underset{n}{\lim}\,\tau^B_{\geq n}\, \rmE$.

Therefore, we have a well-defined functor $\tau^B_{\geq}:\Filcomplete\to \Fil(\Filcomplete(\Mod_k))$.

Notice also that if $\rmE$ is in (\cref{notationfilteredobjectsconstantfromlevel0}) $\Filcomplete^{\mathrm{const}\geq 0}(\Mod_k)$, then the formula \eqformula{gradedpiecesbeilisontruncation} implies that every truncation $\tau^B_{\geq n}\rmE$ is also in $\Filcomplete^{\mathrm{const}\geq 0}(\Mod_k)$.
\end{notation}

\begin{remark}
\mylabel{doublespeed}
In \cite{1808.05246} Antieau computes the Whitehead tower of $\mathrm{F}_{\mathrm{HKR}}^{\bullet}\HCminus{A}$ for $A$ smooth, (see \cref{Bennotation})  with respect to the Beilinson $t$-structure, showing that under the equivalence between $\Fil(\Mod_k)^{\heartsuit_B}$ and chain complexes (see \cite[Thm 5.4-(3)]{MR3949030} for this equivalence) given by the formula \eqformula{formulapibeilison}, for $n\geq 0$, the filtered object $\pi_{2n}^B(\mathrm{F}_{\mathrm{HKR}}^{\bullet}\HCminus{A})$ corresponds to the chain complex $\Omega_A^{\geq n}$ of the \cref{truncatedderhamcomplex}. Moreover, $\pi_{n}^B(\mathrm{F}_{\mathrm{HKR}}^{\bullet}\HCminus{A})$ is zero for positive odd $n$'s. For $n<0$ one can use the formula \eqformula{formulapibeilison} to show that  $\pi_{n}^B(\mathrm{F}_{\mathrm{HKR}}^{\bullet}\HCminus{A})=\pi_{0}^B(\mathrm{F}_{\mathrm{HKR}}^{\bullet}\HCminus{A})$ if $n$ is even, and zero if $n$ is odd (see also \cite[Example 2.4]{1808.05246} or the projection formula for invariants on the trivial circle action discussed in the \cref{trivialrepresentationgradedandunderlying} below for $\B\circle\simeq \B\Filcircleund$.)

\noindent For this reason one is allowed to run the Whitehead tower for the Beilison $t$-structure at a double-speed, since $\tau^B_{\geq 2n+2}(\mathrm{F}_{\mathrm{HKR}}^{\bullet}\HCminus{A})\to \tau^B_{\geq 2n+1}(\mathrm{F}_{\mathrm{HKR}}^{\bullet}\HCminus{A})$ are equivalences $\forall n$.
\end{remark}

\medskip
Although the Beilinson-Whitehead tower is not complete for a general filtered object (see \cref{whiteheadBeilisontower}), it is in the case that concerns us:

\medskip

\begin{proposition}
\mylabel{BeilisontowerofBeniscomplete}
Let $A\in \Nerve(\Poly{k})$. Then, the double-speed Beilinson-Whitehead tower of the filtered object $\mathrm{F}_{\mathrm{HKR}}^{\bullet}\HCminus{A}$,

\begin{equation}
\label{twist3}
[\cdots \to \underbrace{\tau^{B}_{\geq 2n}\mathrm{F}^\bullet_{\mathrm{HKR}}\HCminus{A}}_{n}\to \cdots \to \underbrace{\tau^{B}_{\geq 2}\mathrm{F}^\bullet_{\mathrm{HKR}}\HCminus{A}}_{1}\to \underbrace{\tau^{B}_{\geq 0}\mathrm{F}^\bullet_{\mathrm{HKR}}\HCminus{A}}_{0}\to \cdots]
\end{equation}

\noindent is complete.
\begin{proof}
The discussion in the \cref{doublespeed} shows that the objects $\pi_{n}^B(\mathrm{F}_{\mathrm{HKR}}^{\bullet}\HCminus{A})$ are non-zero only for $n=0,2,4,..., 2d$, where $d=dim\, A$. This implies that the double-speed Beilison-Whitehead tower \eqformula{twist3} is constant for $n\geq d+1$. Therefore its limit coincides with the filtered object $\tau^{B}_{\geq 2d+2}\mathrm{F}^\bullet_{\mathrm{HKR}}\HCminus{A}$. We argue that this object is zero in $\Fil(\Mod)$.
Since $\mathrm{F}^\bullet_{\mathrm{HKR}}\HCminus{A}$ is complete (see \cref{Bennotation}), by the discussion in the \cref{whiteheadBeilisontower}, $\tau^{B}_{\geq 2d+2}\mathrm{F}^\bullet_{\mathrm{HKR}}\HCminus{A}$ is complete. Therefore by the discussion in \cref{remark-completemodulesleftorthogonal}, to prove that it is zero it is enough to show all its graded pieces vanish.


\noindent Using \eqformula{gradedpiecesbeilisontruncation} we have

$$\gr^i\tau^{B}_{\geq 2d+2}\mathrm{F}^\bullet_{\mathrm{HKR}}\HCminus{A}\simeq \tau_{\geq 2d+2-i}\gr^i(\mathrm{F}^\bullet_{\mathrm{HKR}}\HCminus{A}) $$

\noindent But the construction of $\mathrm{F}^\bullet_{\mathrm{HKR}}\HCminus{A}$ (see \cref{Bennotation} and \cite[Example 2.4]{1808.05246} ) shows that

$$
\gr^i(\mathrm{F}^\bullet_{\mathrm{HKR}}\HCminus{A})= \begin{cases}
    (\Omega^i_A[i])^{\mathrm{h}\circle}\simeq \Omega^i_A[i]\otimes \cochains(\B\circle,k)\simeq  \underset{j\geq 0}{\bigoplus}\Omega^i_A[i-2j] , & \text{if $0\leq i\leq d$}.\\
    0\,, & \text{otherwise}
  \end{cases} 
$$

\noindent Therefore, all graded pieces are automatically zero outside the range $0\leq i\leq d$. Inside the range, we find $2d+2-i\geq i+2> i$ so that the truncations also vanish. This concludes the proof.

\end{proof}
\end{proposition}

\medskip

\noindent The following is our first comparison result: the smooth polynomial case.

\begin{proposition}
\mylabel{comparisonBenresult1}
\noindent Consider the commutative diagram \eqformula{comparisonBensmoothdiagram1}:

\begin{equation}
\label{compatiblewithBeilinsontower}
    \xymatrix{
    \Nerve(\Poly{k})\ar[d]_{(\tau^{\Fil}_{\geq}\,\HHFil)^{\mathrm{h}\Filcircle}_{lvl}}\ar[drr]^{\mathrm{F}^\bullet_{\mathrm{HKR}}\HCminus{-}}&&\\
    \Filcomplete(\Filcomplete(\Mod_k))\ar[rr]^{\Fil(\colim)}&& \Filcomplete(\Mod_k)
    }
\end{equation}

\noindent Then the universal property of the double-speed Beilinson-Whitehead tower $\tau_{\geq \, 2\ast}^{B}: \Fil(\Mod_k)\to \Fil(\Fil(\Mod_k))$ renders the commutativity of 

\begin{equation}
\label{compatiblewithBeilinsontower2}
    \xymatrix{
    \Nerve(\Poly{k})\ar[d]_{(\tau^{\Fil}_{\geq}\,\HHFil)^{\mathrm{h}\Filcircle}_{lvl}}\ar[drr]^{\mathrm{F}^\bullet_{\mathrm{HKR}}\HCminus{-}}&&\\
    \Filcomplete(\Filcomplete(\Mod_k))\ar[d]^{\sim}_{\mathrm{twist}}&& \Filcomplete(\Mod_k)\ar[dll]_{\tau_{\geq2\ast-}^{B}}\\
    \Filcomplete(\Filcomplete(\Mod_k))&&
    }
\end{equation}

\noindent where the equivalence $\Fil(\Fil)\simeq \Fil(\Fil)$ is the twist interchanging the order of the filtrations.
\begin{proof}
\noindent By the universal property of the Whitehead tower, it is enough to show that for each polynomial algebra $A$, the composition 

$$
\mathrm{twist}\, \circ \, (\tau^{\Fil}_{\geq}\,\HHFil(A))^{\mathrm{h}\Filcircle}_{lvl}
$$

\noindent has the properties described in \cref{construction-whiteheadtowerfunctor} in $\Fil(\Fil(\Mod_k))$. The bifiltered object $(\tau^{\Fil}_{\geq}\,\HHFil(A))^{\mathrm{h}\Filcircle}_{lvl}$  is given by the sequence of filtered objects

\begin{equation}
\label{twist1}
[\cdots \to (\tau^{\Fil}_{\geq i}\,\HHFil(A) )^{\mathrm{h}\Filcircle}_{}\to \cdots \to (\tau^{\Fil}_{\geq 1 }\,\HHFil(A) )^{\mathrm{h}\Filcircle}_{}\to (\tau^{\Fil}_{\geq 0 }\,\HHFil(A) )^{\mathrm{h}\Filcircle}_{} = \cdots]
\end{equation}

\noindent To describe its image under the twist operation, let us denote by $\rmF_n$ the filtered object obtained by extracting the level $n$ in \eqformula{twist1}:

$$
\rmF_n:=[ \cdots \to [(\tau^{\Fil}_{\geq i}\,\HHFil(A) )^{\mathrm{h}\Filcircle}_{}]_n\to \cdots \to [(\tau^{\Fil}_{\geq 1 }\,\HHFil(A) )^{\mathrm{h}\Filcircle}_{}]_n\to [(\tau^{\Fil}_{\geq 0 }\,\HHFil(A) )^{\mathrm{h}\Filcircle}_{}]_n = \cdots]
$$

\noindent Then, the construction $\mathrm{twist}\, \circ \, (\tau^{\Fil}_{\geq}\,\HHFil(A))^{\mathrm{h}\Filcircle}_{lvl}$ is given by the bifiltered object $\rmF$

\begin{equation}
\label{twist2}
\rmF:=[\cdots \to \rmF_{n+1}\to \rmF_{n}\cdots \to \rmF_{1}\to \rmF_{0} \to \rmF_{-1}  \to \cdots]
\end{equation}

\medskip

\noindent Our goal is to show that  \eqformula{twist2} is equivalent to the double-speed Beilinson-Whitehead tower \eqformula{twist3}. We start with the computation of the graded pieces $\gr^n(\rmF)$. This is a filtered object and unfolding the construction of $\rmF$ we find that its $i$-th level is given by

$$
\gr^n(\rmF)_i\simeq \gr^n[(\tau^{\Fil}_{\geq i }\,\HHFil(A) )^{\mathrm{h}\Filcircle}]
$$

\noindent Furthermore, the base-change formula of  \eqformula{beckchevalley2} tells us that

$$
\gr[(\tau^{\Fil}_{\geq i }\,\HHFil(A) )^{\mathrm{h}\Filcircle}]\simeq [\gr(\tau^{\Fil}_{\geq i }\,\HHFil(A) )]^{\mathrm{h}\Filcirclegr}
$$

\noindent The compatibility with truncations of \eqformula{compatibilitywhiteheadgraded} in \cref{Whiteheadtowerrepcirclecommute}-(ii), tells us that 

$$
 [\gr(\tau^{\Fil}_{\geq i }\,\HHFil(A) )]^{\mathrm{h}\Filcirclegr}\simeq   [\tau_{\geq i}^{\gr}\,\gr(\tau^{\Fil}_{\geq i }\,\HHFil(A) )]^{\mathrm{h}\Filcirclegr}
$$

\medskip

\noindent and one now observes that the canonical map $\tau_{\geq i}^{\Fil}\HHFil(A)\to \HHFil(A)$ induces an equivalence

$$
\tau_{\geq i}^{\gr}\,\left(\gr(\tau^{\Fil}_{\geq i }\,\HHFil(A) \right)\simeq \tau_{\geq i}^{\gr}\,\left( \gr (\HHFil(A)\right)
$$

\medskip

\noindent Indeed, using the \cref{thmhkr}-(b) and the \cref{remark-smoothcasefiltrationisPostnikov} for the smooth polynomial case, this map reads as

$$
\tau_{\geq i}^{\gr}\,\left( \underset{j\geq i}{\bigoplus} \Omega^j_A[j](j)\right)\, \to \tau_{\geq i}^{\gr}  \left(\underset{j\geq 0}{\bigoplus} \Omega^j_A[j](j)\right)
$$

\noindent which we can check to be an equivalence using the definition of $\tau_{\geq i}^{\gr}$ as weight-wise truncations, togehter with the fact that for a smooth polynomial algebra, the $\Omega^i_A$ are all concentrated in homological degree $0$. All summarized, we now have obtained

$$
\gr^n(\rmF)_i\simeq \gr^n\left[\left(\tau_{\geq i}^{\gr}\DR(A)\right)^{\mathrm{h}\Filcirclegr}\right]
$$

\medskip

\noindent with $\tau_{\geq i}^{\gr}\DR(A) \simeq \bigoplus_{j\geq 0} \tau_{\geq i}\Omega^j[j]\simeq \bigoplus_{j\geq i} \Omega^j[j]$, and we find

$$
\gr^n(\rmF)_i\simeq \gr^n\left[\left(\bigoplus_{j\geq i} \Omega^j[j]\right)^{\mathrm{h}\Filcirclegr}\right]
$$

\noindent Let us first deal with the case $n\geq 0$.\\

\noindent We find that for $i\leq 0$, $\gr^n(\rmF)$ is constant and that for $i=0$, we have $\gr^n(\rmF)_0\simeq \gr^n[\DR(A)^{\mathrm{h}\Filcirclegr}]\simeq \Omega_A^{\geq n}[2n]$ (see \cref{truncatedderhamcomplex} and \cref{thmhkr}-(c)). We compute $\gr^n(\rmF)_1$ using the cofiber sequence of graded objects

$$
\xymatrix{
\ar[d]\gr^n(\rmF)_1\simeq \gr^n[(\bigoplus_{j\geq 1} \Omega^j[j])^{\mathrm{h}\Filcirclegr}]\ar[rr]&& \gr^n[(\bigoplus_{j\geq 0} \Omega^j[j])^{\mathrm{h}\Filcirclegr}]\simeq \Omega_A^{\geq n}[2n] \ar[d]\\
0\ar[rr]&& \gr^n[(\Omega_A^0)^{\mathrm{h}\Filcirclegr}]
}
$$

\noindent where $\Omega_A^0$ is seen as a graded module pure of weight $0$ with a trivial action of $\Filcirclegr$. Using the projection formula discussed in the  \cref{trivialrepresentationgradedandunderlying}, we obtain $(\Omega_A^0)^{\mathrm{h}\Filcirclegr}\simeq \Omega_A^0\otimes \cochains(\B\Filcircle, \structuresheaf)^{\gr}$. After the \cref{groupcohomologyfilteredcircle}, we find $ \Omega_A^0\otimes \cochains(\B\Filcircle, \structuresheaf)^{\gr}\simeq \bigoplus_{k\geq 0}\,\Omega_A^0(-k)[-2k]$ so that

$$\gr^n[(\Omega_A^0)^{\mathrm{h}\Filcirclegr}] = \begin{cases}
    0\,, & \text{if $n>0$}.\\
    \Omega_A^0(n)[2n], & \text{if $n\leq 0$}
  \end{cases}  $$

\noindent and

$$\gr^n(\rmF)_1= \begin{cases}
    \gr^n(\rmF)_0\simeq \Omega_A^{\geq n}[2n]\,, & \text{if $n>0$}.\\
    \fiber(\,\Omega_A^{\geq 0}\to \Omega_A^0)\simeq \,\Omega_A^{\geq 1}\,, & \text{if $n=0$}
  \end{cases}  $$
  
\noindent Recall that by \cref{thmhkr-stackversion}-(2) $\Omega_A^1[1]$ is in weight $1$ in $\DR(A)$. A similar computation as above shows that for a fixed $i\geq 0$, and the graded object of pure weight $i$, $\Omega_A^i[i](i)$ with a trivial action of $\Filcirclegr$  we find

$$
\gr^n[(\Omega_A^i[i](i))^{\mathrm{h}\Filcirclegr}]\simeq (\Omega_A^i[i](i))\otimes \cochains(\B\Filcircle, \structuresheaf)^{\gr} = \begin{cases}
    0\,, & \text{if $n>i$}.\\
    \Omega_A^i[-i+2n], & \text{if $0\leq n\leq i$}
  \end{cases} 
$$

\noindent By induction on $i$, using the cofiber sequences

\begin{equation}
    \label{keycofibersequenceinduction2}
\xymatrix{
\ar[d]\gr^n(\rmF)_{i+1}\simeq \gr^n[(\bigoplus_{j\geq i+1} \Omega^j[j])^{\mathrm{h}\Filcirclegr}]\ar[rr]&& \gr^n[(\bigoplus_{j\geq i} \Omega^j[j])^{\mathrm{h}\Filcirclegr}]\ar[d]\\
0\ar[rr]&& \gr^n[(\Omega_A^i[i](i))^{\mathrm{h}\Filcirclegr}]
}
\end{equation}

\noindent for $n\geq 0$, we find the filtered object 

$$
\gr^n(\rmF)=[\cdots \to \underbrace{\Omega_A^{\geq n+2}[2n]}_{n+2} \to \underbrace{\Omega_A^{\geq n+1}[2n]}_{n+1} \to \underbrace{\Omega_A^{\geq n}[2n]}_{n} = \underbrace{\Omega_A^{\geq n}[2n]}_{n-1} = \cdots  = \underbrace{\Omega_A^{\geq n}[2n]}_{0} = \cdots ]
$$

\noindent Following the discussion in the \cref{doublespeed}, we see that we have an equivalence of filtered objects 
$$
\gr^n(\rmF)\simeq \pi_{2n}^B(\mathrm{F}^\bullet_{\mathrm{HKR}}\HCminus{A})[2n] \,\,\, \text{ if $n\geq 0$}
$$

\medskip

\noindent For $n<0$ we have 
$$
\gr^{n}(\rmF) \simeq \pi_{0}^B(\mathrm{F}^\bullet_{\mathrm{HKR}}\HCminus{A})[2n]\simeq \Omega^{\geq 0}_A[2n]  \,\,\, \text{ if $n< 0$}
$$

\noindent Indeed, this follows from the use of the same cofiber sequences \eqformula{keycofibersequenceinduction2}, but this time starting from the formula \eqformula{negativegradedpieces} for the negative graded pieces and comparing with the discussion in the \cref{doublespeed}.\\

\noindent Since both \eqformula{twist2}  and \eqformula{twist3} are complete, this shows they are the same double-speed Beilinson-Whithead tower.

\end{proof}
\end{proposition}

\medskip

\begin{lemma}
\mylabel{lemmacohomologytrivialaction}
Let $\pi:\B\Filcircle\to \Filstack$ be the canonical projection and let $\rmE\in \Qcoh(\Filstack)^{<\infty}$ (homologically bounded above). Then the canonical map

\begin{equation}
\label{mapprojectionformula}
 (\pi^\ast (\rmE))^{\mathrm{h}\Filcircle}\simeq \pi_\ast \, \pi^\ast (\rmE)\to \pi_\ast (\structuresheaf)\otimes \rmE \simeq  \cochains(\B\Filcircle, \structuresheaf)\otimes \rmE
\end{equation}
\noindent is an equivalence in $\Qcoh(\Filstack)$.
\begin{proof}
Following \cref{left-complete-t-structure}-(b), $\pi_\ast$  is left-$t$-exact. Moreover, since the t-structure in $\Qcoh(\B\Filcircle)$ is determined by pullback along the atlas $e: \Filstack\to \B\Filcircle$ (see \cref{tstructureQcohBfilteredcircle}), we have $e^\ast\circ \pi^\ast =\mathrm{id}$ which implies that $\pi^\ast$ is $t$-exact.

\noindent Moreover, since for every $N$, $\Qcoh(\Filstack)_{\leq N}\subseteq \Qcoh(\Filstack)$ is stable under filtered colimits (because homology groups commute with filtered colimits), so is $\Qcoh(\B\Filcircle)_{\leq N}\subseteq \Qcoh(\B\Filcircle)$. This shows that $\pi^\ast:\Qcoh(\Filstack)_{\leq N}\to \Qcoh(\B\Filcircle)_{\leq N}$ preserves filtered colimits. Finally, using \cite[A.1.3 - part 2]{1402.3204} (evoking our flat atlas by affines for $\B\Filcircle$ of \cref{tstructureQcohBfilteredcircle}) we deduce that  $\pi_\ast:\Qcoh(\B\Filcircle)_{\leq N}\to \Qcoh(\Filstack)_{\leq N}$ also commutes with filtered colimits. In parallel, is is easy to notice that the operation $\rmE\mapsto \cochains(\B\Filcircle, \structuresheaf)\otimes \rmE$ is well-defined as a functor 
$$\Qcoh(\Filstack)_{\leq N}\to \Qcoh(\Filstack)_{\leq N}$$ 
\noindent  and also commutes with filtered colimits. Since the family $\structuresheaf(n)\in \Qcoh(\Filstack)$ generate $\Qcoh(\Filstack)$, we are reduced to show that \eqformula{mapprojectionformula} is an equivalence when $\rmE=\structuresheaf(n)$. But since the $\structuresheaf(n)$ are invertible objects (therefore dualizable), we are reduced to show the claim for $\rmE=\structuresheaf$, where this is a tautology.

\end{proof}
\end{lemma}

\medskip

\begin{remark}
\mylabel{trivialrepresentationgradedandunderlying}
The arguments and conclusion  of \cref{lemmacohomologytrivialaction} apply mutatis-mutantis to the two projections
$\B\Filcirclegr\to \B\Gm{}$ and $\B\Filcircleund\to \Spec \,\integerslocalp$.
\end{remark}

\medskip

\vspace{1cm}
This concludes the comparison of the two constructions in the smooth polynomial case. We now embark on a new digression to explain how to extend this comparison to all $\SCRings{k}$. This is not a straightforward consequence of \cref{comparisonBenresult1} because none of the functors used in the diagram \eqformula{compatiblewithBeilinsontower}, when Kan extended, give rise to the correct construction, as the following remark intends to illustrate:

\medskip

\begin{remark}
\mylabel{failureHCminusfilteredcommutesifted} The composition 
$$
\xymatrix{
\HCminFil:\SCRings{k}\ar[r]^{\HHFil}& \widehat{\Qcoh(\B\Filcircle)}\ar[r]^{(-)^{\mathrm{h}\Filcircle}}& \Filcomplete(\Mod_k)}
$$
\noindent does not commute with sifted colimits in general. Indeed, recall from the \cref{remark-completemodulesleftorthogonal} that $\Filcomplete(\Mod_k)$ is a presentable localization of $\Fil(\Mod_k)$ with a left adjoint to the inclusion $\Filcomplete(\Mod_k)\subseteq \Fil(\Mod_k)$ given by the completion functor $\widehat{(-)}$. Let then $A:K\to \SCRings{k}$ be a sifted diagram with $B:=\colim_{k\in K}\, A_k$. Then we a canonical map

$$
\underbrace{\colim_{k\in K}}_{\text{in }\,\,\, \Filcomplete(\Mod_k)} \HCminFil(A_k)\to \HCminFil(B)
$$

\noindent But by definition of a presentable localization, l.h.s is computed as the completion of the colimit in $\Fil(\Mod_k)$

$$
\underbrace{\colim_{k\in K}}_{\text{in }\,\,\, \Filcomplete(\Mod_k)} \HCminFil(A_k)\simeq \widehat{(\underbrace{\colim_{k\in K}}_{\text{in }\,\,\, \Fil(\Mod_k)} \HCminFil(A_k))}
$$

\noindent In other words, the commutation with sifted colimits is measured by fact the canonical map in $\Fil(\Mod_k)$

\begin{equation}
\label{covidattacksagain}
(\underbrace{\colim_{k\in K}}_{\text{in }\,\,\, \Fil(\Mod_k)} \HCminFil(A_k))\to \HCminFil(B)
\end{equation}

\noindent becomes an isomorphism after completion. But again using the \cref{remark-completemodulesleftorthogonal} we know that the map \eqref{covidattacksagain} becomes an equivalence after completion if and only if the induced map between the associated graded objects is an equivalence:

\begin{equation}
\label{covidattacksagain2}
[\underbrace{\colim_{k\in K}}_{\text{in }\,\,\, \Fil(\Mod_k)} \HCminFil(A_k))]^{\gr}\to \HCminFil(B)^{\gr}
\end{equation}

\noindent Even better, since $\gr$ commutes with all colimits, we only have to test the map of graded objects

\begin{equation}
\label{covidattacksagain3}
\underbrace{\colim_{k\in K}}_{\text{in }\,\,\, \Mod_k^{\bbZ-\gr}} \HCminFil(A_k))^{\gr}\to \HCminFil(B)^{\gr}
\end{equation}

\noindent But this poses a problem: in light of \cref{thmhkr}-(c), both sides are computed in terms of the completed truncated de Rham complex which as explained in \cref{derivedderrhamcomplex} contains an infinite product of wedge powers of the cotangent complex. But infinite products do not commute with sifted colimits in general. 
\end{remark}

\medskip

The idea to solve the problem posed in the \cref{failureHCminusfilteredcommutesifted} is to consider an extra filtration that disassembles the components of the infinite product found in the formula \eqformula{covidattacksagain3}.

\medskip

\begin{construction}(Skeletal Filtration)\mylabel{skeletalfiltration} 
\noindent Consider the map \eqformula{uinS1rep} of \cref{elementu}:
\begin{equation}
u:\pi^\ast(\structuresheaf(-1))[-2]\to \structuresheaf
\end{equation}
\noindent in $\Qcoh(\B\Filcircle)$. By abuse of notation, we will write again $\structuresheaf(-1)[-2]:=\pi^\ast(\structuresheaf(-1))[-2]$. Since $\structuresheaf(-1)$ is a line bundle, $u$ can also be seen as a map

$$
\structuresheaf\to \structuresheaf(1)[2]
$$

\noindent in $\Qcoh(\B\Filcircle)$. In what follows, we will consider the $\bbZ^\op$-sequence $\widetilde{\structuresheaf}$ in $\Qcoh(\B\Filcircle)$ given by

$$
\xymatrix{
\cdots \ar[r]&  \widetilde{\structuresheaf}_1=\structuresheaf\ar[r]^{=}& \widetilde{\structuresheaf}_0=\structuresheaf\ar[r]^-u& \widetilde{\structuresheaf}_{-1}=\structuresheaf (1)[2]\ar[r]^-u &\widetilde{\structuresheaf}_{-2}=\structuresheaf (2)[4]\ar[r]^-u &\cdots } 
$$

\noindent Since we have a canonical identification

$$
(\rmE)^{\mathrm{h}\Filcircle}:=\pi_\ast(\rmE)\simeq \RHom_{\B\Filcircle}(\structuresheaf, E)
$$

\noindent we can use $\widetilde{\structuresheaf}$ to obtain a lift $\widetilde{\pi_{\ast}}$ of $\pi_{\ast}$:

\begin{equation}
\label{liftskeletal}
    \xymatrix{
    \Qcoh(\B\Filcircle)\ar[r]^{\pi_\ast}\ar[d]^{\widetilde{\pi_\ast}}& \Qcoh(\Filstack)\\
    \Fil(\Qcoh(\Filstack))\ar[ur]^{\colim}
    }
\end{equation}

\noindent Informally, this lift sends $\rmE\in \Qcoh(\B\Filcircle)$ to

$$
[\cdots \to \RHom_{\B\Filcircle}(\structuresheaf(2)[4], E)\to \RHom_{\B\Filcircle}(\structuresheaf(1)[2], E)\to = (\rmE)^{\mathrm{h}\Filcircle}= (\rmE)^{\mathrm{h}\Filcircle}=\cdots]
$$

\noindent Moreover, since $\structuresheaf(t)[2t]$ are line bundles, we have 

$$
\RHom_{\B\Filcircle}(\structuresheaf(t)[2t], E)= (\rmE)^{\mathrm{h}\Filcircle}\otimes \structuresheaf(-t)[-2t]=: (\rmE)^{\mathrm{h}\Filcircle}(-t)[-2t]$$

\medskip

\noindent for $t\geq 0$, and the object $\widetilde{\pi_\ast}(E)$ becomes

\begin{equation}
    \label{skeletaltower}
[\cdots \to \underbrace{(\rmE)^{\mathrm{h}\Filcircle}(-2)[-4]}_{2}\to \underbrace{(\rmE)^{\mathrm{h}\Filcircle}(-1)[-2]}_{1}\to = \underbrace{(\rmE)^{\mathrm{h}\Filcircle}}_{0}= \underbrace{(\rmE)^{\mathrm{h}\Filcircle}}_{-1}=\cdots]
\end{equation}

\end{construction}

\medskip

\begin{construction}
\mylabel{usualskeletalfiltration} The extra filtration of the \cref{skeletalfiltration} is our version for the filtered circle of the usual skeletal filtration of the topological space $\B \circle$ by complex projective spaces $\bbC\mathbb{P}^n$ used in \cite[\S 4]{1808.05246} and \cite{lurie-K}. Indeed, we have $(\bbC)^{\mathrm{h}\circle}\simeq \cochains(\B\circle, \bbC)\simeq \bbC[u]$ where $u$ is a generator in homological degree $(-2)$. In particular, we find cofiber sequences

$$
\xymatrix{
\bbC[u][-2t]\ar[rr]^u \ar[d]&& \bbC[u]=(\bbC)^{\mathrm{h}\circle}\ar[d]\\ 
0\ar[rr]&& (\bbC)^{\mathrm{h}\Omega\, \bbC\mathbb{P}^{t-1}}\simeq \bbC[0]\oplus \bbC[-2]\oplus \cdots \oplus \bbC[-2(t-1)]
}
$$

\noindent More generally, if $\rmM\in \Qcoh(\B\circle)$, we define the skeletal filtration of $(\rmM)^{\mathrm{h}\circle}$ to be the $\bbZ^{op}$ sequence $\rmF_\bullet(\rmM)$ whose level $t\geq 0$ is given by

$$
\rmF_t(\rmM):= \fiber\, ((\rmM)^{\mathrm{h}\circle}\to (\rmM)^{\mathrm{h}\Omega \bbC\mathbb{P}^{t-1}})
$$

\noindent and is constant equal to $(\rmM)^{\mathrm{h}\circle}$ for $t\leq 0$. This construction provides a lifting of $(-)^{\mathrm{h}\circle}:\Qcoh(\B\circle)\to \Mod_k$ to filtered objects

\begin{equation}
\label{liftskeletaltopological}
    \xymatrix{
    \Qcoh(\B\circle)\ar[r]^-{(-)^{\mathrm{h}\circle}}\ar[d]_{\widetilde{(-)^{\mathrm{h}\circle}}}& \Mod_k\\
    \Fil(\Mod_k)\ar[ur]_-{\colim}
    }
\end{equation}
\end{construction}

\medskip

\begin{remark}
\mylabel{skeletalfiltrationscompatibleunderlyingobjects} The \cref{usualskeletalfiltration} and the \cref{skeletalfiltration} are compatible under the extraction of underlying objects. Indeed,  the commutativity of the diagram

$$
\xymatrix{
\B\Filcircle\ar[d]^{\pi}& \ar[l]_{1_a} \B\circle\ar[d]^{\pi^{\mathrm{u}}}\\
\Filstack \ar[d]^q & \ar[l]_{1} \Spec\, \integerslocalp\ar[dl]_e\\
\B \Gm{} 
}
$$

\noindent implies that 

$$
(1_a)^\ast \, (\structuresheaf(t)[2t])\simeq (1_a)^\ast\, \pi^\ast \,q^\ast (\integerslocalp(t)[2t])\simeq (\pi^{\mathrm{u}})^\ast \, e^\ast (\integerslocalp(t)[2t])
$$

\noindent but since $e$ is the atlas, $e^\ast$ amounts to forgetting the grading, so that 

$$(\pi^{\mathrm{u}})^\ast \, e^\ast (\integerslocalp(t)[2t])\simeq \structuresheaf[2t]$$.

\noindent It follows, since $1^\ast$ is symmetric monoidal, that  

$$
((\rmE)^{\mathrm{h}\Filcircle}(-t)[-2t])^{\mathrm{u}}\simeq  ((\rmE)^{\mathrm{h}\Filcircle})^{\mathrm{u}}[-2t]
$$

\noindent As in \eqformula{okforboundedabove}, if we assume that $\rmE\in \Qcoh(\B\Filcircle)^{<\infty}$,  we find 

$$
((\rmE)^{\mathrm{h}\Filcircle}(-t)[-2t])^{\mathrm{u}}\simeq  ((\rmE^\mathrm{u})^{\mathrm{h}\circle}[-2t]
$$

\medskip

\noindent In summary, this discussion establishes the commutativity of the diagram

\begin{equation}
\label{bifilteredcompatibleunderlyingspectrum}
    \xymatrix{
    \ar[d]^{\widetilde{\pi_\ast}}\Qcoh(\B\Filcircle)^{<\infty}\ar[r]^-{1_a^\ast} &\Qcoh(\B\circle)\ar[d]^{\widetilde{(-)^{\mathrm{h}\circle}}}\\
    \Fil(\Qcoh(\Filcircle))\ar[r]^-{\Fil(\, 1_a^\ast)}& \Fil(\Mod_k)
    }
\end{equation}

\end{remark}

\medskip

\begin{notation}
\mylabel{notationbifilteredHCminus}
We denote by $\HCminBiFil$ the composition
$$
\xymatrix{\SCRings{k}\ar[r]^-{\HHFil}& \Qcoh(\B\Filcircle)\ar[r]^-{\widetilde{\pi_\ast}}& \Fil(\Fil(\Mod_k))}
$$
\noindent providing $\HCminFil$ with the extra skeletal filtration of \cref{skeletalfiltration}.
\end{notation}

\medskip

\begin{proposition}
\mylabel{skeletalfiltrationcomplete}
The skeletal filtration of \cref{skeletalfiltration} is complete for any $\rmE\in \Qcoh(\B\Filcircle)_{>-\infty}$ (homologically bounded below). In particular, the functor $\HCminBiFil$ factors through bi-complete filtered modules:

$$\HCminBiFil:\SCRings{k}\to \Filcomplete(\Filcomplete(\Mod_k))$$

\begin{proof}
Let $\rmE\in \Qcoh(\B\Filcircle)_{>-\infty}$. We want to show that the limit of the tower \eqformula{skeletaltower} is zero. Let  us denote by $C_t$ the fiber in $\Qcoh(\B\Filcircle)$

$$
\xymatrix{
C_t\ar[d]\ar[r]&0\ar[d]\\
\structuresheaf\ar[r]^{u^t}& \structuresheaf(t)[2t]
}
$$

\noindent Unfolding the definitions, this gives us cofiber sequences

\begin{equation}
\label{diagramCt}
\xymatrix{
(\rmE)^{\mathrm{h}\Filcircle}(-t)[-2t] \ar[d]\ar[r]^{u^t}&(\rmE)^{\mathrm{h}\Filcircle}\ar[d]\\
0\ar[r]^{}& \RHom_{\B\Filcircle}(C_t, \rmE)
}
\end{equation}

\noindent so that by stability,  the limit of \eqformula{skeletaltower} vanishes if and only  the induced map

\begin{equation}
    \label{inducedmaptothelimittowerskeletal}
    (\rmE)^{\mathrm{h}\Filcircle}=\RHom_{\B\Filcircle}(\structuresheaf, \rmE)\to \lim_t\, \RHom_{\B\Filcircle}(C_t, \rmE)\simeq \RHom_{\B\Filcircle}(\colim_t \, C_t, \rmE)
\end{equation}

\noindent is an equivalence. Using the assumption that  $\rmE$ is homologically bounded below we are reduced to show the statement for $\rmE$ in the heart of the structure of $\Qcoh(\B\Filcircle)$. Indeed, since the $t$-structure is left-$t$-complete (see \cref{left-complete-t-structure})-(c)), every object $\rmE$ is the limit of its Postnikov tower
$\rmE\simeq \lim_n \tau^{\Fil}_{\leq n} \rmE$. In particular, if we assume that $\rmE$ is bounded below, there exists $N\in \mathbb{Z}$ such that 
$\rmE\simeq \lim_{n\geq N}\, \tau^{\Fil}_{\leq n} \rmE$ with $\tau^{\Fil}_{\leq N} \rmE = (\pi_N^{\Fil}\, \rmE )[N]$. Since both sides of \eqformula{inducedmaptothelimittowerskeletal} preserve limits in $\rmE$, and by the nature of the Postnikov tower, by induction we are reduced to show that \eqformula{inducedmaptothelimittowerskeletal} is an equivalence when $\rmE$ is in $\Qcoh(\B\Filcircle)^{\heartsuit}$. But in this case the action of $\Filcircle$ must be trivial and we can use the \cref{lemmacohomologytrivialaction} and the fact that $\otimes$ is compatible with colimits separetely in each variable, to establish that the cofiber sequence \eqformula{diagramCt} is equivalent to

\begin{equation}
\label{diagramCt2}
\xymatrix{
\rmE\otimes \cochains(\B\Filcircle, \structuresheaf)(-t)[-2t] \ar[d]\ar[r]^-{u^t}&\rmE\otimes \cochains(\B\Filcircle, \structuresheaf)\ar[d]\\
0\ar[r]^{}& \rmE\otimes \RHom_{\B\Filcircle}(C_t, \structuresheaf)
}
\end{equation}
\noindent in filtered objects. Therefore, to conclude the proof it will be enough to show that if $\rmE\in \Fil(\Mod_k)^{\heartsuit_{\mathrm{std}}}$ (the levelwise standard $t$-structure), then the the canonical map of filtered objects

\begin{equation}
    \label{canonicalmapbetweenlimits}
    \rmE\otimes \cochains(\B\Filcircle, \structuresheaf)\to \underset{t}{\lim}\,\rmE\otimes \RHom_{\B\Filcircle}(C_t, \structuresheaf)
\end{equation}

\noindent is an equivalence of filtered objects. For this purpose we start with the observation that the since the standard $t$-structure in $\Mod_k$ is both left and right complete (see \cite[7.1.1.13]{lurie-ha}), so is the levelwise $t$-structure on $\Fil(\Mod_k)$. Therefore, if we denote by $\pi^{\Fil}_i$ the homotopy groups with respect to the levelwise $t$-structure in $\Fil(\Mod_k)$, to show that \eqformula{canonicalmapbetweenlimits} is an equivalence in $\Fil(\Mod_k)$ it is enough to show that it for every $i$ the induced map

\begin{equation}
    \label{canonicalmapbetweenlimits2}
    \pi^{\Fil}_{i}(\rmE\otimes \cochains(\B\Filcircle, \structuresheaf))\to \pi^{\Fil}_{i}(\underset{t}{\lim}\,\rmE\otimes \RHom_{\B\Filcircle}(C_t, \structuresheaf))
\end{equation}

\noindent is an isomorphism. By the conclusion of \cref{groupcohomologyfilteredcircle}, we know that as a filtered object $\cochains(\B\Filcircle, \structuresheaf)$ decomposes as a direct sum of filtered objects

\begin{equation}
\label{cohomologyBS1filtereddirectsum}
    \cochains(\B\Filcircle, \structuresheaf))\simeq \underset{p\geq 0}{\bigoplus} \integerslocalp(-p)[-2p]
    \end{equation}

\noindent This formula, together with the cofiber-sequence \eqformula{canonicalmapbetweenlimits} when $\rmE=\structuresheaf$, tells us that, as filtered objects, we have a direct sum decomposition

\begin{equation}
\label{cohomologyCPnformuladirectsum}
\RHom_{\B\Filcircle}(C_t, \structuresheaf))\simeq \, \bigoplus^{t-1}_{p\geq 0} \integerslocalp(-p)[-2p]
\end{equation}

\noindent Combining the formulas \eqformula{cohomologyBS1filtereddirectsum} and \eqformula{cohomologyCPnformuladirectsum} and the compatibility of the Day tensor products with direct sums in $\Fil(\Mod_k)$, the morphism \eqformula{canonicalmapbetweenlimits} can be written as

\begin{equation}
    \label{canonicalmapbetweenlimits3}
    \underset{p\geq 0}{\bigoplus} \rmE(-p)[-2p]\to \underset{t}{\lim}\, \, \bigoplus^{t-1}_{p\geq 0} \rmE(-p)[-2p]
\end{equation}

\noindent By definition of $\pi^{\Fil}_{i}$ and the assumption that $\rmE = \pi^{\Fil}_{0}(\rmE)$, we find

\begin{equation}
\label{formulapifilteredheart}
 \pi^{\Fil}_{i}(\underset{p\geq 0}{\bigoplus} \rmE(-p)[-2p])\simeq \underset{p\geq 0}{\bigoplus} \,\pi^{\Fil}_{i}(\rmE(-p)[-2p]))\simeq \begin{cases}
    \rmE(-p) & \text{for } i=-2p\\
    0, & \text{ otherwise } 
    \end{cases}
\end{equation}

\noindent To conclude we have to compute the objects $\pi^{\Fil}_{i}(\underset{t}{\lim}\, \, \bigoplus^{t-1}_{p\geq 0} \rmE(-p)[-2p])$. For this purpose we use the limit spectral sequence. Since $\Fil(\Mod_k)$ is a stable $\infty$-category and the sequence

$$t\mapsto \underset{t}{\lim}\, \, \rmE\otimes \RHom_{\B\Filcircle}(C_{t}, \structuresheaf))$$

\noindent vanishes (by definition) for $t<0$, we can use a dual version of \cite[1.2.2.14]{lurie-ha} providing a spectral sequence with first page given by

$$
\pi^{\Fil}_{t+q}(\fiber\,[\,\,\rmE\otimes\RHom_{\B\Filcircle}(C_{t}, \structuresheaf))\to \rmE\otimes \RHom_{\B\Filcircle}(C_{t-1}, \structuresheaf)\,\,]\, )$$

\medskip

\noindent and converging to $\pi^{\Fil}_{t+q}(\underset{t}{\lim}\, \, \bigoplus^{t-1}_{p\geq 0} \rmE(-p)[-2p]))$. But an immediate computation with the formula \eqformula{cohomologyCPnformuladirectsum} shows that 

$$
\fiber\,[\rmE\otimes\RHom_{\B\Filcircle}(C_{t}, \structuresheaf))\to \rmE\otimes \RHom_{\B\Filcircle}(C_{t-1}, \structuresheaf)]\,\,\simeq \rmE(-t)[-2t]
$$

\noindent 

\noindent and the spectral sequence becomes

$$
\pi^{\Fil}_{t+q}(\rmE(-t)[-2t])\implies \,\,\pi^{\Fil}_{t+q}(\underset{t}{\lim}\, \, \bigoplus^{t-1}_{p\geq 0} \rmE(-p)[-2p]))
$$

\medskip

\noindent But since $\rmE$ is in the heart, we have  

$$\pi^{\Fil}_{t+q}(\rmE(-t)[-2t])\simeq \, \begin{cases}
    \rmE(-t) & \text{for } t+q=-2t\\
    0, & \text{ otherwise } 
    \end{cases}$$
    
\medskip

\noindent and therefore the spectral sequence degenerates at the first page and we find

$$\pi^{\Fil}_{i}(\underset{t}{\lim}\, \, \bigoplus^{t-1}_{p\geq 0} \rmE(-p)[-2p]))= \begin{cases}
    \rmE(-t) & \text{for } i=-2t\\
    0, & \text{ otherwise } 
    \end{cases}$$

\noindent thus coinciding with \eqformula{formulapifilteredheart}. This concludes the proof.

\end{proof}

\end{proposition}

\medskip

\begin{remark}
\mylabel{skeletalfiltrationiscompleteusualcircle}
The argument in the proof of \cref{skeletalfiltrationcomplete} can also be used to show that the functor $\widetilde{(-)^{\mathrm{h}\circle}}$ in \cref{usualskeletalfiltration} also lands in complete filtrations. See the proof of  \cite[Lemma 4.2]{1808.05246}. 
\end{remark}

\medskip

We will now compute the bigraded pieces $\HCminBiFil$. We start with a simple observation:

\begin{remark}
\mylabel{bifilteredassociatedgraded}
The two ways of extracting the associated bigraded object from a biltered object, coincide. More precisely, we have a commutative diagram 
\begin{equation}
  \label{compatiblebigradedtwoways}  \xymatrix{
  \Fil(\,\Fil(\Mod_k)\,)\ar[rr]^{\Fil(\,\gr\,)}\ar[d]^{\gr(\,-\,)}&& \Fil(\,\Mod_k^{\bbZ-\gr}\,)\ar[d]^{\gr(\,-\,)}&\\
    (\,\Fil(\Mod_k)\,)^{\bbZ-\gr}\ar[rr]^-{(\,\gr\,)^{\bbZ-\gr}}&& (\,\Mod_k^{\bbZ-\gr}\,)^{\bbZ-\gr}\ar[r]^-{\sim}& \Mod_k^{\bbZ\times \bbZ-\gr}
    }
\end{equation}
To check this it is enough to see what happens for each bigraded piece $(i,n)$: if $\rmE\in \Fil(\Fil(\Mod_k))$ with $i$ the index of exterior filtration and $n$ for the interior, we find

$$
\text{down-left composition} = \cofiber (\rmE^{n+1}_{i}/\rmE^{n+1}_{i+1}\to \rmE^{n}_{i}/\rmE^{n}_{i+1})
$$

$$
\text{up-right composition} = \cofiber (\rmE^{n}_{i+1}/\rmE^{n+1}_{i+1}\to \rmE^{n}_{i}/\rmE^{n+1}_{i})
$$

\noindent It is an easy exercice to use concatenation of pushouts in $\Mod_k$ to check that the two agree canonically.

\end{remark}

\begin{remark}
\mylabel{gradedpiecesskeletalfiltration}Let us compute the graded pieces of the skeletal filtration of \cref{skeletalfiltration}. Given $\rmE\in \Qcoh(\B\Filcircle)$, the bifiltered object $\widetilde{\pi_\ast}(\rmE)$ is given by a diagram of filtered objects

$$
[\cdots \to \underbrace{\rmE^{\mathrm{h}\Filcircle}(-t-1)[-2(t+1)]}_{t+1}\to  \underbrace{\rmE^{\mathrm{h}\Filcircle}(-t)[-2t]}_{t}\to \cdots \to \underbrace{\rmE^{\mathrm{h}\Filcircle}}_{0}= \rmE^{\mathrm{h}\Filcircle}=\cdots  ]
$$

\noindent In particular, the  associated graded piece of degree $t$ with respect to the skeletal filtration $\gr_{\mathrm{sk}}^t(\widetilde{\pi_\ast}(\rmE))$, is the filtered object  given by the cofiber in $\Fil(\Mod_k)$ 

\begin{equation}
\label{cofibersequenceskeletalgraded}
\xymatrix{
\ar[d]\rmE^{\mathrm{h}\Filcircle}(-t-1)[-2(t+1)]\ar[r]&\rmE^{\mathrm{h}\Filcircle}(-t)[-2t]\ar[d]\\
0\ar[r]& \gr_{\mathrm{sk}}^t(\widetilde{\pi_\ast}(\rmE))
}
\end{equation}

\noindent The filtered object $\gr_{\mathrm{sk}}^t(\widetilde{\pi_\ast}(\rmE))$ has its own internal associated graded $\gr(\gr_{\mathrm{sk}}^t(\widetilde{\pi_\ast}(\rmE)))$ that can be computed as follows: since $\gr:\Fil(\C)\to \C$ commutes with all colimits, the cofiber sequence \eqformula{cofibersequenceskeletalgraded} produces a cofiber sequence in graded objects

\begin{equation}
\label{cofibersequencebigradedskeletal}
\xymatrix{
\ar[d]\gr(\rmE^{\mathrm{h}\Filcircle}(-t-1)[-2(t+1)])\ar[r]&\gr(\rmE^{\mathrm{h}\Filcircle}(-t)[-2t])\ar[d]\\
0\ar[r]& \gr(\gr_{\mathrm{sk}}^t(\widetilde{\pi_\ast}(\rmE)))
}
\end{equation}

\noindent Finally, we notice that we have a canonical equivalence in graded objects

$$
\gr(\rmE^{\mathrm{h}\Filcircle}(-t)[-2t])\simeq \gr(\rmE^{\mathrm{h}\Filcircle})\otimes \gr(\structuresheaf(-t)[-2t]))\simeq \gr(E)^{\mathrm{h}\Filcirclegr}(-t)[-2t]
$$

\noindent that follows from the base-change formula \eqformula{beckchevalley2} together with the fact that $\gr$ is identified with the pullback along $0:\B\Gm{}\to \Filstack$ (see \cref{prop:gradingsandstacks}) so that $\gr(\structuresheaf(-t)[-2t])\simeq 0^\ast(q^\ast(\integerslocalp(-t)[-2t])= \mathrm{Id}^\ast (\integerslocalp(-t)[-2t])= \integerslocalp(-t)[-2t]$ using the notations in \cref{elementu}.

\noindent By definition of the Day tensor product on graded objects and since $\integerslocalp(-1)$ is pure of weight -1, we find 

$$
\gr^i(\gr(E)^{\mathrm{h}\Filcirclegr}(-t][-2t])\simeq  \gr^{i+t}(\gr(E)^{\mathrm{h}\Filcirclegr})[-2t]
$$ 

\noindent and the cofiber sequence \eqformula{cofibersequencebigradedskeletal} becomes

\begin{equation}
\label{cofibersequencebigradedskeletal2}
\xymatrix{
\ar[d]\gr^{i+t+1}(\gr(E)^{\mathrm{h}\Filcirclegr})[-2t-2]\ar[r]&\gr^{i+t}(\gr(E)^{\mathrm{h}\Filcirclegr})[-2t]\ar[d]\\
0\ar[r]& \gr^i(\gr_{\mathrm{sk}}^t(\widetilde{\pi_\ast}(\rmE)))
}
\end{equation}

\end{remark}

\medskip

\begin{proposition}
\mylabel{gradedpiecesbifilteredHCminus}
\noindent Let $A\in \SCRings{k}$. Then the graded pieces of the bifiltered object $\HCminBiFil(A)$ ( \cref{notationbifilteredHCminus}) are given by

\begin{equation}
\label{formulaforbigradedpieces}
\gr^i(\gr_{\mathrm{sk}}^t(\HCminBiFil(A)))\simeq \bigwedge^{i+t}\, \cotangent_A[i-t]
\end{equation}

\begin{proof}
Using the \cref{thmhkr}-(b), the cofiber sequence \eqformula{cofibersequencebigradedskeletal2} becomes

\begin{equation}
\label{cofibersequencebigradedskeletal3}
\xymatrix{
\ar[d]\gr^{i+t+1}(\DR(A)^{\mathrm{h}\Filcirclegr})[-2t-2]\ar[r]&\gr^{i+t}(\DR(A)^{\mathrm{h}\Filcirclegr})[-2t]\ar[d]\\
0\ar[r]& \gr^i(\gr_{\mathrm{sk}}^t(\HCminBiFil(A)))
}
\end{equation}

and using \cref{thmhkr}-(c), it becomes

\begin{equation}
\label{cofibersequencebigradedskeletal3}
\xymatrix{
\ar[d]\mathbb{L}\widehat{\DR}^{\,\,\geq i+t+1}(A/k)[-2t-2]\ar[r]&\mathbb{L}\widehat{\DR}^{\,\,\geq i+t}(A/k)[-2t]\ar[d]\\
0\ar[r]& \gr^i(\gr_{\mathrm{sk}}^t(\HCminBiFil(A)))
}
\end{equation}

\noindent Finally, using the explicit formula for the derived completed truncated de Rham complex of \cref{derivedderrhamcomplex}, we obtain 
$$
\gr^i(\gr_{\mathrm{sk}}^t(\HCminBiFil(A)))\simeq \bigwedge^{i+t}\, \cotangent_A[ 2(i+t)-(i+t)][-2t]\simeq \bigwedge^{i+t}\, \cotangent_A[i-t]
$$

\end{proof}
\end{proposition}

\medskip

\begin{proposition}
\mylabel{HCminBiFilpreservessifted}
\noindent The functor $\HCminBiFil:\SCRings{k}\to \Filcomplete(\Filcomplete(\Mod_k))$ (see \cref{skeletalfiltrationcomplete}) preserves sifted colimits. In particular it is the left Kan extension of its restriction to polynomial algebras.
\begin{proof}
This is now a consequence of the fact that a map of bi-complete, bi-filtered complexes is an equivalence if and only the associated bi-graded map is an equivalence - this follows from applying the discussion in the \cref{remark-completemodulesleftorthogonal} to filtered objects in filtered objects. We then proceed as in the \cref{failureHCminusfilteredcommutesifted} but this time it is enough to check that the colimit map \eqformula{covidattacksagain3} on the bi-graded pieces is compatible with sifted colimits. But this exactly where the formula in the \cref{gradedpiecesbifilteredHCminus} comes into play: the bi-graded pieces $ \bigwedge^{i+t}\, \cotangent_A[i-t]$ commute with sifted colimits. This can be deduced for instance as a combination of the \cref{prop-HHfilterediscomplete}
and the fact the functor $\gr$ commutes with all colimits and the graded pieces of $\HHFil(A)$ are precisely what we need by the \cref{thmhkr}-(b).
\end{proof}
\end{proposition}

\medskip We are now in position to establish the comparison result for all simplicial commmutative algebras.

\medskip

\begin{construction}
\mylabel{finalcomparisondiagram}
 By incorporating the skeletal filtration of \cref{skeletalfiltration} in the diagrams \eqformula{comparisonBensmoothdiagram1}, \eqformula{comparisonBensmoothdiagramleft1},  \eqformula{comparisonBensmoothdiagramleft2} and \eqformula{bifilteredcompatibleunderlyingspectrum}, we obtain a commutative diagram

 \begin{equation}
     \label{comparisonBenalldiagram1}
     \xymatrix{
     &&\Nerve(\Poly{k})\ar[d]^{\HHFil}\ar[drr]^{\HH}&&\\
     &&\Qcoh(\B\Filcircle)^{\dagger, <\infty}_{\geq 0}\ar@{=}[dll]\ar[rr]^{(-)^{\mathrm{u}}}\ar[d]^{\tau_{\geq }^{\Fil}}&& \Qcoh(\B\circle)^{<\infty}\ar[d]^{\tau_{\geq}^{\mathrm{std}}}\\
    \ar[d]^{\widetilde{\pi_\ast}}\Qcoh(\B\Filcircle)^{\dagger, <\infty}_{\geq 0} &&\underset{\,\mathrm{Wht}}{\Filcomplete}^{\mathrm{const}\geq 0}[\,\,\Qcoh(\B\Filcircle)^{\dagger, <\infty}_{\geq 0}\,\,]\ar[ll]^-{\colim_{\mathrm{Wth}}}\ar[d]^{(\widetilde{\pi_\ast})_{lvl}}\ar[rr]^-{(-)^{\mathrm{u}}_{lvl}}&& \underset{\mathrm{Wht}}{\Filcomplete}^{\mathrm{const}\geq 0}[\,\,\Qcoh(\B \circle)\,\,]\ar[d]^{(\widetilde{(-)^{\mathrm{h}\circle}})_{lvl}}\\
   \underset{\mathrm{sk}}{\Filcomplete}[\,\underset{\Filcircle}{\Filcomplete}(\Mod_k)\,]   &&\ar[ll]\underset{\mathrm{Wht}}{\Filcomplete}^{\mathrm{const}\geq 0}[\,\,\underset{\mathrm{sk}}{\Filcomplete}[\,\underset{\Filcircle}{\Filcomplete}(\Mod_k)\,]\,\,]\ar[rr]^{\colim_{\Filcircle}}\ar[ll]^{\colim_{\mathrm{Wth}}}\ar[d]_{\sim}^{\mathrm{twist}}&& \underset{\mathrm{Wht}}{\Filcomplete}^{\mathrm{const}\geq 0}[\,\,\underset{\mathrm{sk}}{\Filcomplete}[\Mod_k]\,\,]\ar[d]_{\sim}^{\mathrm{twist}_{\mathrm{sk},\mathrm{Wth}}}\\
   && \underset{\mathrm{sk}}{\Filcomplete}[\,\,\underset{\Filcircle}{\Filcomplete}[\,\underset{\mathrm{Wht}}{\Filcomplete}^{\mathrm{const}\geq 0}(\Mod_k)\,]\,\,]\ar[rr]^{\colim_{\Filcircle}}\ar[ull]^{\colim_{\mathrm{Wth}}}&& \underset{\mathrm{sk}}{\Filcomplete}[\,\,\underset{\mathrm{Wht}}{\Filcomplete}^{\mathrm{const}\geq 0}[\Mod_k]\,\,]
     }
\end{equation}

\noindent where the subscript indicates the origin of the filtration: Whitehead tower for the standard $t$-structures, $\mathrm{Wth}$, skeletal $\mathrm{sk}$ and our filtered circle $\Filcircle$. Finally, the comparison result of \cref{comparisonBenresult1} (and our \cref{notationbifilteredHCminus}) provides commutativity for

\begin{equation}
    \label{comparisonBenalldiagram2}
    \xymatrix{
    &&\Nerve(\Poly{k})\ar[drr]\ar[d]\ar[dll]_{\HCminBiFil}&&\\
    \underset{\mathrm{sk}}{\Filcomplete}[\,\underset{\Filcircle}{\Filcomplete}(\Mod_k)\,]   &&\ar[ll]^-{\colim_{\mathrm{Wth}}} \underset{\mathrm{sk}}{\Filcomplete}[\,\,\underset{\Filcircle}{\Filcomplete}[\,\underset{\mathrm{Wht}}{\Filcomplete}^{\mathrm{const}\geq 0}(\Mod_k)\,]\,\,]&& \underset{\mathrm{sk}}{\Filcomplete}[\,\,\underset{\mathrm{Wht}}{\Filcomplete}^{\mathrm{const}\geq 0}[\Mod_k]\,\,]\ar[ll]^-{\underset{\mathrm{sk}}{\Filcomplete}[\,\,\tau_{\geq\, 2\ast}^{B}\,]}
    }
\end{equation}
\end{construction}

\begin{notation}
\mylabel{samenotationsasBen}
\noindent As in \cite[\S 4]{1808.05246}, we denote the right-diagonal arrow in the diagram \eqformula{comparisonBenalldiagram2}, $\Nerve(\Poly)\to \underset{\mathrm{sk}}{\Filcomplete}[\,\,\underset{\mathrm{Wht}}{\Filcomplete}[\Mod_k]\,\,]$, by $\rmF_{\mathrm{HKR}}\rmF_{\mathrm{CW}}\HCmin$ and we denote by $\rmF^\bullet_B\HCmin$ the composition

$$
\rmF^\bullet_B\HCmin:= (\colim_{\mathrm{Wth}})\, \circ \, (\underset{\mathrm{sk}}{\Filcomplete}[\,\,\tau_{\geq\, 2\ast}^{B}\,])\, \circ \, \rmF_{\mathrm{HKR}}\rmF_{\mathrm{CW}}\HCmin
$$

\noindent In \cite[\S 4]{1808.05246} the author defines the value of $\rmF^\bullet_B\HCmin$ for objects in $\SCRings{k}$ by Kan extension from polynomial algebras, under sifted colimits. Let us denote this extension by $\mathrm{LKE}(\rmF^\bullet_B\HCmin)$.

\end{notation}

\medskip

\begin{corollary}
\mylabel{finalcomparisonBen}
The two $\infty$-functors

$$\rmF^\bullet_B\HCmin, \HCminBiFil: \Nerve(\Poly{k})\to \Filcomplete(\Filcomplete(\Mod_k))
$$

\noindent are naturally equivalent under the commutativity of \eqformula{comparisonBenalldiagram2}. Moreover, since $\HCminBiFil$ preserves sifted colimits (\cref{HCminBiFilpreservessifted}), both $\infty$-functors

$$\mathrm{LKE}(\rmF^\bullet_B\HCmin), \HCminBiFil: \Nerve(\Poly{k})\to \Filcomplete(\Filcomplete(\Mod_k))
$$

\noindent are naturally equivalent via the equivalence provided by restriction along $\Nerve(\Poly{k})\subseteq \SCRings{k}$:

$$
\xymatrix{
\Fun^{\mathrm{sifted}}(\SCRings{k}, \Filcomplete(\Filcomplete(\Mod_k)))\ar[r]^{\sim}& \Fun(\Nerve(\Poly{k}), \Filcomplete(\Filcomplete(\Mod_k)))
}
$$

\end{corollary}

\medskip

\section{Applications and complements}\mylabel{lastsection}

\subsection{Towards shifted symplectic structures 
in non-zero characteristics}
\mylabel{section-shiftedsymplectic}

Let $k$ be a commutative $\integerslocalp$-algebra. We 
assume that $p\neq 2$.\\

For a commutative simplicial $k$-algebra $A$, \cref{thmhkr} provides a filtered version  of negative cyclic homology complex $\HCminFil(A/k)$ and tells us that the
graded pieces are canonically given by 
$$\gr^i\HCminFil(A/k) \simeq \mathbb{L}\widehat{\DR}^{\,\,\geq i}(A/k).$$
The \ifunctors $A \mapsto \HCminFil(A/k)$ and $A \mapsto \mathbb{L}\widehat{\DR}^{\,\,\geq i}(A/k)$
are extended by descent to all derived stacks $\calY$:

$$\HCminFil(\calY/k)=\lim_{\Spec\, A \rightarrow \calY}\HCminFil(A/k)$$

\noindent and this comes equiped with 
a canonical filtration whose graded pieces are $\mathbb{L}\widehat{\DR}^{\,\,\geq i}(\calY/k)$, also defined by left Kan extension.
The natural generalization of the notion of shifted symplectic structures of \cite{MR3090262} is the following definition.

\begin{definition}
\mylabel{definition-closedforms}
\hfill \\
\begin{enumerate}
    \item 

For a derived stack $\calY$ over $k$, we define the \emph{ complex
of closed $q$-forms on $\calY$} to be 
$$\mathcal{A}^{cl,q}(\calY/k):=\gr^q\HCminFil(\calY/k)[-q]
\simeq \mathbb{L}\widehat{\DR}^{\geq q}(\calY/k)[-q].$$
\item A closed $2$-form $\omega$ of degree 
$n$ on a derived stack $\calY$ is \emph{non-degenerate}
if the underlying element in $\mathsf{H}^n(\calY,\wedge^2 \cotangent_{\calY/k})$ is non-degenerate in the
sense of \cite{MR3090262}.
\end{enumerate}
\end{definition}

\medskip

The above definition is a rather naive notion, as we believe that there may exist more subtle versions. For instance, 
it is very natural to ask for a shifted symplectic structure to lift to an element in the second level of the filtration $F^2\HCminFil(\calY/k)$.  We thus
define the notion of \emph{enhanced shifted symplectic structures} as follows.

\begin{definition}
\mylabel{definition-enhanced}
For a derived stack $\calY$ over $k$, an \emph{enhanced shifted symplectic
structure of degree $n$ on $\calY$} consists of the data of an element
$$\omega \in H^n(F^2\HCminFil(\calY/k)[-2])$$
such that its image by the canonical map
$$F^2\HCminFil(\calY/k) \longrightarrow \gr^2\HCminFil(\calY/k)$$
defines a non-degenerate closed $2$-form of degree 
$n$ on $\calY$.
\end{definition}

In characteristic zero, the HKR theorem implies that there is always a canonical
lift, but outside of this case lifts might not even exist (and if it does, it may not 
be canonical).
The data of such a lifting seems to be of some importance to us, 
in particular for questions concerning quantization. 
As a first example of existence of this type of lift we
show below that most shifted symplectic structures constructed in nature do possesses such lifts, by means of the Chern character construction. \\

\begin{reminder}[Chern Character]
\mylabel{cherncharacter}
Recall the existence of a Chern character $\chern : \Ktheoryconnective(A) \longrightarrow 
\HC^-(A/k)$, 
which is here considered as a map of spectra
(see for instance \cite{MR3338682}). 
Note also that $\Ktheoryconnective$ stands here for the space of connective
$K$-theory of $A$. This map can be enhanced into a morphism of stacks of spaces
on the site of derived affine schemes over $k$
$$\chern : \Ktheoryconnectivesheaf \longrightarrow \HC^-$$
($\Ktheoryconnectivesheaf$ is the stack associated to $A \mapsto \Ktheoryconnective(A)$, note that 
$A \mapsto \HC^-(A/k)$ is itself already a stack because $\HH$ itselt is a stack in the étale topology \cite{MR1120653}). 

When $\calY$ is a smooth Artin stack over $k$, theorem \ref{thmhkr} $(d)$ implies that 
the canonical map $\HCminFil(\calY/k) \longrightarrow \HC^-(\calY/k)$
is an equivalence, and thus the Chern character produces a well defined map
$$Ch : \Kspace_0(\calY) \longrightarrow \mathrm{H}^0(\HCminFil(\calY/k)).$$
This applies, in particular, when $\calY$ is of the form $\B G$ for $G$ a
smooth group scheme over $k$.

\end{reminder}

\medskip

We now consider the stack $\B\Sln$, and contemplate the following splitting 
statement.

\begin{lemma}
We have a canonical splitting 
$$\mathrm{H}^0(\HCminFil(\B\Sln)) \simeq \mathrm{H}^0(F^2\HCminFil(\B\Sln)) \oplus \mathrm{H}^0(\gr^0\HCminFil(\B\Sln)).$$
Moreover, the unit in the graded ring $\mathrm{H}^0(\gr^*\HCminFil(\B\Sln))$
induces an isomorphism
$$k\simeq \mathrm{H}^0(\gr^0\HCminFil(\B\Sln)).$$
\end{lemma}

\begin{proof}
First of all, by \ref{thmhkr} $(c)$,
$\gr^0\HCminFil(\B\Sln)$ computes the algebraic de Rham
complex of the stack $\B\Sln$. Therefore, the unit map 
$$k \longrightarrow \mathrm{H}^0(\gr^0\HCminFil(\B\Sln))$$ is indeed an isomorphism. Moreover, we see by the same
argument that all the complexes $\gr^p\HCminFil(\B\Sln)$ are
cohmologically concentrated in non-negative degrees. As a result, we have
a short exact sequence of $k$-modules
$$\xymatrix{
0 \ar[r] & \mathrm{H}^0(F^1\HCminFil(\B\Sln)) \ar[r] & 
\mathrm{H}^0(\HCminFil(\B\Sln)) \ar[r] & k.}$$
This sequence splits on the right because of the unit map
$k \rightarrow \mathrm{H}^0(\HCminFil(\B\Sln))$.

In order to prove the lemma it thus remains to show that the natural 
map
$$\mathrm{H}^0(F^2\HCminFil(\B\Sln)) \longrightarrow \mathrm{H}^0(F^1\HCminFil(\B\Sln))$$
is bijective. 

Because of the exact sequence
$$\xymatrix{
0 \ar[r] & \mathrm{H}^0(F^2\HCminFil(\B\Sln)) \ar[r] & \mathrm{H}^0(F^1\HCminFil(\B\Sln))
\ar[r] & \mathrm{H}^0(\gr^1\HCminFil(\B\Sln)),
}$$
it only remains to show that $\mathrm{H}^0(\gr^1\HCminFil(\B\Sln))=0$. By the theorem
\ref{thmhkr} $(c)$ we now that $\gr^1\HCminFil(\B\Sln)\simeq 
\mathbb{L}\widehat{\DR}^{\,\,\geq 1}(\B\Sln/k)$. We claim that 
$\mathrm{H}^0(\mathbb{L}\widehat{\DR}^{\,\,\geq 1}(\B\Sln/k))\simeq (\sln^*)^{\Sln}$
is the module of $\SLn$-invariants in the dual of the Lie algebra $\sln$, and 
thus
is indeed zero. To see this we use the Hodge filtration on the
derived de Rham complex, which here consists of an exact triangle
$$\xymatrix{\mathbb{L}\widehat{\DR}^{\,\,\geq 2}(\B\Sln/k)[-2]
\ar[r] & \mathbb{L}\widehat{\DR}^{\,\,\geq 1}(\B\Sln/k) \ar[r] & 
\Gamma(\B\Sln,\mathbb{L}_{\B\Sln})[1].}$$
The cotangent complex of $\B\Sln$ is given by $\sln^*[-1]$, with
its natural action of $\Sln$. Therefore, 
$\Gamma(\B\Sln,\mathbb{L}_{\B\Sln})[1]\simeq (\sln^*)^{h\Sln}$
is the homotopy fixed point of coadjoint representation. We have $\mathrm{H}^0((\sln^*)^{h\Sln})\simeq (\sln^*)^{\Sln}\simeq 0$.
\end{proof}

The lemma implies that the Chern character of the 
tautological rank $n$ bundle on $\B\Sln$ defines an element
$$Ch_{\geq 2} \in \mathrm{H}^0(F^2\HCminFil(\B\Sln)).$$
The image of this element in $\mathrm{H}^0(\gr^2\HCminFil(\B\Sln))\simeq 
(Sym^2(\sln)^*)^{\Sln}$ is the quadratic form
$$\sln \otimes \sln \longrightarrow k$$
sending $(A,B)$ to $Tr(AB)$ and is thus non-degenerate. In other words
$Ch_{\geq 2}$ defines an enhanced shifted symplectic structure 
of degree $2$ on $\B\SLn$.

As a result, any vector bundle on a derived stack $X$, equipped
wit a trivialization of its determinant line bundle, defines
an element
$$\omega \in \mathrm{H}^0(F^2\HCminFil(X/k))$$
by pull-back of $Ch_{\geq 2}$ by the classifying map $X \longrightarrow
\B\Sln$. This construction can be used in order to construct
enhancements of the shifted symplectic structures on moduli 
of $\Sln$-bundles on Calabi-Yau varieties constructed in 
\cite{MR3090262}.

\medskip

\subsection{Filtration on Hochschild cohomology}
\mylabel{filtrationonHochschildcohomology}

As a second example of possible applications of the filtered circle, we explain here how it can also provide interesting filtrations on Hochschild cohomology. For this, we 
will have to consider $\Filcircle$ not 
as a group anymore but as a cogroup object
inside filtered stacks. It is even more, 
as it carries a $\Etwo$-cogroupoid structure
over $\Filstack$
which can be exploited to get a filtration 
on Hochschild cohomology compatible with 
its natural $E_2$-structure. Moreover, all the constructions in this part make sense
over the sphere spectrum, and so provide filtrations on topological Hochschild cohomology as well. \\

As a start we consider the natural closed embedding of stacks
$$0 : \B\Gm{} \hookrightarrow
\Filstack.$$
The direct image of the structure sheaf 
defines a commutative cosimplicial algebra
over $\Filstack$. Let us
denote it by $\calE$, and let 
$\structuresheaf \longrightarrow \calE$ be the unit
map. The nerve of this map
produces a groupoid object inside
commutative cosimplicial algebras
over $\Filstack$, which
is denoted by $\calE^{(1)}$. In the same
manner, we can consider the nerve of $\structuresheaf \longrightarrow \calE^{(1)}$ to get
an $\Etwo$-groupoid object (that is a groupoid
object inside groupoid objects) 
$\calE^{(2)}$ and so on and so forth.

We define this way an $\En$-groupoid object
$\calE^{(n)}$ inside the $\infty$-category of
commutative cosimplicial algebras over 
$\Filstack$.

\begin{definition}
\mylabel{filteredsphere}
The \emph{filtered $n$-sphere} is defined to be $\cospec\, \calE^{(n+1)}$. It is an 
$\Enone$-cogroupoid object in affine stacks
over $\Filstack$. It is
denoted by $\Filsphere$.
\end{definition}

Note that $\Filsphere$ possesses an 
underlying object of $(1,...,1)$-morphisms.
Explicitly, this is given by 
$\Spec\, \mathsf{H}^*(\nsphere,\integerslocalp)$, and is called the
formal or graded sphere. When $n=1$
we recover our filtered circle $\Filcircle$
as a filtered affine stack, but now 
it comes equiped with an $\Etwo$-cogroupoid
structure rather than a group structure.

We now consider $\mathcal{L}_{Fil}^{(n)}X=
\Map(\Filsphere,X)$, for
a derived affine scheme $X=\Spec\, A$.
The cogroupoid structure on $\Filsphere$ 
endows $\mathcal{L}_{Fil}^{(n)}X$ with an 
$\Enone$-groupoid structure acting on $X$.
Passing to functions and taking linear duals
we get a filtered $\Enone$-algebra
over $\bbZ_p$ whose underlying 
object is $\HH_{\En}^{*}(A)$, the $n$-the iterated Hochschild cohomology of $A$, and the
associated graded is $\Sym_A(\cotangent_{A}[n])^{\vee}$, the dual 
of shifted differential forms, which 
can be defined as shifted polyvector fields over $X$. In summary, we expect the following proposition:

\begin{proposition}
\mylabel{iteratedHH}
Let $k$ be a commutative $\bbZ_p$-algebra, 
and $A$ a commutative simplicial $k$-algebra. The iterated 
Hochschild cohomology $\HH^*_{\Enone}(A/k)$,
carries a canonical filtration compatible with its $\Enone$-multiplicative structure,
whose associated graded is 
the $\Einfinity$-algebra of
$n$-shifted polyvectors on $X$.
\end{proposition}

The last proposition can possibly be used 
in order to define
singular supports of coherent sheaves, or of sheaves of linear categories, over any base scheme. For instance, 
in the context of bounded coherent sheaves, 
this could allows one to extend the notion and construction of 
\cite{MR3300415}.

\subsection{Generalized cyclic homology and formal groups}
\mylabel{section-generalizedcyclicformalgroups}
The filtered circle $\Filcircle$ we have constructed in this paper is
part of a much more general framework that associates
a circle $\circle_E$ to any reasonable abelian formal group $E$. To be more precise:

\begin{construction}
\mylabel{filteredEcircle}
We can start by an abelian formal 
group $E$ over some base commutative ring $k$, and assume that 
$E$ is formally smooth and of relative dimension $1$ over $k$. 
The Cartier dual $G_E$ of $E$ is a flat abelian group scheme over $\Spec\, k$, obtained as $\Spec\, \mathcal{O}(E)^\vee$, 
where $\mathcal{O}(E)^{\vee}$ is the commutative and cocommutative
Hopf algebra of distributions on $E$. Because $E$ is smooth and
of relative dimension $1$, $\mathcal{O}(E)^\vee$ is a flat 
commutative $k$-coalgebra which is locally for the Zariski topology on $k$ 
isomorphic to $k[X]$ with the standard comultiplication 
$\Delta(X^n)=\sum_{i+j=n}\binom{n}{i}X^i\otimes X^j$. The \emph{$E$-circle} is defined as the group stack over 
$k$ defined by 
$$\circle_E:=\B G_E.$$
\end{construction}

Under reasonable assumptions on $E$ 
the stack $\circle_E$ is an affine stack over $k$. Moreover, its \icategory of representations, $\Qcoh(\B \circle_E)$ is naturally equivalent
to the \icategory of mixed complexes over $k$, at least locally 
on $\Spec\, k$. To be more precise, if we denote by $\omega_E$ the
line bundle of relative $1$-forms on $G_E$, the \icategory 
$\Qcoh(\B \circle_E)$ is equivalent to $\omega_E$-twisted mixed complexes, namely
comodules over the $k$-coalgebra $k\oplus \omega_E[-1]$.
However, the symmetric
monoidal structure on $\Qcoh(\B \circle_E)$ corresponds to a non-standard 
monoidal structure on mixed complexes that depends on the formal group
structure on $E$. 

The filtration on $\Filcircleund$ whose associated
graded is $\Filcirclegr$ seems to also exists in some interesting examples
of formal group laws. 



We recover the results in this paper when $E$ is either the additive or the multiplivative formal group $\formalGa{}$, resp. $\formalGm{}$:

\begin{construction}
\mylabel{deformationGmtoGa}
Let $k$ be a commutative ring. There exists a filtered group deforming $\Gm{k}$ to $\Ga{k}$. Namely, given $\lambda\in k$ take $\Gm{k}^\lambda=:\Spec (k[T, \frac{1}{1+\lambda T}])$. This is a group scheme under the multiplicative rule $T\mapsto 1\otimes T+ T\otimes 1+ \lambda.T\otimes T$ and unit $T\mapsto 0$. When $\lambda=0$ we get $\Ga{k}$ and for $\lambda=1$ we get $\Gm{k}$. Taking formal completions this deforms $\formalGm{k}$ to $\formalGa{k}$. 
\end{construction}

\personal{\margin{new comment}Explain this in terms of the deformation to the normal cone from $\formalGm{}$ to the exponential of its lie algebra $\formalGa{}$. Also say that in char.0,  that the exponential makes a map of groups but does not produce a map of cogroups and this is the origin of HKR (reference)}

\begin{proposition}
\mylabel{recoverviaCartierduality}
Let $k=\integerslocalp$. Then 
$$\circle_{\formalGa{}}:= \B G_{\formalGa{}}\simeq  \Filcirclegr\,\,\,\, \text{ and }\,\,\,\, \circle_{\formalGm{}}:= \B G_{\formalGm{}}\simeq \Filcircleund$$
Moreover, the filtration on $\Frobfixed$ is Cartier dual to the filtration on $\formalGm{k}$ of \cref{deformationGmtoGa}.

\begin{proof}
The proposition is equivalent to the claims that:
\begin{enumerate}
    \item $\Frobfixed$ is Cartier dual to $\formalGm{}$;
    \item $\Frobkernel$ is Cartier dual to $\formalGa{}$;
    \item the filtrations are Cartier dual
\end{enumerate}
This is precisely the content of \cite[Theorem]{MR1829979}. See \cite[37.3.4]{MR2987372} for Cartier duality.
\end{proof}
\end{proposition}

Let $E$ be an abelian formal group over $k$ as before and $\circle_E$
the corresponding $E$-circle. For any 
derived affine $k$-scheme $X$ we define
the $E$-loop space $\mathcal{L}_E X:=\Map(\circle_E,X)$, that comes equiped with 
an $\circle_E$-action. The $E$-Hochschild homology of $X$ over $k$
is by definition the complex of functions $\mathcal{O}(\mathcal{L}_EX)$. 
It is denoted by $\HH(X,E)$. The $\circle_E$-action on $\HH^E(X)$ induces
a mixed structure on $\HH(X,E)$ whose total complex computes the
$\circle_E$-equivariant cohomology and is called by definition 
the negative cyclic $E$-homology $\HC^-(X,E)$. When a filtration 
exists on $E$, then there is an HKR-type filtration on $\HC^-(X,E)$ whose
associated graded is again derived de Rham cohomology.\\

Of course, the results of this work are recovered when $E$ is taken to be
the multiplicative formal group law and we recover an isomorphism of filtered group schemes $\Hgroup=G_{E_T}$. See \cite{MR1829979}.\\

An example of particular interest is when $E$ comes, by completion, 
from an elliptic curve. The corresponding Hochscild and cyclic homology
can be called \emph{elliptic Hochschild and cyclic homology}.\\



 \subsection{Topological and $q$-analogues}\mylabel{topoqanalogue}

The filtered circle $\Filcircle$ constructed in this work
possesses at least two extensions, both of quantum /non-commutative nature: one
as a non-commutative group stack over the sphere spectrum and
a second extension as filtered group stack over
$\bbZ[q,q^{-1}]$.  \\

\subsubsection{\textbf{q-analogue}}
\mylabel{qanaloguefilteredcircle}
 As a start, we relax the restriction of working $p$ locally and attempt to make sense of these constructions over $\bbZ$. A first possibility is simply to use big Witt vectors and 
define the filtered  group scheme $H$ as the intersection of all kernels of the endomorphisms $\calG_p$
for all primes $p$. There is however a second 
possible description, which has the merit of showing 
the natural q-deformed version, which we now describe. 

We start by the filtered formal group $\mathbb{G}$, 
interpolating between the formal multiplicative and
the formal additive group over $\bbZ$. The corresponding
formal group over $\mathbb{A}^1$ is given by 
$X+Y+\lambda XY$ where $\lambda$ is the coordinate on 
the affine line. The underlying formal group 
is $\widehat{\mathbb{G}_m}$ whereas the associated graded 
is $\widehat{\mathbb{G}_a}$ together with its natural graduation given the natural action of $\mathbb{G}_m$. 
The algebra of distributions of the filtered formal 
scheme $\mathbb{G}$ defines a filtered commutative and
cocommutative Hopf algebra $\mathcal{R}$. This 
algebra can be described explicitely as being the
algebra of integer valued polynomials, that is the 
subring of $\bbQ[X]$ formed by all polynomials $P$ 
such that $P(\mathbb{Z})\subset \mathbb{Z}$. The filtration is then induced the degree of polynomials. 

The associated graded to this filtration is the ring 
$Gr\mathcal{R}$ of divided powers over $\mathbb{Z}$. This
is the subring of $\bbQ[X]$ generated by the $\frac{X^n}{n!}$. \\

An integral version of the filtered group scheme $\Hgroup$, 
and of the filtered circle $\Filcircle$, can then be
defined as $\mathsf{H}_{\bbZ}:=\Spec(Rees(\mathcal{R}))$, 
where $Rees(\mathcal{R})$ is the Rees construction associated to the filtered Hopf algebra $\mathcal{R}$. 
The integral version of the filtered circle is then 
defined to be
$$\FilcircleglobalZ:=\B \mathsf{H}_\bbZ.$$
It is a pleasant exercise to show that when restricted
over $\Spec\, \bbZ_p$ this recovers our filtered circle
$\Filcircle$. We believe that all the statement proved in this work can be extended over $\bbZ$, but some of the strategies of proof we use do not obviously extend to 
the situation where we deal with an infinite number of primes. \\

One advantage of the above presentation using 
integer valued polynomial algebras is the striking fact that these admit natural q-deformed versions. 
The q-deformed version  $\mathcal{R}_q$ of the ring $\mathcal{R}$
is introduced and studied in \cite{MR3604068}, and is essentially 
the Cartan part $\textbf{U}^0(\sltwo)$ of the divided power
quantum group of Lusztig (see \cite{MR3604068} end of section 4). 
In particular, 
we think that the filtered Hopf algebra
$\mathcal{R}$ possesses a q-deformed version 
$\mathcal{R}_q$, which is a commutative and cocommutative
filtered Hopf algebra over $\bbZ[q,q^{-1}]$, recovering
$\mathcal{R}$ when $q=1$. The spectrum of this
provides a q-deformed version of $\mathsf{H}_\bbZ$ that we denote
by $\mathsf{H}_{\bbZ,q}$. Its classifying stack is by definition
the q-deformed filtered circle.

\begin{definition}
\mylabel{qdeformedcircledefinition}
The $q$\emph{-deformed filtered circle} is the 
filtered stack $\FilcircleglobalZ(q):=\B \mathsf{H}_{\bbZ,q}$. It is a stack over
$\Filstack\times \Spec\, \bbZ[q,q^{-1}]$.
\end{definition}

As in \cref{thmhkr}, by considering the derived mapping 
stack $\Map(\FilcircleglobalZ(q),X)$, it is then possible to define q-analogues of Hochschild
and cyclic homology of a scheme $X$, 
together with a filtration 
whose associated graded may be used to define 
a notion
of a q-deformed derived de Rham cohomology; this should be
compared to various notions of q-deformed
de Rham complexes appearing the literature (see
for instance \cite{MR1792063} and 
\cite{MR1225788}).

However, 
to make the above definition precise requires some extra work. 
For instance, it seems to us that the associated graded
of $\FilcircleglobalZ(q)$ can not truly exist  as
a naive commutative object and requires to work over
some braided monoidal base category associated to 
$\bbZ[q,q^{-1}]$, as this is done for instance in the theory 
of Ringel-Hall algebras, see for instance \cite{MR1778168}. In fact, we
expect the associated graded of $\FilcircleglobalZ(q)$ to be 
of the form $\B K(q)$, where $K(q)$ is the spectrum 
of the Ringel-Hall algebra over the one point Quiver. 

\subsubsection{\textbf{Topological Analogue}}
\mylabel{topologicalanaloguefilteredcircle}
 Let us mention yet another extension of the filtered circle, now over the sphere spectrum. We do not believe that the filtered stack $\Filcircle$ can exist 
as a spectral stack in any sense, as
the associated graded $\Filcirclegr$ probably cant exist
over the sphere spectrum. However, it is possible to construct a non-commutative version of this object, using the $2$-periodic sphere spectrum of \cite{lurie-K}. As shown
in \cite{lurie-K} there exists a filtered $\Etwo$-algebra 
whose underlying object is $\Sphere^{K(\bbZ,2)}$
(so is $\Einfinity$) but its associated graded 
is a $2$-periodic version of the sphere spectrum
$\Sphere[\beta,\beta^{-1}]$. This $2$-periodic 
sphere spectrum is known not to exist as an $\Einfinity$-ring. However, we can consider the natural augmentation 
$$\Sphere^{K(\bbZ,2)} \longrightarrow \Sphere$$
and consider the spectrum
$$A:=\Sphere\otimes_{\Sphere^{K(\bbZ,2)}}\Sphere
$$
As a mere spectrum, this is equivalent to the 
group ring over the circle $A \simeq \Sphere[K(\bbZ,1)]$. However, the $E_2$-filtration 
on $\Sphere^{K(\bbZ,2)}$ induces a structure of a filtered bialgebra on $A$, which should be considered
as a non-commutative analogue of the filtered circle. 

More precisely, the idea is to study 
the dual filtered bialgebra $B=A^*$ and consider
$\Spec B$ in some sense to produce a 
topological version of the filtered circle. We however
do not know currently know how to exploit the existence of $B$ and invite the interested reader to pursue this approach further.

\BIBLIO

\end{document}